\documentclass[12pt,english]{article}
\usepackage[margin=1.2in]{geometry}
\usepackage[T1]{fontenc}
\usepackage[latin9]{inputenc}

\usepackage{amsmath}
\usepackage{amsthm}
\usepackage{amssymb}

\usepackage{graphicx}
\usepackage{float} 
\usepackage{xcolor}
\usepackage{shuffle}
\usepackage{comment}
\usepackage{stmaryrd}
\PassOptionsToPackage{normalem}{ulem}
\usepackage{ulem}
\usepackage{shuffle}

\usepackage{hyperref}
\hypersetup{
     colorlinks   = true,
     citecolor    = blue,
     linkcolor    = blue
}

\makeatletter

\numberwithin{equation}{section}
\theoremstyle{plain}
\newtheorem{thm}{\protect\theoremname}[section]

\newtheorem{rem}[thm]{\protect\remarkname}
\theoremstyle{plain}
\newtheorem{lem}[thm]{\protect\lemmaname}
\theoremstyle{definition}
\newtheorem{defn}[thm]{\protect\definitionname}
\theoremstyle{plain}
\newtheorem{cor}[thm]{\protect\corollaryname}
\theoremstyle{definition}
\newtheorem{example}[thm]{\protect\examplename}
\theoremstyle{plain}
\newtheorem{proposition}[thm]{\protect\propositionname}
\newtheorem*{convention}{Convention}

\usepackage{babel}
\providecommand{\corollaryname}{Corollary}
\providecommand{\definitionname}{Definition}
\providecommand{\examplename}{Example}
\providecommand{\lemmaname}{Lemma}
\providecommand{\remarkname}{Remark}
\providecommand{\theoremname}{Theorem}
\providecommand{\propositionname}{Proposition}

\newtheorem*{assumption}{Assumption}
\newtheorem*{theoremnoname}{Theorem}

\date{}

\makeatother

\begin{document}
\title{Cartan's Path Development, the Logarithmic Signature and a Conjecture of Lyons-Sidorova}
\author{Horatio Boedihardjo\thanks{Department of Statistics, University of Warwick, Coventry CV4 7AL,
United Kingdom. Email:horatio.boedihardjo@warwick.ac.uk.}$,\ \  $Xi Geng\thanks{School of Mathematics and Statistics, University of Melbourne,  Parkville VIC 3010, Australia.
Email: xi.geng@unimelb.edu.au.}$\ \ $and  Sheng Wang\thanks{School of Mathematics and Statistics, University of Melbourne, Parkville VIC 3010, Australia.
Email: shewang4@student.unimelb.edu.au.}}
\maketitle
\begin{abstract}
The signature transform, which is defined in terms of iterated path integrals of all orders, provides a faithful representation of the group of tree-reduced geometric rough paths. While the signature coefficients are known to decay factorially fast, the coefficients of the logarithimic signature generically only possess geometric decay. It was conjectured by T. Lyons and N. Sidorova \cite{LS06} that the only tree-reduced paths with bounded variation (BV) whose logarithmic signature can have infinite radius of convergence are straight lines. This conjecture was confirmed in the same work for certain types of paths and the general BV case remains unsolved.

The aim of the present article is to develop a deeper understanding towards the Lyons-Sidorova conjecture. We prove that, if the logarithmic signature has infinite radius of convergence, the signature coefficients must satisfy an infinite system of rigid algebraic identities defined in terms of iterated integrals along complex exponential one-forms. These iterated integral identities impose strong geometric constraints on the underlying path, and in some special situations, confirm the conjecture. 

As a non-trivial application of our integral identities, we prove a strengthened version of the conjecture, which asserts that if the logarithmic signature of a BV path has infinite radius of convergence over all sub-intervals of time, the underlying path must be a straight line. 

Our methodology relies on Cartan's path development onto the complex semisimple Lie algebras $\mathfrak{sl}_m(\mathbb{C})$. The special root patterns of $\mathfrak{sl}_m(\mathbb{C})$ allow one to project the infinite-dimensional Baker-Campbell-Hausdorff (BCH) formula in a very special finite dimensional manner to yield meaningful quantitative relations between BCH-type singularities and the vanishing of certain iterated path integrals.
  
\tableofcontents{}
\end{abstract}

\section{Introduction }

The \textit{signature transform} (or simply the signature) of a multidimensional
continuous path $\gamma:[0,T]\rightarrow\mathbb{R}^{d}$ with bounded
variation (BV) is the formal tensor series 
\begin{align*}
S(\gamma) & \triangleq\big(1,\gamma_{T}-\gamma_{0},\int_{0<s<t<T}d\gamma_{s}\otimes d\gamma_{t},\cdots,\int_{0<t_{1}<\cdots<t_{n}<T}d\gamma_{t_{1}}\otimes\cdots\otimes d\gamma_{t_{n}},\cdots\big)
\end{align*}
defined by the global iterated path integrals of all orders. This
path transformation was originally introduced by the geometer K.T.
Chen \cite{Che73} in the 1950s to study the topology of loop spaces over
manifolds. The signature transform plays a fundamental role in Lyons'
rough path theory as well as its broad applications to problems in
stochastic analysis and more recently, in machine learning. Its theoretical
importance partly comes from the \textit{signature uniqueness theorem},
which asserts that every rough path is uniquely determined by its
signature up to tree-like equivalence (cf. \cite{Che58, HL10, BGL16}). A probabilistic
counterpart of this theorem was obtained by Chevyrev-Lyons \cite{CL16}.
The monographs \cite{FV10, LCL07} contain excellent exposition on the mathematical
theory of the signature transform and some applications to stochastic
analysis are discussed e.g. in \cite{Bau04, KS87, LV04}. Some recent applications of signature-based
methods in machine learning are contained e.g. in \cite{CGS23, NSW23, SCF21} (and references therein).

While the signature encodes essential geometric information about
the underlying path, it contains a lot of algebraic redundancies since
different components of the signature are related by some universal
algebraic constraints (the shuffle product formula). An effective way of removing
such algebraic dependencies is to pass to the so-called \textit{logarithmic
signature transform}. The two objects are naturally isomorphic, however,
it was a deep theorem of Chen \cite{Che57} that the logarithmic signature
is a free Lie series and thereby does not contain algebraic relations
among its components. This leads to dimension-reduction from a practical
viewpoint. An important motivation for studying the logarithmic signature
is related to the study of controlled differential equations, which
will be elaborated further in Section \ref{subsec:Motivation} below.
In recent years, methods based on the logarithmic signature have also
been developed to study various problems in machine learning (cf. \cite{FLM23, WMQ24} and references therein).

Analytic properties of the signature and logarithmic signature are
very different on the other hand. It is well-known (which is rather
trivial in the BV case) that the signature components decay factorially
fast with respect to the degree. However, it is highly non-trivial
that components of the logarithmic signature generically only possess
geometric decay. This is like the infinite dimensional analogue of
the elementary fact that the exponential function $e^{z}$ is entire
while the logarithmic function $\log(1+z)$ only has a finite radius
of convergence (R.O.C.). Understanding such a property, even just
in the BV case, is still an unsolved open problem in rough path theory. 

\subsection{\label{subsec:Motivation}Motivation: a conjecture
of T. Lyons and N. Sidorova}

The logarithmic signature arises naturally when one attempts to solve
a controlled differential equation (CDE) using the so-called \textit{log-ODE
method}. Consider the following CDE
\begin{equation}
dX_{t}=\sum_{i=1}^{d}V_{i}(X_{t})d\gamma_{t}^{i},\ \ \ 0\leqslant t\leqslant T\label{eq:CDEIntro}
\end{equation}
on a differentiable manifold $M$, where $V_{1},\cdots,V_{d}$ are
smooth vector fields on $M$ and $\gamma_{t}$ is a smooth $\mathbb{R}^{d}$-valued
path. At a formal level, the flow of diffeomorphisms on $M$ induced
by (\ref{eq:CDEIntro}) admits a ``logarithm'' in the Lie algebra
${\cal L}(V_{1},\cdots,V_{d})$ of smooth vector fields generated
by the $V_{i}$'s. Such a logarithm can be explicitly expressed in
terms of the logarithmic signature of $\gamma$. 

To elaborate this, let $W$ be a smooth vector field on $M$. We use
$\exp(W)$ to denote the diffeomorphism given by the time-one map
of the flow induced by $W$. In other words, 
\[
\exp(W):M\rightarrow M,\ \exp(W)x=y_{1}
\]
where $(y_{t})_{0\leqslant t\leqslant1}$ solves the ODE 
\[
\dot{y}_{t}=W(y_{t}),\ y_{0}=x.
\]
The solution at time $T$ to the CDE (\ref{eq:CDEIntro}) can formally
be expressed as 
\begin{equation}
X_{T}=\exp\big(\sum_{I}\Lambda_{I}(\gamma)V_{[I]}\big)(x).\label{eq:CSForm}
\end{equation}
Here the summation is taken over all words $I=(i_{1},\cdots,i_{n})$
where $n\geqslant1$ and $i_{j}\in\{1,\cdots,d\}.$ The numbers $\Lambda_{I}(\gamma)$
are the logarithmic signature coefficients of $\gamma$; one
has
\[
\log S(\gamma)=\sum_{I=(i_{1},\cdots,i_{n})}\Lambda_{I}(\gamma)[e_{i_{1}},[e_{i_{2}},\cdots,[e_{i_{n-1}},e_{i_{n}}]\cdots]].
\]
The vector fields $V_{[I]}$ are defined by $V_{[I]}\triangleq[V_{i_{1}},[V_{i_{2}},\cdots,[V_{i_{n-1}},V_{i_{n}}]]]$.
The relation (\ref{eq:CSForm}) is often known as the \textit{Chen-Strichartz
formula} (cf. \cite{Bau04, Str87}). 

However, the formula (\ref{eq:CSForm}) is only formal and it requires
additional analytic assumptions on $\log S(\gamma)$ and the vector
fields $V_{i}$ to make it precise. It is accurate if $\gamma$ is
a straight line or if the Lie algebra ${\cal L}(V_{1},\cdots,V_{d})$
is nilpotent (in both cases, the series (\ref{eq:CSForm}) becomes
a finite sum). In general, one expects that the series (\ref{eq:CSForm})
should converge if the ``size'' of the vector fields is within
the R.O.C. of $\log S(\gamma)$. This makes the R.O.C. for the logarithmic
signature a natural object of study. However, it is still unclear
how one can make such intuition mathematically precise. In practice,
one usually applies a truncated version of (\ref{eq:CSForm}) over
a partition of $[0,T]$ to obtain an approximating solution (numerical
scheme) and tries to prove convergence when the mesh size is sent
to zero (cf. \cite{Bau04, FL15}). 

\subsubsection*{Cartan's path development}

There is a simple yet useful situation where the analysis can be made
very precise. Let $G$ be a matrix Lie group with Lie algebra $\mathfrak{g}$.
Let $V=(V_{1},\cdots,V_{d})$ be $d$ left invariant vector fields
on $G$. The Lie algebra ${\cal L}(V_{1},\cdots,V_{d})$ is now
a sub-algebra of $\mathfrak{g}$. One can prove that the series (\ref{eq:CSForm})
is convergent if $\|V\|<{\rm R.O.C}.$ of $\log S(\gamma)$, where
certain matrix norms has to be taken and fixed. It is known from \cite{CL16} that the R.O.C. is strictly positive and in
particular, the series (\ref{eq:CSForm}) is always convergent (in
the current Lie group setting) as long as the vector fields $V_{i}$
are ``small enough''. The resulting solution $X_{t}$ to the CDE
(\ref{eq:CDEIntro}) defines a path in the Lie group $G$ which is
known as\textit{ Cartan's path development }of $\gamma$ in $G$. If $\gamma$
is a straight line, say $\gamma_{t}=tw$ ($w\in\mathbb{R}^{d},t\in[0,T]$),
the formula (\ref{eq:CSForm}) is well-defined for \textit{all} left
invariant vector fields. The solution $X_{t}$ is just the one-parameter
subgroup defined by 
\[
X_{t}=\exp(t\sum_{i=1}^{d}w^{i}V_{i}),
\]
where $\exp$ is now the exponential map for $G.$ This aligns with
the fact that $\log S(\gamma)$ has infinite R.O.C.

Initiated by the celebrated work of Lyons-Sidorova \cite{LS06}, it has
been widely believed that\textit{ straight lines are the only (tree-reduced)
BV paths whose logarithmic signatures can have infinite R.O.C.} The
logic in Lyons-Sidorova's argument can be summarised as follows. Suppose
that $\log S(\gamma)$ has infinite R.O.C. Then the end point $X_{T}$
of Cartan's development of $\gamma$ into any Lie group always admits
a logarithm, namely, $X_{T}$ is always an element in the image of
the exponential map. In other words, if one is able to construct a
special Cartan's development such that $X_{T}$ turns out to be outside
${\rm Im}\exp$, this immediately suggests that $\log S(\gamma)$
only has finite R.O.C. Of course, this method requires the use of
(non-compact) Lie groups with non-surjective exponential map (e.g.
the special linear group $\mathrm{SL}_{2}(\mathbb{R})$). Lyons and
Sidorova \cite{LS06} proved the above claim for two special classes
of paths based on such a geometric viewpoint: paths which are strictly
increasing in one particular direction (what they called\textit{ $1$-monotone
paths}) and a generic class of piecewise linear paths (which they
called \textit{non-double paths}). They conjecturered that the claim
should be true in the most general BV setting. To the best of our
knowledge, this problem remains unsolved. 

\subsection{Summary of main results and novelty}

The main goal of the present article is to study the aforementioned
Lyons-Sidorova (LS) conjecture in depth. Although we have not resolved
the conjecture in its full generality, in what follows we summarise
our main findings and disucss the novelty of our work.

\vspace{2mm}\noindent \textit{Main results.} The general philosophy of our
results is that having infinite R.O.C. for the logarithmic signature
leads to rigid algebraic constraints on signature coefficients and
thus geometric constraints on the underlying path. 

\vspace{2mm}\noindent (A) The very first (and important) comment
is that the conjecture turns out to be a \textit{conjugacy class property.}
As a result, it needs to be modified accordingly (for otherwise, it
cannot be true in its original form). See Section \ref{subsec:LSConj}
(in particular, Proposition \ref{prop:RevConj}) below for the details.

\vspace{2mm}\noindent (B) Our first main result is that \textit{if
the logarithmic signature of a non-closed path has infinite R.O.C.,
the path has to satisfy a specific type of line integral identities}.
Such integral identities naturally lead to geometric constraints on
the underlying path, and in certain special situations, yield the
conclusion of the conjecture. See Section \ref{sec:LineInt} below for the precise
formulation.

\vspace{2mm}\noindent (C) Our second main result is a strengthened
version of the conjecture which holds in the most general BV setting.
More precisely, we prove that \textit{if the logarithmic signature
of an arbitrary BV path has infinite R.O.C. over any time interval
$[s,t]$, the path must live on a straight line}. This result is a non-trivial
application of Theorem \ref{thm:LineInt} together with fine analysis of path-geometry based on winding number considerations. See Section \ref{sec:StrCj} below for more details. 

\vspace{2mm}\noindent (D) Our last main result is the higher order
extension of (B). We show that if the logarithmic signature of
a non-closed path has infinite R.O.C., the path has to satisfy \textit{an
infinite system of specific iterated integral identities}. These iterated
integral identities impose much stronger geometric constraints on
the path than the line integral identities given by Theorem
\ref{thm:LineInt} do. In particular, they allow one to handle paths that cannot
be detected by Theorem \ref{thm:LineInt} (cf. Example \ref{exam:Eight} below).

\vspace{2mm}\noindent\textit{Novelty}. Apart from the nature of the aforementioned
main results, an important aspect of novelty in the current work is
the methodology we develop. The original technique of Lyons-Sidorova
\cite{LS06} has a geometric nature; they studied the geometric behaviour of the lifted
path in the Lie group and showed that the lifted path will eventually
leave the image of the exponential map under some very special choice
of Cartan's development. 

We also work with Cartan's development as a starting point, however,
our method is very different in nature to the one developed in \cite{LS06}.
Essentially, the finiteness of R.O.C. property is closely related
to how the Baker-Campbell-Hausdorff (BCH) formula exhibits singularity
properties in a certain sense. This problem is already quite delicate
in the simplest case of a piecewise linear path with two edges, say
$v,w$ (which corresponds to the classical BCH formula for $\log e^{v}e^{w}$).
The Lyons-Sidorova conjecture can be viewed as an infinite dimensional
version of singularity analysis for the BCH formula. 

A main novelty of the current work is that such singularity properties
can be studied through \textit{Cartan's developments into complex
semisimple Lie algebras}. It turns out that the special root patterns of semisimple Lie algebras allows one to project
the infinite dimensional BCH formula in a very special (finite dimensional)
manner, so that one can perform singularity analysis at various explicit
and quantitative levels. 

Our approach thus has a strong algebraic nature in contrast to Lyons-Sidorova's
geometric perspective. Stated in vague terms, after performing suitable
Cartan's developemnt at the Lie algebra level, the logarithmic signature
$\log S(\gamma)$ is transformed into linear combinations of products
between certain path functionals, say $S_{m}(a_{1},\cdots,a_{m})$
(defined through iterated path integrals) and meromorphic functions,
say $\phi(a_{1},\cdots,a_{m})$ (arising from the Hausdorff series).
If $\log S(\gamma)$ has infinite R.O.C., its Cartan's development
must define an entire function. As a consequence, at the singularities
of the function $\phi(a_{1},\cdots,a_{m})$ one must have $S_{m}(a_{1},\cdots,a_{m})=0$.
This is precisely how the integral identities in our main theorems
arise. 

\vspace{2mm} Cartan's development was first used by Hambly-Lyons \cite{HL10} for studying the signature uniqueness problem in the BV setting. Such a method was used in various works related to the study of (both deterministic and random) rough paths and their signatures (see e.g. \cite{BG23, CT24, CL16, LX15}). The use of Cartan's development at the logarithmic signature level first appeared in \cite{BGS20} for the study of signature asymptotics for pure rough paths.

\subsection{Further questions}

There are a few basic questions which naturally arise from the current work but remain to be answered.
\begin{enumerate}
\item The strengthened LS conjecture is proved in Theorem \ref{thm:StrCj} for BV paths only.
It is reasonable to expect that for weakly geometric rough paths,
the logarithmic signature has infinite R.O.C. over all time intervals
$[s,t]$ if and only if the path is a \textit{pure rough path}, i.e.
the exponential of a line segment in the space of Lie polynomials.
Is such an extension true?
\item It is not particularly clear what exact geometric constraints are
imposed on the path by Theorems \ref{thm:DoubInt}, \ref{thm:IterIntCond} and whether they eventually
lead to the conclusion that the path is conjugate to a line segment. We
expect that there is still a non-trivial gap which requires deeper
analysis and new ideas. 
\item Theorems \ref{thm:DoubInt}, \ref{thm:IterIntCond} only hold for non-closed paths, which is indeed a limitation
of the current analysis. Can one extend the current approach to the
case when the underlying path is a loop? The modified LS conjecture
in the context of loops takes a particularly elegant form: the logarithmic
signature of a tree-reduced, closed, BV path $\gamma$ has infinite R.O.C.
if and only if $\gamma$ is constant. 
\item Let $\xi\in G((V))$ be a group-like element. What are the exact analytic
conditions on $\xi$ so that $\xi$ is the signature of
a BV path over $V$? At the signature level, having exact factorial
decay seems to be the natural analytic constraint on $\xi$. However, at the
logarithmic signature level one encounters the finiteness of R.O.C.
property. Are these two analytic properties related to each other
and how do they yield a suitable characterisation of the image of
the signature transform?
\end{enumerate}

\section{The logarithmic signature transform}

In this section, we review the background materials. In particular,
we discuss several basic properties of the signature and logarithmic
signature transforms where most details can be found in \cite{FV10, LCL07}.
We also recall the precise formulation of the Lyons-Sidorova conjecture from \cite{LS06}, which provides the main source of motivation for the current
work. 

\subsection{Definitions and basic properties}

Let $V$ be a finite dimensional normed vector space over $\mathbb{K}=\mathbb{R}$
or $\mathbb{C}$. We define the infinite tensor algebra 
\[
T((V))\triangleq\prod_{n=0}^{\infty}V^{\otimes n}
\]
where $V^{\otimes0}\triangleq\mathbb{K}$ as a convention. Addition
in $T((V))$ is defined component-wisely and multiplication is defined
by 
\[
(\xi\otimes\eta)_{n}\triangleq\sum_{k=0}^{n}\xi_{k}\otimes\eta_{n-k},\ \ \ n\geqslant0
\]
where $\xi=(\xi_{0},\xi_{1},\cdots)$, $\eta=(\eta_{0},\eta_{1},\cdots)$
are given tensor series. This makes $T((V))$ into an associative
$\mathbb{K}$-algebra with unit ${\bf 1}\triangleq(1,0,0,\cdots).$
Elements in $T((V))$ are known as \textit{formal tensor series} over $V.$
We use 
\[
\pi_{n}:T((V))\rightarrow V^{\otimes n},\ \pi^{(n)}:T((V))\rightarrow\bigoplus_{k=0}^{n}V^{\otimes k}
\]
to denote the canonical projections. 

Any tensor series of the form $\xi=(1,\xi_{1},\xi_{2},\cdots)$ has
a multiplicative inverse given by 
\[
\xi^{-1}=\sum_{n=0}^{\infty}(-1)^{n}(\xi-{\bf 1})^{\otimes n}.
\]
The above series is well-defined since the summation is locally finite
(i.e. the projection onto the $m$-th component for each fixed $m$
only involves finitely many nonzero terms).

There are two basic operations over $T((V))$ that will be important
to us. To define them, let us introduce two subspaces: 
\[
T_{0}((V))\triangleq\{\xi=(\xi_{0},\xi_{1},\cdots)\in T((V)):\xi_{0}=0\},\ T_{1}((V))\triangleq\{\xi\in T((V)):\xi_{0}=1\}.
\]
The \textit{exponential} and \textit{logarithmic} transforms are defined
by 
\[
e^{(\cdot)}:T_{0}((V))\rightarrow T_{1}((V)),\ e^{\xi}\triangleq\sum_{n=0}^{\infty}\frac{1}{n!}\xi^{\otimes n};
\]
\[
\log(\cdot):T_{1}((V))\rightarrow T_{0}((V)),\ \log\xi\triangleq\sum_{n=1}^{\infty}\frac{(-1)^{n-1}}{n}(\xi-{\bf 1})^{\otimes n}.
\]
It can be shown that these two functions are inverse to each other.
It is also easy to see that $(e^{\xi})^{-1}=e^{-\xi}$.
\begin{defn}
Let $\gamma:[0,T]\rightarrow V$ be a continuous path with bounded
variation. The \textit{signature transform} of $\gamma$ (or simply
the signature) is the formal tensor series defined by 
\[
S(\gamma)\triangleq\big(1,\gamma_{T}-\gamma_{0},\int_{0<s<t<T}d\gamma_{s}\otimes d\gamma_{t},\cdots,\int_{0<t_{1}<\cdots<t_{n}<T}d\gamma_{t_{1}}\otimes\cdots\otimes d\gamma_{t_{n}},\cdots\big),
\]
where the path integrals are defined in the Lebesgue-Stieltjes sense.
The tensor series $\log S(\gamma)$ is known as the \textit{logarithmic
signature} of $\gamma$. 
\end{defn}
One can also consider the signature $S(\gamma)_{s,t}$ over $[s,t]\subseteq[0,T]$
by replacing the above iterated integrals with the ones over $[s,t]$.
It is easily check that
\begin{equation}
S(\gamma)_{s,u}=S(\gamma)_{s,t}\otimes S(\gamma)_{t,u}\ \ \ \forall s\leqslant t\leqslant u\in[0,T].\label{eq:ChenIden}
\end{equation}
This is known as \textit{Chen's identity}. Stated in a more general
form, one has 
\[
S(\alpha\sqcup\beta)=S(\alpha)\otimes S(\beta),
\]
where $\alpha\sqcup\beta$ means the concatenation between the two
paths $\alpha,\beta.$ It can also be shown that $S(\overleftarrow{\gamma})=S(\gamma)^{-1}$
where $\overleftarrow{\gamma}$ denotes the reversal of $\gamma$. 

\subsubsection*{Algebraic properties}

There are universal algebraic dependencies among different signature
components, which are precisely described by the so-called \textit{shuffle
product formula}:
\begin{equation}
X^{m}\otimes X^{n}=\sum_{\sigma\in{\cal P}(m,n)}P_{\sigma}(X^{m+n})\ \ \ \forall m,n\in\mathbb{N}.\label{eq:Shuffle}
\end{equation}
Here $X^{m}$ is the $m$-th component of $S(\gamma)$. ${\cal P}(m,n)$ denotes the set of $(m,n)$-shuffles,
i.e. the set of permutations $\sigma$ of order $m+n$
satisfying 
\[
\sigma(1)<\cdots<\sigma(m),\ \sigma(m+1)<\cdots<\sigma(m+n).
\]
Given any permutation $\sigma$, the operator $P_{\sigma}$ is the linear transformation over $V^{\otimes(m+n)}$
induced by 
\[
P_{\sigma}(v_{1}\otimes\cdots\otimes v_{m+n})\triangleq v_{\sigma(1)}\otimes\cdots\otimes v_{\sigma(m+n)}.
\]The set of permutations of order $m$ is denoted as $\mathcal{S}_m$.

\begin{defn}
A tensor series $\xi=(1,\xi_{1},\xi_{2},\cdots)\in T_{1}((V))$ is
 \textit{group-like} if it satisfies the relation (\ref{eq:Shuffle})
(with $X^{n}$ replaced by $\xi_{n}$). The space of group-like elements
is denoted as $G((V))$. 
\end{defn}

An important reason for considering the logarithmic signature is that
there are no algebraic dependencies among its components, because
it takes values in the ``free'' Lie algebra. This is the content
of the celebrated \textit{Chen's theorem}. To state this theorem,
we first define 
\[
{\cal L}((V))\triangleq\prod_{n=1}^{\infty}{\cal L}_{n}(V)\subseteq T_0((V))
\]
where ${\cal L}_{1}(V)\triangleq V$ and 
\[
\ {\cal L}_{n+1}(V)\triangleq[{\cal L}_{n}(V),V]\triangleq{\rm Span}\big\{[\xi,v]:\xi\in{\cal L}_{n}(V),v\in V\big\}
\]
for all $n\geqslant1.$ Here $[\cdot,\cdot]$ is the commutator defined
by $[\xi,\eta]\triangleq\xi\otimes\eta-\eta\otimes\xi$. Elements
in ${\cal L}((V))$ are \textit{formal Lie series} and elements in $\oplus_{k=1}^{n}{\cal L}_{k}(V)$
are \textit{Lie polynomials} of degree $n$. Chen's theorem is stated as follows
(cf. \cite{Che57, Reu93}).
\begin{thm}
A tensor series $\xi\in T_{1}((V))$ is group-like if and only if
$\log\xi\in{\cal L}((V))$. 
\end{thm}
One can also consider truncated versions of signatures and logarithmic
signatures. Let $T^{(N)}(V),$ $G^{(N)}(V)$ and ${\cal L}^{(N)}(V)$
be the truncations up to level $N$ (i.e. taking the first $N$
components) of the previous infinite dimensional spaces. The tensor
product and Lie bracket restrict to these truncated spaces in the
obvious way. 

\subsubsection*{Analytic properties}

We also need to deal with analytic properties of the signature. For
this purpose, we need to consider suitable tensor norms on the tensor
product spaces. From now on, we assume that $V^{\otimes n}$ is equipped
with a given norm $\|\cdot\|_{n}$. The norms $\{\|\cdot\|_{n}\}$
are \textit{admissible} in the sense that 
\[
\|\xi\otimes\eta\|_{m+n}\leqslant\|\xi\|_{m}\|\eta\|_{n}\ \ \ \forall m,n\in\mathbb{N}
\]
and 
\[
\|P_{\sigma}(\xi)\|_{n}=\|\xi\|_{n}\ \ \ \forall\xi\in V^{\otimes n},\sigma\in{\cal S}_{n}.
\]
An explicit example is the Hilbert-Schmidt tensor norm on $V^{\otimes n}$
induced by an Euclidean (or Hermitian if $\mathbb{K}=\mathbb{C}$)
metric on $V$. What plays a basic role in rough path theory is the so-called projective norm. 

\begin{defn}
The \textit{projective tensor norm} of $\xi\in V^{\otimes n}$ is defined by \[
\|\xi\|_{n;{\rm proj}}\triangleq\inf\big\{\sum_{j}|v_{1}^{j}|_{V}\cdots|v_{n}^{j}|_{V}\big\},
\]where the infimum is taken over all possible representations of $\xi$ as linear combinations of tensor monomials: \[
\xi=\sum_{j}v_{1}^{j}\otimes\cdots\otimes v_{n}^{j},\ \ \ v_{i}^j\in V.
\]
\end{defn}

Throughout the rest of this article, we will be exclusively using the projective tensor norm. For simplicity, we will also omit the subscripts if no confusion will be caused.

The signature transform can be defined for arbitrary rough paths.
Let $p\geqslant1$ be a given fixed number. A $p$-\textit{rough path}
is a continuous functional 
\[
{\bf X}_{\cdot,\cdot}=(1,X_{\cdot,\cdot}^{1},\cdots,X_{\cdot,\cdot}^{[p]}):\Delta_{T}\triangleq\{(s,t):0\leqslant s\leqslant t\leqslant T\}\rightarrow T^{([p])}(V)
\]
such that 
\[
{\bf X}_{s,u}={\bf X}_{s,t}\otimes{\bf X}_{t,u}\ \ \ \forall s\leqslant t\leqslant u\in[0,T]
\]
and ${\bf X}$ has finite $p$-variation in the sense that 
\begin{equation}\label{eq:pVar}
\omega_{{\bf X}}\triangleq\|\mathbf{X}\|^p_{p\text{-var}}\triangleq\sum_{i=1}^{[p]}\sup_{{\cal P}}\sum_{t_{l}\in{\cal P}}\|X_{t_{l-1},t_{l}}^{i}\|^{p/i}<\infty,
\end{equation}
where the supremum is taken over all finite partitions ${\cal P}$
of $[0,T]$. A $p$-rough path ${\bf X}$ is \textit{weakly geometric}
if it takes values in $G^{([p])}(V)$. Sometimes we just call $\gamma:[0,T]\rightarrow V$
a rough path over $V$ but its precise meaning is the multi-level
functional ${\bf X}$ ($\gamma$ is the first level component of ${\bf X}$).

Let ${\bf X}$ be a $p$-rough path. According to \textit{Lyons' extension
theorem} (cf. \cite{Lyo98}), there exists a unique extension of ${\bf X}_{\cdot,\cdot}$
to a functional 
\[
S({\bf X})_{\cdot,\cdot}=(1,X_{\cdot,\cdot}^{1},\cdots,X_{\cdot,\cdot}^{n},\cdots):\Delta_{T}\rightarrow T((V)),
\]
such that Chen's identity (\ref{eq:ChenIden}) holds for $S({\bf X})_{\cdot,\cdot}$
in $T((V))$ and its truncation up to any level $n\geqslant[p]$ has
finite $p$-variation in $T^{(n)}(V)$. The functional $S({\bf X})_{\cdot,\cdot}$
is called the \textit{signature path} of ${\bf X}$. The tensor series $S({\bf X})_{0,T}$
(respectively, $\log S({\bf X})_{0,T}$) is the \textit{signature} (respectively,
\textit{logarithmic signature}) of ${\bf X}$. If ${\bf X}$ is weakly geometric,
it can be shown that its signature path takes values in $G((V))$. 

An important analytic property of the signature is its rapid decay
with respect to the degree. More precisely, there exists a universal
number $\beta>0$, such that the following estimate
\begin{equation}
\|X_{0,T}^{n}\|\leqslant\frac{\omega_{{\bf X}}^{n/p}}{\beta\cdot(n/p)!}\ \ \ \forall n\geqslant1\label{eq:FacDec}
\end{equation}
holds for all $p$-rough paths. Here $(n/p)!\triangleq\Gamma(n/p+1)$
where $\Gamma(\cdot)$ is the Gamma function. When $p=1$, ${\bf X}$
is just a classical path in $V$ with bounded variation and (\ref{eq:FacDec})
follows trivially from the triangle inequality. The estimate (\ref{eq:FacDec})
for the general rough paths is contained as part of Lyons' extension theorem.

\subsection{\label{subsec:LSConj}The Lyons-Sidorova conjecture}

To state the main question of our study, we first introduce the following
definition. 
\begin{defn}
Let $\xi=(\xi_{0},\xi_{1},\cdots)\in T((V))$ be a given tensor series.
Its \textit{radius of convergence} (R.O.C.) is the radius of convergence
for the power series 
\[
P_{\xi}(z)\triangleq\sum_{n=0}^{\infty}\|\xi_{n}\|z^{n}
\]
in the standard real analysis sense. 
\end{defn}
The factorial decay (\ref{eq:FacDec}) of signature shows that the
signature of a rough path always has infinite R.O.C. However, a highly
non-trivial fact is that the logarithmic signature (generically) 
only decays geometrically fast and should thus have a \textit{finite}
R.O.C. Consider a line segment $\gamma_{t}=tv$ ($0\leqslant t\leqslant1$) in $V.$ It is immediate that $\log S(\gamma)=v$ and
thus $\log S(\gamma)$ has infinite R.O.C. It was conjectured
by T. Lyons and N. Sidorova \cite{LS06} that these are the only 
tree-reduced BV paths whose logarithmic signatures
can have infinite R.O.C. 

\vspace{2mm}\noindent \textbf{The Lyons-Sidorova (LS) Conjecture}.
The logarithmic signature of a continuous, tree-reduced BV path $\gamma$
has infinite R.O.C. if and only if $\gamma$ is a line segment. 

\vspace{2mm}\noindent In their original work \cite{LS06}, the conjecture
was confirmed for two special classes of paths: $1$-monotone paths
and non-double piecewise linear paths. The general BV case is still an unsolved open
problem in rough path theory.

Before developing our main results, the very first comment we shall
make is that \textit{the LS conjecture needs to be modified} to reflect
the following simple observation.
\begin{proposition}
\label{prop:RevConj}Let ${\bf X},{\bf Y}$ be given rough paths.
Then $\log S({\bf X})$ has infinite R.O.C. if and only if $\log S({\bf Y}\sqcup{\bf X}\sqcup\overleftarrow{{\bf Y}})$
has infinite R.O.C. 
\end{proposition}
\begin{proof}
Let us prove a more general claim. Suppose that $l\in T_{0}((V))$
and $g_{1},g_{2}$ are signatures (of certain rough paths). Then $l$
has infinite R.O.C. if and only if $g_{1}\otimes l\otimes g_{2}$
has infinite R.O.C.

To prove this claim, since $g_{1},g_{2}$ are signatures, one knows
from Lyons' factorial decay estimate (\ref{eq:FacDec}) that 
\begin{equation}
\|\pi_{n}(g_{i})\|\leqslant\frac{C\cdot\omega^{n/p}}{(n/p)!}\ \ \ \forall n\geqslant1,\ i=1,2.\label{eq:FacDecSCPf}
\end{equation}
Here $C,p,\omega$ are suitable constants depending on the underlying
paths defining $g_{1},g_{2}.$ Let us denote $l'\triangleq g_{1}\otimes l\otimes g_{2}$.
Then one has 
\[
\pi_{n}(l')=\sum_{k=1}^{n}\sum_{r=0}^{n-k}\pi_{r}(g_{1})\otimes l_{k}\otimes\pi_{n-k-r}(g_{2}),
\]
where $l_{k}\triangleq\pi_{k}(l)$. Since the tensor norms are admissible,
one finds that 
\[
\|\pi_{n}(l')\|\leqslant C^{2}\sum_{k=1}^{n}\sum_{r=0}^{n-k}\frac{\omega^{r/p}}{(r/p)!}\frac{\omega^{(n-k-r)/p}}{\big((n-k-r)/p\big)!}\|l_{k}\|.
\]
We now apply the following so-called \textit{neo-classical inequality}
(cf. \cite{HH10}):
\begin{equation}\label{eq:Neo}
\sum_{i=0}^{m}\frac{1}{(i/p)!\big((m-i)/p\big)!}\leqslant\frac{p2^{m/p}}{(m/p)!}\ \ \ \forall m\in\mathbb{N},\ p\geqslant1.
\end{equation}
It follows that 
\begin{equation}
\|\pi_{n}(l')\|\leqslant C^{2}p\sum_{k=1}^{n}\frac{(2\omega)^{(n-k)/p}}{\big((n-k)/p\big)!}\|l_{k}\|.\label{eq:SCPf2}
\end{equation}

Let $\rho>0$ be given fixed. Since $l$ has infinite R.O.C., there exists $K>0$
such that $\|l_{k}\|\leqslant\rho^{k}$ for all $k>K$. Therefore,
one has from (\ref{eq:SCPf2}) that 
\begin{align*}
\|\pi_{n}(l')\| & \leqslant C^{2}p\big(\sum_{k=1}^{K}\frac{(2\omega)^{(n-k)/p}}{\big((n-k)/p\big)!}\|l_{k}\|+\sum_{k=K+1}^{n}\frac{(2\omega)^{(n-k)/p}}{\big((n-k)/p\big)!}\rho^{k}\big)\\
 & \leqslant C^{2}p\big(\sum_{k=1}^{K}\frac{(2\omega)^{(n-k)/p}}{\big((n-k)/p\big)!}\|l_{k}\|+\rho^{n}\sum_{k=0}^{\infty}\frac{(2\omega\rho^{-p})^{k/p}}{(k/p)!}\big)
\end{align*}
for all $n>K.$ When $n$ is sufficiently large, one can ensure that
\[
\sum_{k=1}^{K}\frac{(2\omega)^{(n-k)/p}}{\big((n-k)/p\big)!}\|l_{k}\|<\rho^{n}
\]
and thus 
\[
\|\pi_{n}(l')\|\leqslant C^{2}p\big(1+\sum_{k=0}^{\infty}\frac{(2\omega\rho^{-p})^{k/p}}{(k/p)!}\big)\rho^{n}.
\]
This shows that 
\[
\underset{n\rightarrow\infty}{\overline{\lim}}\|\pi_{n}(l')\|^{1/n}\leqslant\rho.
\]
Since $\rho$ is arbitrary, one concludes that $l'$ has infinite
R.O.C.

In the context of the proposition, one takes 
\[
l=\log S({\bf X}),\ g_{1}=S({\bf Y}),\ g_{2}=S(\overleftarrow{{\bf Y}})=g_{1}^{-1}.
\]
The crucial observation is that 
\[
S({\bf Y}\sqcup{\bf X}\sqcup\overleftarrow{{\bf Y}})=g_{1}\otimes e^{l}\otimes g_{1}^{-1}=e^{g_{1}\otimes l\otimes g_{1}^{-1}}
\]
which implies that 
\[
\log S({\bf Y}\sqcup{\bf X}\sqcup\overleftarrow{{\bf Y}})=g_{1}\otimes l\otimes g_{1}^{-1}.
\]
The result follows from the above claim since $g_{1}$ is a signature. 
\end{proof}
Proposition \ref{prop:RevConj} shows that the logarithmic signature
of any path that is conjugate to a line segment must have infinite
R.O.C. Here we say that a path $\beta$ is \textit{conjugate to} $\gamma$
if ${\bf \beta}=\alpha\sqcup\gamma\sqcup\overleftarrow{\alpha}$ for
some path $\alpha$. This result suggests that having finite R.O.C.
is a property of conjugate classes. This naturally leads to the following
modification of the LS conjecture. 

\vspace{2mm}\noindent \textbf{Modified LS Conjecture.} The logarithmic
signature of a continuous, tree-reduced BV path has infinite R.O.C.
if and only if it is conjugate to a line segment. 

\begin{rem}
Since $V$ is finite dimensional, the property of having infinite / finite R.O.C. does not depend on the specific choices of the (admissible) tensor norms. 
\end{rem}

\begin{rem}\label{rem:PosROC}
It was proved by Chevyrev-Lyons \cite{CL16} that the logarithmic signature of any rough path always has positive R.O.C.
\end{rem}

\section{First order integral identities}\label{sec:LineInt}

In this section, we establish our first main result:
\[
\text{Infinite R.O.C. for log signature}\implies\text{A special type of line integral identities.}
\]
Although we could formulate the theorem in a more general form, the result is essentially stated in terms of two-dimensional projections. To make the statement as simple as possible, we will just (and always) assume that $\gamma_{t}=(x_{t},y_{t})_{0\leqslant t\leqslant1}$ is a weakly geometric rough path over $\mathbb{R}^2$. In this section, we only work with \textit{non-closed} paths (i.e. $\gamma_1\neq\gamma_0$). In particular, we impose the following normalisation conditions: 
\begin{equation}\label{eq:PathNorm}
x_{0}=y_{0}=0,\ x_{1}=1.
\end{equation}
Note that $\gamma$ is the first level component of the actual rough path which will not be referred to. The main result of this section is stated as follows.
\begin{thm}
\label{thm:LineInt}Let $\gamma$ be a weakly geometric rough path over $\mathbb{R}^2$ which satisfies the normalisation condition (\ref{eq:PathNorm}). Suppose that $\log S(\gamma)$ has infinite R.O.C. Then one has 
\begin{equation}
\int_{0}^{1}e^{2k\pi i\cdot x_{t}}dy_{t}=0\label{eq:LineInt}
\end{equation}for all nonzero integers $k$.
\end{thm}

\begin{rem}\label{rem:Genx1}
If the condition $x_1=1$ is dropped (still assuming $x_1\neq0$), the conclusion (\ref{eq:LineInt}) should be replaced by $\int_0^1e^{2k\pi i x_t/x_1}dy_t = 0$. For a general path over $V$, the result is formulated in terms of arbitrary projections of $\gamma$ onto two-dimensional subspaces of $V$ where $(x_t,y_t)$ are the coordinates of the projected path with respect to a suitable basis. 
\end{rem}

\begin{rem}
If one assumes additionally that $y_{1}=0$, then (\ref{eq:LineInt})
holds for all $k\in\mathbb{Z}.$ In this case, one has $\int_{0}^{1}f(x_{t})dy_{t}=0$
for all smooth, $1$-periodic functions $f$. This result
is stronger than the one obtained by Lyons-Sidorova's method in \cite{LS06}; the
latter established the integral property (\ref{eq:LineInt}) for $k\in2\mathbb{Z}+1$ which is difficult to be extended to cover even $k$'s using their method.
\end{rem}

A more general (and useful) formulation of Theorem \ref{thm:LineInt} is given as follows. 

\begin{thm}
\label{thm:GenLineInt}Under the assumptions of Theorem \ref{thm:LineInt},
suppose further that $y_{1}=0$. Then one has
\begin{equation}
\int_{0}^{1}\Phi(d\gamma_{t})=0\label{eq:GenLineInt}
\end{equation}
for all smooth one-forms $\Phi(x,y)=f(x,y)dx+g(x,y)dy$ on $\mathbb{R}^{2}$
which satisfy the following two conditions:

\vspace{2mm}\noindent (i)
$\Phi(x+1,y)=\Phi(x,y)$ for all $x,y\in\mathbb{R}$;

\vspace{2mm}\noindent (ii) $\int_{0}^{1}f(x,y)dx=0$ for all $y\in\mathbb{R}$.
\end{thm}

To prove Theorem \ref{thm:GenLineInt}, we first present a basic lemma. 

\begin{lem}
\label{lem:ProjNormLinear}Let $A:\mathbb{R}^{2}\rightarrow\mathbb{R}^{2}$
be a linear map, where $\mathbb{R}^2$ is equipped with the standard Euclidean norm. Let $\hat{A}:\left(\mathbb{R}^{2}\right)^{\otimes n}\rightarrow\left(\mathbb{R}^{2}\right)^{\otimes n}$
be the unique linear map such that 
\[
\hat{A}\left(v_{1}\otimes\cdots\otimes v_{n}\right)=Av_{1}\otimes\cdots\otimes Av_{n}.
\]
Then one has
\[
\Vert\hat{A}\left(w\right)\Vert\leqslant\Vert A\Vert_{\mathrm{op}}^{n}\Vert w\Vert
\]for all $w\in\left(\mathbb{R}^{2}\right)^{\otimes n}.$ Here $\|A\|_{\rm op}$ denotes the operator norm of $A$ and the tensor products are all equipped with the projective tensor norm.
\end{lem}

\begin{proof}
Suppose that $w$ admits a representation given by
\begin{equation}
w=\sum_{j=1}^{N}v_{1}^{j}\otimes\cdots\otimes v_{n}^{j}.\label{eq:TensorDecom}
\end{equation}
Then one has 
\begin{align*}
\Vert\hat{A}\left(w\right)\Vert & =\big\Vert\sum_{j=1}^{N}\hat{A}\big(v_{1}^{j}\otimes\cdots\otimes v_{n}^{j}\big)\big\Vert
 \leqslant\big\Vert\sum_{j=1}^{N}A\left(v_{1}^{j}\right)\otimes\cdots\otimes A\left(v_{n}^{j}\right)\big\Vert\\
 & \leqslant\sum_{j=1}^{N}\big\Vert A\left(v_{1}^{j}\right)\big\Vert\cdots\big\Vert A\left(v_{n}^{j}\right)\big\Vert\\
 & \leqslant\sum_{j=1}^{N}\big\Vert A\big\Vert_{\text{op}}^{n}\big\Vert v_{1}^{j}\big\Vert\cdots\big\Vert v_{n}^{j}\big\Vert.
\end{align*}
Taking infimum over all $\left\{ v_{1}^{j},\ldots,v_{n}^{j}\right\} _{j=1}^{N}$
satisfying (\ref{eq:TensorDecom}) gives the Lemma.
\end{proof}

\begin{proof}[Proof of Theorem \ref{thm:GenLineInt}]
Let $\gamma$ be a path satisfying $\gamma_0 = (0,0)$ and $\gamma_1 = (1,0)$. Suppose that $\log S(\gamma)$ has infinite R.O.C. One knows from Theorem \ref{thm:LineInt} that 
$\int_{0}^{1}e^{2k\pi ix_{s}}dy_{s}=0$  for all $k\in\mathbb{Z}$. Given $k\in\mathbb{Z}\setminus\{0\}$ and $\lambda\in\mathbb{R}$,
we define 
\begin{equation}\label{eq:HatGam}
\hat{\gamma}_{s}=\left(\begin{array}{ll}
1 & \frac{\lambda}{2k\pi}\\
0 & 1
\end{array}\right)\left(\begin{array}{l}
x_{s}\\
y_{s}
\end{array}\right).
\end{equation}
Note that the $n$-th component of $\log S\left(\hat{\gamma}\right)$
is obtained by applying $\hat{A}$ to the $n$-th component of $\log S\left(\gamma\right)$ ($A$ is the transform defined by the matrix in (\ref{eq:HatGam})).
According to Lemma \ref{lem:ProjNormLinear}, $\log S\left(\hat{\gamma}\right)$ has infinite R.O.C. with respect to the projective tensor norm.
The new path $\hat{\gamma}$ also satisfies $\hat{\gamma}_{0}=(0,0)$
and $\hat{\gamma}_{1}=(1,0)$. As a result, one has
\begin{equation}\label{eq:GenLinePf}
\int_{0}^{1}e^{2k\pi ix_{s}+\lambda iy_{s}}dy_{s}=\int_{0}^{1}e^{2k\pi i\hat{x}_{s}}d\hat{y}_{s}=0
\end{equation}
for $k\in\mathbb{Z}\setminus\{0\}$ and $\lambda\in\mathbb{R}$. It is plain to check that (\ref{eq:GenLinePf}) holds when $k=0$ (and thus for all $k\in\mathbb{Z}$ and $\lambda\in\mathbb{R}$).
\begin{comment}

If $k=0$ but $\lambda\in\mathbb{R}\setminus\{0\}$, one has 
\[
\int_{0}^{1}e^{2k\pi ix_{s}+\lambda iy_{s}}dy_{s}=\int_{0}^{1}e^{\lambda iy_{s}}dy_{s}=\frac{1}{\lambda i}(e^{\lambda iy_{s}})|_{0}^{1}=0.
\]
If $k=\lambda=0$, then $$\int_{0}^{1}e^{2k\pi ix_{s}+\lambda iy_{s}}dy_{s}=y_{1}-y_{0}=0.$$
As a result, the relation $$\int_{0}^{1}e^{2k\pi ix_{s}+\lambda iy_{s}}dy_{s}=0$$
holds for all $k\in\mathbb{Z}$, $\lambda\in\mathbb{R}$. 
\end{comment}

Now let 
\[
K_{1}\triangleq\inf\left\{ y_{t}:t\in\left[0,1\right]\right\} ,\ K_{2}\triangleq\sup\left\{ y_{t}:t\in\left[0,1\right]\right\} 
\]
and recall $\mathbb{S}^{1}$ is the topological circle obtained by
identifying the points $0$ with $1$ on $\left[0,1\right]$. Note that
the set of functions 
\[
S=\big\{ e^{2k\pi ix+\lambda iy}:k\in\mathbb{Z},\lambda\in\mathbb{R}\big\} 
\]
separate points on the compact space $\mathbb{S}^{1}\times\left[K_{1},K_{2}\right]$ and its linear span is a sub-algebra of continuous functions. 
By the Stone-Weierstrass theorem, one knows that $\mathrm{Span{S}}$ that
\begin{comment}the unital $^{*}$-algebra
generated by $S$, i.e.
\[
\left\{ \left(x,y\right)\rightarrow\sum_{k,\lambda\in J}e^{2k\pi ix+2\lambda\pi iy}:k\in\mathbb{Z},\lambda\in\mathbb{R}\right\} 
\]
\end{comment}
is dense in $C\left(\mathbb{S}^{1}\times\left[K_{1},K_{2}\right],\mathbb{C}\right)$. Since $g\rightarrow\int_{0}^{1}g\left(x_{s},y_{s}\right)dy_{s}$
is continuous with respect to the uniform norm, one concludes that $\int_{0}^{1}g(x_{s},y_{s})dy_{s}=0$  for any continuous
function $g(x,y)$ that is $1$-periodic in the $x$-variable. This gives the desired integral condition for the $dy$-integral.

To prove the corresponding relation for the $dx$-integral, let $k\in\mathbb{Z}\setminus\{0\}$ and $\lambda\in\mathbb{R}\setminus\{0\}$.
Using integration by parts, one finds that
\[
\begin{aligned}0= & \int_{0}^{1}e^{2k\pi ix_{s}+\lambda iy_{s}}dy_{s}=\frac{1}{\lambda i}\int_{0}^{1}e^{2k\pi ix_{s}}de^{\lambda iy_{s}}\\
= & \frac{1}{\lambda i}\big(e^{2k\pi ix_{s}+\lambda iy_{s}}|_{0}^{1}-2k\pi i\int_{0}^{1}e^{2k\pi ix_{s}+\lambda iy_{s}}dx_{s}\big)\\
= & -\frac{k}{\lambda}\int_{0}^{1}e^{2k\pi ix_{s}+\lambda iy_{s}}dx_{s}.
\end{aligned}
\]
Since $k\neq0$, one obtains that 
\begin{equation}
\int_{0}^{1}e^{2k\pi ix_{s}+\lambda iy_{s}}dx_{s}=0.\label{eq:ExponInteral}
\end{equation}
The continuity of the
integral as a function of $\lambda$ implies that the relation (\ref{eq:ExponInteral})
 holds for all $k\in\mathbb{Z}\setminus\{0\}$ and  $\lambda\in\mathbb{R}$. By exactly the same Stone-Weierstrass argument, one concludes that $\int_0^1f(x_s,y_s)dx_s=0$ for any continuous function $f(x,y)$ which is $1$-periodic in $x$ and satisfies $\int_0^1f(x,y)dx=0$ for every $y\in\mathbb{R}$ (the latter constraint appears since the relation (\ref{eq:ExponInteral}) excludes the zero Fourier mode). 

 Now the proof of Theorem \ref{thm:GenLineInt} is complete. 
 
 \begin{comment}
 As a result, the relation (\ref{eq:ExponInteral})
holds for all $k\in\mathbb{Z}\setminus\{0\}$ and $\lambda\in\mathbb{R}$.
In particular, the below holds for any $k\in\mathbb{Z}\backslash\left\{ 0\right\} $
and $\lambda\in\mathbb{R}$, 
\[
\int_{0}^{1}e^{2k\pi ix_{s}+2\lambda\pi iy_{s}}dx_{s}=\int_{0}^{1}\left(\int_{0}^{1}e^{2k\pi ix+2\lambda\pi iy_{s}}dx\right)dx_{s}.
\]
By exactly the same Stone-Weierstrass argument we saw previously in
this proof, we have that the following for all continuous functions
$f_{}:\mathbb{S}^{1}\times\left[K_{1},K_{2}\right]\rightarrow\mathbb{C}$
\[
\int_{0}^{1}f\left(x_{s},y_{s}\right)dx_{s}=\int_{0}^{1}\left(\int_{0}^{1}f\left(x,y_{s}\right)dx\right)dx_{s}.
\]
This is equivalent to saying, for any continuous $f$ such that $f(x+1,y)=f(x,y)$
and for all $y$ $\int_{0}^{1}f(x,y)dx=0$, we have $\int_{0}^{1}f(x_{s},y_{s})dx_{s}=0$.
\end{comment}

\end{proof}

\subsection{Two immediate applications}

The line integral condition (\ref{eq:LineInt}) (or more generally, (\ref{eq:GenLineInt}))
implicitly leads to rigid geometric constraints on the underlying
path, and in some special situations, confirms the LS conjecture. We
use one class of examples to illustrate this point. A more inspiring application of Theorem \ref{thm:GenLineInt} on path-geometry is given in Section \ref{sec:StrCj} below where a strengthened version of the LS conjecture is proved.
\begin{proposition}
Let $\gamma_{t}=(x_{t},y_{t})$ be a two-dimensional, continuous BV
path satisfying the assumptions in Theorem \ref{thm:GenLineInt}.
Suppose further that $0\leqslant x_{t}\leqslant1$ and $\gamma$ is
non-self-intersecting. Then there exists $\varepsilon\in\mathbb{R}$
such that $\gamma=(\varepsilon{\rm e}_{2})\sqcup{\rm e}_{1}\sqcup(-\varepsilon{\rm e}_{2})$,
where $\{{\rm e}_{1},{\rm e}_{2}\}$ is the canonical basis of $\mathbb{R}^{2}$. 
\end{proposition}
\begin{proof}
Suppose on the contrary that $\gamma$ is not of the form $(\varepsilon{\rm e}_{2})\sqcup{\rm e}_{1}\sqcup(-\varepsilon{\rm e}_{2})$.
Then there exist two points $z=(x_{1},y_{1})$ and $w=(x_{2},y_{2})$
such that $0<x_{1},x_{2}<1,$ $y_{1}=y_{2}$ and $z\in{\rm Im}\gamma$,
$w\notin{\rm Im}\gamma$. One can choose a small $\delta>0,$ such
that by defining 
\[
U_{\delta}\triangleq\{(x,y):|x-x_{1}|\vee|y-y_{1}|<\delta\},\ V_{\delta}\triangleq\{(x,y):|x-x_{2}|\vee|y-y_{2}|<\delta\}
\]
one has $U_{\delta}\cap V_{\delta}=\emptyset$ and $V_{\delta}\cap{\rm Im}\gamma=\emptyset$. 

By using the argument in \cite{BGL16}, one can construct a smooth one-form
$\varphi$ which is compactly supported in $U_{\delta}$ such that
$\int_{0}^{1}\varphi(d\gamma_{t})\neq0.$ We define the one-form $\psi$
by 
\[
\psi(x,y)=\begin{cases}
\varphi(x,y), & (x,y)\in U_{\delta};\\
-\varphi(x-x_{2}+x_{1},y), & (x,y)\in V_{\delta}.
\end{cases}
\]
It is clear that $\psi$ is a smooth one-form supported in $U_{\delta}\cap V_{\delta}$
and satisfies 
\[
\int_{0}^{1}f(x,y)dx=0
\]
for all $y\in\mathbb{R}$, where $f(x,y)$ denotes the $dx$-coefficient
of $\psi.$ Let $\Phi$ denote the $1$-periodic extension of $\psi$
(in the $x$-direction). It follows that $\Phi$ satisfies Properties
(i), (ii) in Theorem \ref{thm:GenLineInt} but
\[
\int_{0}^{1}\Phi(d\gamma_{t})=\int_{0}^{1}\varphi(d\gamma_{t})\neq0.
\]
This leads to a contradiction. Therefore, $\gamma=(\varepsilon{\rm e}_{2})\sqcup{\rm e}_{1}\sqcup(-\varepsilon{\rm e}_{2})$
for some $\varepsilon\in\mathbb{R}$. 
\end{proof}
\begin{rem}
The above argument fails if the image of $\gamma$ is not contained
in the strip $0\leqslant x\leqslant1$, since the geometry of $\gamma$
outside the strip may affect the line integral against the one-form
$\Phi$ constructed before in an unexpected way. It is possible to
replace this geometric assumption by another kind of geometric restriction.
But we will not pursue such discussions here. 
\end{rem}
Another application of Theorem \ref{thm:LineInt} is the Brownian
motion case. This is an important example due to its basic role
in the study of stochastic differential equations.
\begin{proposition}
Let $B_{t}$ be a standard $d$-dimensional Brownian motion over $[0,1].$
Then with probability one, its logarithmic (Stratonovich) signature
has finite R.O.C.
\end{proposition}
\begin{proof}
It is obvious that one only needs to prove the claim for the first
two components of $B_{t}$, say $(X_{t},Y_{t})$. According to Theorem
\ref{thm:LineInt} (see also Remark \ref{rem:Genx1}), one only needs
to show that 
\[
\int_{0}^{1}\sin(2\pi X_{t}/X_{1})dY_{t}\neq0\ \text{a.s.}
\]
This can be obtained by proving the stronger claim that the random
variable $Z\triangleq\int_{0}^{1}\cos(2\pi X_{t}/X_{1})dY_{t}$ admits
a density with respect to the Lebesgue measure on $\mathbb{R}$. 

We use Malliavin's calculus to prove such a claim. We assume that
$(X_{t},Y_{t})$ are realised on the canonical path space. To prove
the claim, it suffices to show that the Malliavin derivative $DZ\neq0$
a.s. (cf. \cite{Nua06}). Let $h$ be any Cameron-Martin path (i.e. absolutely
continuous with $L^{2}$-derivative). Explicit calculation together
with integration by parts shows that the Malliavin derivative of $Z$
along the direction $h$ is given by 
\[
D_{h}Z=\int_{0}^{1}\big[\frac{2\pi(U_{1}-U_{t})}{X_{1}}-\frac{2\pi V}{X_{1}^{2}}+\sin\big(\frac{2\pi X_{t}}{X_{1}}\big)\big]dh_{t},
\]
where 
\[
U_{t}\triangleq\int_{0}^{t}\cos\big(\frac{2\pi X_{t}}{X_{1}}\big)dY_{t},\ V\triangleq\int_{0}^{1}X_{t}\cos\big(\frac{2\pi X_{t}}{X_{1}}\big)dY_{t}.
\]
We now choose 
\[
h_{t}\triangleq\int_{0}^{t}\big[\frac{2\pi(U_{1}-U_{s})}{X_{1}}-\frac{2\pi V}{X_{1}^{2}}+\sin\big(\frac{2\pi X_{s}}{X_{1}}\big)\big]ds.
\]
It follows that 
\[
D_{h}Z=\int_{0}^{1}\big[\frac{2\pi(U_{1}-U_{t})}{X_{1}}-\frac{2\pi V}{X_{1}^{2}}+\sin\big(\frac{2\pi X_{t}}{X_{1}}\big)\big]^{2}dt.
\]
The above integral is nonzero with probability one, for otherwise
the integrand would be identically zero which is clearly not the case.
Therefore, $DZ\neq0$ a.s. 
\end{proof}

\subsection{Proof of Theorem \ref{thm:LineInt}}

The fundamental idea behind our proof of Theorem \ref{thm:LineInt}, as well as its higher order extensions (cf. Theorem \ref{thm:DoubInt} and Theorem \ref{thm:IterIntCond} below), is based on the notion of \textit{Cartan's path development}. Stated conceptually, we  ``develop'' the logarithmic signature into a suitably chosen (complex) Lie algebra $\mathfrak{g}$, so that the Lie structure of $\mathfrak{g}$ ``projects'' the logarithmic signature $\log S(\gamma)$ in a very special way to yield products between certain explicit meromorphic functions and line integrals along $\gamma$ (e.g. see the relation (\ref{eq:FormalHaus}) below). If $\log S(\gamma)$ has infinite R.O.C., its development must produce an entire function. As a result, the line integral must have zeros matching the singularities of the meromorphic function. This will naturally lead to the line integral condition (\ref{eq:LineInt}) at the discrete locations $k\in\mathbb{Z}\backslash\{0\}$. 

\begin{comment}
The essential idea behind the proof of Theorem \ref{thm:LineInt}
is to choose a development 
\[
F:\mathbb{R}^{2}\rightarrow\mathfrak{g},\ F({\rm e}_{1})=:A,\ F({\rm e}_{2})=:D
\]
into some finite dimensional complex Lie algebra $\mathfrak{g}$ such
that 
\[
{\rm ad}_{A}(D)\triangleq[A,D]=2k\pi i\cdot D,
\]
where $k\in\mathbb{Z}\backslash\{0\}$ is given fixed. As we will
see, this implies $\int_{0}^{1}e^{2k\pi i\cdot x_{t}}dy_{t}=0.$ Note
that such a choice is always possible. Indeed, let $\mathfrak{g}$
be a complex semisimple Lie algebra and let $A'$ be a nonzero element
of its Cartan subalgebra. Since ${\rm ad}_{A'}\in{\rm End}(\mathfrak{g})$
is diagonalisable and nonzero (due to semisimpleness), there exists
$D\in\mathfrak{g}$ and $\lambda\neq0$ such that ${\rm ad}_{A'}(D)=\lambda D.$
One then takes $A\triangleq2k\pi iA'/\lambda.$ In what follows, we
develop the analytic details to reveal why such a Lie algebraic development
is relevant to the problem.
\end{comment}

\subsubsection{The Baker-Campbell-Hausdorff formula}

Let $\{\mathrm{e}_1,\mathrm{e}_2\}$ be the standard basis of $\mathbb{R}^2$. Our analysis relies crucially on the consideration of the path $\tilde{\gamma}\triangleq\gamma\sqcup\overleftarrow{{\rm e}_{1}}$ as a starting point (here $\sqcup$ means concatenation). Clearly, the following relation holds on $T((\mathbb{R}^2))$: 
\begin{equation}\label{eq:LDecomp1}
e^{L(\gamma)}=e^{\tilde{L}(\gamma)}\otimes e^{{\rm e}_{1}},
\end{equation}where $L(\gamma),\tilde{L}(\gamma)$ denote the logarithmic signatures of $\gamma,\tilde{\gamma}$ respectively. Our next step is to apply the Baker-Campbell-Hausdorff
(BCH) formula to the right hand side of (\ref{eq:LDecomp1}), so that one can write 
\begin{equation}\label{eq:LDecomp2}
e^{\tilde{L}(\gamma)}\otimes e^{{\rm e}_{1}}=e^{B(\tilde{L}(\gamma),{\rm e}_{1})},
\end{equation}where the BCH functional
\[
B(\xi,\eta)=\xi+\eta+\frac{1}{2}[\xi,\eta]+\frac{1}{12}[\xi,[\xi,\eta]]-\frac{1}{12}[\eta,[\xi,\eta]]+\cdots
\]can be expressed in terms of commutators in a universal way.

We now recall the precise definition of the BCH functional $B(\cdot,\cdot)$. Let $v,w$ be two letters and let $V$ be the vector space generated by them. The quantity $B(v,w)$ is the Lie series defined by the following formula:
\begin{equation}
B(v,w)=\sum_{n=1}^{\infty}H_{n}(v,w)\in T((V)).\label{eq:HausRep}
\end{equation}Here  
\begin{equation}
H_{1}(v,w)\triangleq\sum_{m=0}^{\infty}\frac{B_{m}}{m!}({\rm ad}_{w})^{m}(v),\label{eq:H1Formula}
\end{equation}
where $\{B_{m}\}$ are the \textit{Bernoulli numbers} defined by the expansion 
\[
\frac{z}{e^{z}-1}=\sum_{m=0}^{\infty}\frac{B_{m}}{m!}z^{m}
\]and $${\rm ad}_{w}(v)\triangleq[w,v]\triangleq w\otimes v-v\otimes w.$$To define the higher order terms $H_{n}$,  we first recall the definition of a derivation. A linear operator $\mathfrak{D}:T((V))\rightarrow T((V))$ is called
a \textit{derivation} if 
\[
\mathfrak{D}(\xi\otimes\eta)=\mathfrak{D}(\xi)\otimes\eta+\xi\otimes\mathfrak{D}(\eta).
\]
Now let $H_{1}(v,w)\partial_w$ denote the derivation
induced by 
\[
(H_{1}(v,w)\partial_w)(v)\triangleq0,\ (H_{1}(v,w)\partial_w)(w)\triangleq H_{1}(v,w).
\]
Then the functional $H_{n}(v,w)$ is defined by 
\begin{equation}\label{eq:HnFormula}
H_{n}(v,w)\triangleq\frac{1}{n!}\big(H_{1}(v,w)\partial_w\big)^{n}(w).
\end{equation}It is clear that the partial degree of $v$ in $H_n(v,w)$ is $n$. The series $H_n$ will be referred to as the \textit{$n$-th Hausdorff series}.

To summarise, one obtains from the relations (\ref{eq:LDecomp1}) and (\ref{eq:LDecomp2}) that
\begin{equation}
L(\gamma)=B(\tilde{L}(\gamma),{\rm e}_{1})\in \mathcal{L}((\mathbb{R}^2)).\label{eq:HausRelationTensor}
\end{equation}Note that here $B(\tilde{L}(\gamma),\mathrm{e}_1)$ is defined through the substitution $(v,w)\leftrightarrow(\tilde{L}(\gamma),\mathrm{e}_1)$ and applying the tensor product structure over $T((\mathbb{R}^2))$. We remark that $B(\tilde{L}(\gamma),{\rm e}_{1})$ is
a well-defined Lie series in ${\cal L}((\mathbb{R}^{2}))$. This is
clear from the observation that the series 
\[
\sum_{n=1}^{\infty}H_{n}(\tilde{L}(\gamma),{\rm e}_{1})
\]
is locally finite, i.e. its projection onto the truncated tensor algebra
only involves finitely many Lie polynomials arising from the above series.

\subsubsection{An adjoint representation of signature}

We now present a basic formula for the signature $S(\gamma)$ on which our analysis is largely based. 

\begin{lem}
\label{lem:AdDecomp}Let $S(\gamma)$ be the signature of $\gamma$.
Then one has
\begin{equation}\label{eq:SigAdRep}
S(\gamma)=\big(\sum_{n=0}^{\infty}\int_{0<t_{1}<\cdots<t_{n}<1}e^{x_{t_{1}}{\rm ad}_{{\rm e}_{1}}}({\rm e}_{2})\otimes\cdots\otimes e^{x_{t_{n}}{\rm ad}_{{\rm e}_{1}}}({\rm e}_{2})dy_{t_{1}}\cdots dy_{t_{n}}\big)\otimes e^{{\rm e}_{1}},
\end{equation}
where 
\[
e^{{\rm ad}_{\xi}}(\eta)\triangleq\sum_{m=0}^{\infty}\frac{({\rm ad}_{\xi})^{m}}{m!}(\eta).
\]
In particular, 
\[
\tilde{L}(\gamma)\triangleq\log\big(\sum_{n=0}^{\infty}\int_{0<t_{1}<\cdots<t_{n}<1}e^{x_{t_{1}}{\rm ad}_{{\rm e}_{1}}}({\rm e}_{2})\otimes\cdots\otimes e^{x_{t_{n}}{\rm ad}_{{\rm e}_{1}}}({\rm e}_{2})dy_{t_{1}}\cdots dy_{t_{n}}\big)
\]
is a Lie series.
\end{lem}

\begin{proof}
The signature path $t\mapsto S(\gamma)_{0,t}$ satisfies the differential
equation 
\[
dS(\gamma)_{0,t}=S(\gamma)_{0,t}\otimes({\rm e}_{1}dx_{t}+{\rm e}_{2}dy_{t}).
\]
Let $C_{t}\triangleq S(\gamma)_{0,t}\otimes e^{-x_{t}{\rm e}_{1}}.$
Then $C_{t}$ satisfies the differential equation 
\[
dC_{t}=C_{t}\otimes e^{x_{t}{\rm e}_{1}}\otimes{\rm e}_{2}\otimes e^{-x_{t}{\rm e}_{1}}dy_{t}.
\]
The result follows from the relation that 
\[
e^{x_{t}{\rm e}_{1}}\otimes{\rm e}_{2}\otimes e^{-x_{t}{\rm e}_{1}}=e^{x_{t}{\rm ad}_{{\rm e}_{1}}}({\rm e}_{2}).
\]
Since $\tilde{L}(\gamma)$ is the logarithmic signature of $\gamma\sqcup\overleftarrow{{\rm e}_{1}}$,
it is clear that $\tilde{L}(\gamma)$ is a Lie series.
\end{proof}

\subsubsection{Basic notions on Cartan's path development}\label{sec:Cartan}

Based on the representation (\ref{eq:SigAdRep}), a fundamental idea in our analysis is to perform  Cartan's path development in such a way that $\mathrm{e}_2$ is mapped into certain (combinations of) eigenspaces of $\mathrm{e}_1$. This will enable one to evaluate $e^{x_t\mathrm{ad}_{\mathrm{e}_1}}(\mathrm{e}_2)$ in an explicit way (within the new Lie algebra). Before introducing such constructions, we first recall  basic notions on Cartan's path development.

\begin{defn}
A \textit{(Cartan's) path deveopment} over $\mathbb{R}^2$ is a pair $(\mathfrak{g},F)$ where $\mathfrak{g}$ is a finite dimensional complex Lie algebra
and $F:\mathbb{R}^{2}\rightarrow\mathfrak{g}$
is a real linear map.
\end{defn}
We always assume that $\mathfrak{g}$ is embedded in the matrix algebra ${\cal M}={\rm Mat}(N,\mathbb{C})$ for some $N$. The space $\mathbb{R}^{2}$ is complexified
into $\mathbb{C}^{2}$ in the canonical way and so is $F.$ All algebraic
relations are understood over the complex field. Working over $\mathbb{C}$ rather than $\mathbb{R}$ is an important point in our analysis.

The linear map $F$ admits a unique extension
to an algebra homomorphism 
\begin{equation}\label{eq:HatF}
\hat{F}:T(\mathbb{C}^{2})\triangleq\bigoplus_{n=0}^{\infty}(\mathbb{C}^{2})^{\otimes n}\rightarrow{\cal M}.
\end{equation}
The restriction of $\hat{F}$ on Lie polynomials over $\mathbb{C}^{2}$
defines a Lie homomorphism, which is the unique
extension of $F$ to a Lie homomorphism from the free Lie algebra over
$\mathbb{C}^{2}$ to $\mathfrak{g}$. It should be noted that $\hat{F}$ is not always well-defined
on $T((\mathbb{C}^{2}))$ (formal tensor series) or on ${\cal L}((\mathbb{C}^{2}))$
(formal Lie series). The extension of $\hat{F}$ to these spaces requires
analytic consideration. In what follows, we always fix a matrix norm
on ${\cal M}.$
\begin{defn}
Let $\xi\in T((\mathbb{C}^{2}))$ be a given tensor series. We say
that $\hat{F}(\xi)$ \textit{exists} and write $\xi\in{\cal D}(\hat{F})$,
if $\underset{N\rightarrow\infty}{\lim}\hat{F}\big(\pi^{(N)}\xi\big)$
exists in ${\cal M}$.
\end{defn}

Let $\gamma$ be a weakly geometric rough path over $\mathbb{C}^{2}$. Since $S(\gamma)$
has infinite R.O.C., it is obvious that $S(\gamma)\in{\cal D}(\hat{F})$.
In addition, if $L(\gamma)\triangleq\log S(\gamma)$ has infinite
R.O.C., one also has $L(\gamma)\in{\cal D}(\hat{F})$. In general, $\hat{F}(L(\gamma))$
always exists when $F$ is ``sufficiently small''. 

\begin{cor}
\label{cor:ConvTilSig}Let $\gamma$ be a  weakly geometric rough path over $\mathbb{C}^2$. Then
there exists $\delta>0$ such that for any path development $(\mathfrak{g},F)$ with $\|F\|_{\mathbb{C}^{2}\rightarrow{\cal M}}<\delta,$
one has $L(\gamma)\in{\cal D}(\hat{F})$. In particular, under the previous
setting $\tilde{L}(\gamma)\in{\cal D}(\hat{F})$  provided that the operator
norm of $F$ is small enough (recall that $\tilde{L}(\gamma)$ is the logarithmic signature
of $\gamma\sqcup\overleftarrow{{\rm e}_{1}}$).
\end{cor}

\begin{proof}
According to \cite{CL16}, it is enough to choose $\delta=\big(2\|\gamma\|_{p\text{-var}}\big)^{-1}$ where $p$ is the roughness of $\gamma$.
\end{proof}

\begin{rem}
The use of path development in the literature, which also justifies its name, has a geometric nature. Given a pair $(\mathfrak{g},F)$, there is a canonical way of ``lifting'' any Euclidean path $\gamma$ in $\mathbb{R}^2$ to a corresponding path $\Gamma$ taking values in a Lie group whose Lie algebra is $\mathfrak{g}$. It turns out that one can study geometric and signature-related properties of the original path $\gamma$ from the lifted path $\Gamma$. The viewpoint we shall take in this work is however quite different. We work with path developments at the Lie algebra level instead of considering the actual lifted path in the group. Our technique thus has an  algebraic flavour in contrast to the geometric viewpoints taken in earlier works. 
\end{rem}

%
\begin{comment}
\begin{rem}
In several other situations (e.g. [...]), path developments are applied at the path-level by considering the lifting of the Euclidean path onto a Lie group $G$ whose Lie algebra is $\mathfrak{g}$. Such a construction defines a group homomorphism from the group of signatures to $G$ which is precisely the restriction of the above $\hat{F}$ on signatures. In addition, one has the following commutative diagram (at least at a formal level): ...  Our methodology is Lie-algebraic, in the sense that we consider path-developments of the logarithmic signature instead of path-liftings onto the Lie group. Such kind of considerations was also used in the work [...]. 

\end{rem}
\end{comment}

\subsubsection{A two-dimensional development}
 
To prove Theorem \ref{thm:LineInt}, we are going to choose a specific path development. Namely, we take
$\mathfrak{g}$ to be the two-dimensional complex Lie
algebra generated by (two symbols) $A,D$ under the Lie structure $[A,D]=D$. For each $\lambda\in \mathbb{C}$, we define the
development map $F_\lambda:\mathbb{C}^2\rightarrow\mathfrak{g}$ to be the linear map induced by
\[
F_{\lambda}({\rm e}_{1})\triangleq\lambda A,\ F_{\lambda}({\rm e}_{2})=D.
\]

\subsubsection*{A formal calculation}

We first show at a formal level how the choice of such a development enables one to prove Theorem \ref{thm:LineInt}. Recall that $\hat{F}_\lambda$ is the induced algebra homomorphism on tensors and also Lie homomorphism on Lie elements. By applying $\hat{F}_{\lambda}$
on both sides of (\ref{eq:HausRelationTensor}) one finds that 
\begin{equation}\label{eq:FormalLineIntPf}
e^{\hat{F}_{\lambda}(L(\gamma))}=e^{B(\hat{F}_{\lambda}(\tilde{L}(\gamma)),\lambda A)}\stackrel{\text{formally}}{\implies}\hat{F}_{\lambda}(L(\gamma))=B(\hat{F}_{\lambda}(\tilde{L}(\gamma)),\lambda A).
\end{equation}
According to Lemma \ref{lem:AdDecomp} and the construction of $F_\lambda$, one has 
\begin{align*}
e^{\hat{F}_{\lambda}(\tilde{L}(\gamma))}=\hat{F}_{\lambda}\big(e^{\tilde{L}(\gamma)}\big) & =\sum_{n=0}^{\infty}\int_{0<t_{1}<\cdots<t_{n}<1}e^{x_{t_{1}}{\rm ad}_{\lambda A}}(D)\cdots e^{x_{t_{n}}{\rm ad}_{\lambda A}}(D)dy_{t_{1}}\cdots dy_{t_{n}}\\
 & =\sum_{n=0}^{\infty}\int_{0<t_{1}<\cdots<t_{n}<1}e^{\lambda(x_{t_{1}}+\cdots+x_{t_{n}})}D^{n}dy_{t_{1}}\cdots dy_{t_{n}}\\
 & =\sum_{n=0}^{\infty}\frac{D^{n}}{n!}\big(\int_{0}^{1}e^{\lambda x_{t}}dy_{t}\big)^{n}=\exp\big(D\int_{0}^{1}e^{\lambda x_{t}}dy_{t}\big).
\end{align*}
This formally implies that
\begin{equation}
\hat{F}_{\lambda}(\tilde{L}(\gamma))=D\int_{0}^{1}e^{\lambda x_{t}}dy_{t}.\label{eq:FTilLFormula}
\end{equation}
By substituting (\ref{eq:FTilLFormula}) into (\ref{eq:FormalLineIntPf}), one finds that 
\[
\hat{F}_{\lambda}(L(\gamma))=B\big(D\int_{0}^{1}e^{\lambda x_{t}}dy_{t},\lambda A\big).
\]

Next, recall from (\ref{eq:HausRep}) that $B=\sum_{n\geqslant1}H_{n}$. The crucial observation is that 
\[
H_{n}\big(D\int_{0}^{1}e^{\lambda x_{t}}dy_{t},\lambda A\big)=0\ \ \ \forall n\geqslant2
\]
and 
\[
H_{1}\big(D\int_{0}^{1}e^{\lambda x_{t}}dy_{t},\lambda A\big)=\big(\sum_{m=0}^{\infty}\frac{B_{m}}{m!}\lambda^{m}\big)\cdot\big(\int_{0}^{1}e^{\lambda x_{t}}dy_{t}\big)D,
\]
both being consequences of the relation $[A,D]=D.$ As a result, one arrives at the following relation:
\begin{equation}
\hat{F}_{\lambda}(L(\gamma))=\big(\sum_{m=0}^{\infty}\frac{B_{m}}{m!}\lambda^{m}\big)\cdot\big(\int_{0}^{1}e^{\lambda x_{t}}dy_{t}\big)D.\label{eq:FormalHaus}
\end{equation}

Now suppose that $L(\gamma)$ has infinite R.O.C. Then the left hand side of (\ref{eq:FormalHaus})
defines a $\mathfrak{g}$-valued entire function (as a function of $\lambda\in\mathbb{C}$). On the other hand, the power
series $\sum_{m=0}^{\infty}\frac{B_{m}}{m!}\lambda^{m}$ defines the
meromorphic function $\phi(\lambda)=\frac{\lambda}{e^{\lambda}-1}$ on $\mathbb{C}$ that has isolated singularities at $\lambda=2k\pi i$
($k\in\mathbb{Z}\backslash\{0\}$). This forces $\int_{0}^{1}e^{\lambda x_{t}}dy_{t}=0$ at these singularities $\lambda = 2k\pi i$. The conclusion of Theorem \ref{thm:LineInt} thus follows.

\subsubsection*{Proof of Theorem \ref{thm:LineInt}}

The above argument is only formal, since one cannot directly pass $\hat{F}$
into the BCH function $B$ without analytical considerations, and to justify this one needs to apply truncations and
then pass to the limit. This is only a technical matter which we now make precise. 

Let $\mathfrak{g}$ be the Lie algebra of $2\times2$ upper triangular matrices over $\mathbb{C}$ with zero trace. Using the previous notation, one can take 
\[
A=\left(\begin{array}{cc}
1/2 & 0\\
0 & -1/2
\end{array}\right),\ D=\left(\begin{array}{cc}
0 & 1\\
0 & 0
\end{array}\right)
\]so that $\mathfrak{g}=\mathrm{Span}\{A,D\}$ and $[A,D]=D$. We consider the path development defined by
\[
F_{\lambda,\mu}({\rm e}_{1})\triangleq\lambda A,\ F_{\lambda,\mu}({\rm e}_{2})\triangleq\mu D\ \ \ (\lambda,\mu\in\mathbb{C}).
\]
According to Corollary
\ref{cor:ConvTilSig}, there exists $\delta>0$ such that 
\[
\tilde{L}(\gamma)\in{\cal D}(\hat{F}_{\lambda,\mu})\ \ \ \forall\lambda,\mu:|\lambda|\vee|\mu|<\delta.
\]
In other words, the limit
\begin{equation}
\hat{F}_{\lambda,\mu}\big(\tilde{L}(\gamma)\big)=\lim_{N\rightarrow\infty}\hat{F}_{\lambda,\mu}\big(\pi^{(N)}\tilde{L}(\gamma)\big)\label{eq:ExistFTilL}
\end{equation}
exists for all such $\lambda,\mu$.

On the other hand, recall from the normalisation (\ref{eq:PathNorm}) that the $x$-increment of $\tilde{\gamma}\triangleq\gamma\sqcup\overleftarrow{{\rm e}_{1}}$
is zero. In particular, $\tilde{L}(\gamma)$ does not contain the ${\rm e}_{1}$-component.
After applying the development $\hat{F}_{\lambda,\mu},$ one can thus write
\begin{equation}
\hat{F}_{\lambda,\mu}\big(\pi^{(N)}\tilde{L}(\gamma)\big)=C_{N}(\lambda,\mu)D\label{eq:FPiNTilL}
\end{equation}
for some $C_{N}(\lambda,\mu)\in\mathbb{C}.$ The function $(\lambda,\mu)\mapsto C_N(\lambda,\mu)$ is clearly entire on $\mathbb{C}^2$ for every fixed $N$. The relation (\ref{eq:ExistFTilL})
then implies that $C_{N}(\lambda,\mu)$ is convergent as $N\rightarrow\infty$ for small $\lambda,\mu$, 
whose limit with no surprise should coincide with (\ref{eq:FTilLFormula}) if one were able to take $\mu=1$.
\begin{lem}
\label{lem:ComputeFTilL}One has 
\[
\hat{F}_{\lambda,\mu}\big(\tilde{L}(\gamma)\big)=\lim_{N\rightarrow\infty}C_N(\lambda,\mu)D=\mu\big(\int_{0}^{1}e^{\lambda x_{t}}dy_{t}\big)D,\ \ \ |\lambda|\vee|\mu|<\delta.
\]
\end{lem}

\begin{proof}
By applying $\hat{F}_{\lambda,\mu}$ to the signature of $\tilde{\gamma}$ and using Lemma \ref{lem:AdDecomp}, 
one obtains that
\[
\hat{F}_{\lambda,\mu}\big(e^{\tilde{L}(\gamma)}\big)=\exp\big(\mu\big(\int_{0}^{1}e^{\lambda x_{t}}dy_{t}\big)D\big),\ \ \ \forall\lambda,\mu\in\mathbb{C}.
\]Note that there is no convergence issue at the signature level due to its factorial decay. 
When $\lambda,\mu$ is small, one also has 
\[
\hat{F}_{\lambda,\mu}\big(e^{\tilde{L}(\gamma)}\big)=e^{\hat{F}_{\lambda,\mu}(\tilde{L}(\gamma))}.
\]
For each fixed $\lambda$ with $|\lambda|<\delta$, both of the functions 
\[
G(\mu)\triangleq\hat{F}_{\lambda,\mu}(\tilde{L}(\gamma)),\ H(\mu)\triangleq\mu\big(\int_{0}^{1}e^{\lambda x_{t}}dy_{t}\big)D
\]
are analytic in $\mu$ for $|\mu|<\delta.$ In addition, the relation
(\ref{eq:FPiNTilL}) shows that $G(\mu)\in{\rm Span}\{D\}$ and thus
$G(\mu)$ is commutative for different $\mu$'s. Since $G$ and $H$
coincide when $\mu=0,$ a standard power series argument shows that
\[
e^{G(\mu)}=e^{H(\mu)}\ \text{for }|\mu|<\delta\implies G(\mu)=H(\mu)\ \text{for }|\mu|<\delta.
\]
The result thus follows.
\end{proof}
On the other hand, by using Lemma \ref{lem:ComputeFTilL} one can also compute $C_{N}(\lambda,\mu)$
explicitly.
\begin{lem}
\label{lem:CNFormula}For any $\lambda,\mu\in \mathbb{C}$, one has 
\begin{equation}\label{eq:FPiNIdentity}
C_{N}(\lambda,\mu)=\mu\sum_{j=0}^{N-1}\big(\int_{0}^{1}\frac{x_{t}^{j}}{j!}dy_{t}\big)\lambda^{j}.
\end{equation}
\end{lem}

\begin{proof}

It suffices to establish the relation for small $\lambda,\mu$, since both sides of the identity define entire functions in $(\lambda,\mu)\in\mathbb{C}^2$. To this end, note that the logarithmic signature $\tilde{L}(\gamma)$ admits a unique decomposition \[
\tilde{L}(\gamma)=\sum_{j=0}^{\infty}c_j{\rm ad}_{{\rm e}_{1}}^{j}({\rm e}_{2})+L',
\]where $L'$ consists of those tensor components with at least two $\mathrm{e}_2$'s. After applying $\hat{F}_{\lambda,\mu}$, one finds that 
\[
C_{N}(\lambda,\mu)=\mu\sum_{j=0}^{N-1}c_{j}\lambda^{j}
\]
for all small $\lambda,\mu.$ According to Lemma \ref{lem:ComputeFTilL},
the $\lambda$-power series of $\hat{F}_{\lambda,\mu}(\tilde{L}(\gamma))$
(for fixed $\mu$) is given by 
\[
\mu D\sum_{j=0}^{\infty}\frac{\lambda^{j}}{j!}\int_{0}^{1}x_{t}^{j}dy_{t}.
\]
By comparing coefficients, one finds that 
\[
c_{j}=\frac{1}{j!}\int_{0}^{1}x_{t}^{j}dy_{t}
\]
and the result thus follows.
\end{proof}
%
\begin{comment}
\begin{rem}
Note that for each fixed $N,$ both functions 
\[
\hat{F}_{\lambda,\mu}\big(\pi^{(N)}\tilde{L}(\gamma)\big),\ \mu\sum_{j=0}^{N-1}\big(\int_{0}^{1}\frac{x_{t}^{j}}{j!}dy_{t}\big)\lambda^{j}
\]
are analytic for all $\lambda,\mu\in\mathbb{C}.$ As a consequence,
the relation 
\begin{equation}
\hat{F}_{\lambda,\mu}\big(\pi^{(N)}\tilde{L}(\gamma)\big)=C_{N}(\lambda,\mu)D=\mu D\sum_{j=0}^{N-1}\big(\int_{0}^{1}\frac{x_{t}^{j}}{j!}dy_{t}\big)\lambda^{j}\label{eq:FPiNIdentity}
\end{equation}
holds for all $\lambda,\mu\in\mathbb{C}$.
\end{rem}
\end{comment}

\begin{proof}[Proof of Theorem \ref{thm:LineInt}]Suppose that $L(\gamma)$
has infinite R.O.C. Then $L(\gamma)\in{\cal D}(\hat{F}_{\lambda,\mu})$
for all $\lambda,\mu\in\mathbb{C}$. It follows from the relation
(\ref{eq:HausRelationTensor}) that 
\[
\lim_{N\rightarrow\infty}\hat{F}_{\lambda,\mu}\big(\pi^{(N)}B(\tilde{L}(\gamma),{\rm e}_{1})\big)=\hat{F}_{\lambda,\mu}(L(\gamma))\in\mathfrak{g}\subseteq{\cal M}.
\]
According to the definition (\ref{eq:HausRep}) of $B,$ one has 
\[
\pi^{(N)}B(\tilde{L}(\gamma),{\rm e}_{1})=\pi^{(N)}\sum_{n=1}^{\infty}H_{n}(\tilde{L}(\gamma),{\rm e}_{1})=\sum_{n=1}^{N}\pi^{(N)}H_{n}(\pi^{(N)}\tilde{L}(\gamma),{\rm e}_{1}).
\]
Note that the right hand side only contains finitely many terms from the expansion of $H_{n}$.
Due to the relations (\ref{eq:FPiNTilL}) and $[A,D]=D$, it is easily
seen that
\[
H_{n}\big(\hat{F}_{\lambda,\mu}\big(\pi^{(N)}\tilde{L}(\gamma)\big),\lambda A\big)=0\ \ \ \forall n\geqslant2.
\]
As a result, according to the formulas (\ref{eq:H1Formula}) and (\ref{eq:FPiNIdentity})
one finds that
\[
\hat{F}_{\lambda,\mu}\big(\pi^{(N)}B(\tilde{L}(\gamma),{\rm e}_{1})\big)=D\sum_{m=0}^{N-1}\frac{B_{m}}{m!}\lambda^{m}C_{N-m}(\lambda,\mu)=\mu D\sum_{m=0}^{N-1}\frac{B_{m}}{m!}\lambda^{m}\sum_{j=0}^{N-m-1}\int_{0}^{1}\frac{x_{t}^{j}}{j!}dy_{t}.
\]
By taking $N\rightarrow\infty$, one arrives at the identity
\[
\hat{F}_{\lambda,\mu}(L(\gamma))=\mu D\cdot\lim_{N\rightarrow\infty}\sum_{m=0}^{N-1}\frac{B_{m}}{m!}\lambda^{m}\sum_{j=0}^{N-m-1}\frac{\lambda^{j}}{j!}\int_{0}^{1}x_{t}^{j}dy_{t}.
\]
Note that the above limit exists for all $\lambda,\mu\in\mathbb{C}$.

The next observation is that the expression
\[
\sum_{m=0}^{N-1}\frac{B_{m}}{m!}\lambda^{m}\sum_{j=0}^{N-m-1}\frac{\lambda^{j}}{j!}\int_{0}^{1}x_{t}^{j}dy_{t}=\sum_{r=0}^{N-1}\big(\sum_{l=0}^{r}\frac{B_{l}}{l!(m-l)!}\int_{0}^{1}x_{t}^{r-l}dy_{t}\big)\lambda^{r}
\]
is precisely the partial sum of the power series defined by expanding
the product
\[
\big(\sum_{m=0}^{\infty}\frac{B_{m}}{m!}\lambda^{m}\big)\big(\sum_{j=0}^{\infty}\frac{\lambda^{j}}{j!}\int_{0}^{1}x_{t}^{j}dy_{t}\big).
\]
For fixed $\mu\in\mathbb{C}$ and $|\lambda|<2\pi$, both of the above power series are absolutely convergent, yielding that
\[
\hat{F}_{\lambda,\mu}\big(L(\lambda)\big)=\mu D\cdot\big(\sum_{m=0}^{\infty}\frac{B_{m}}{m!}\lambda^{m}\big)\big(\sum_{j=0}^{\infty}\frac{\lambda^{j}}{j!}\int_{0}^{1}x_{t}^{j}dy_{t}\big)=\frac{\lambda\mu}{e^{\lambda}-1}D\times\int_{0}^{1}e^{\lambda x_{t}}dy_{t}.
\]
Therefore, 
\[
(e^{\lambda}-1)\hat{F}_{\lambda,\mu}\big(L(\gamma)\big)=\mu D\cdot\lambda\int_{0}^{1}e^{\lambda x_{t}}dy_{t}
\]
for fixed $\mu\neq0$ and all $\lambda$ with $|\lambda|<2\pi.$ Now observe
that both sides are entire functions in $\lambda$ (since $L(\gamma)$
has infinite R.O.C.). As a result, the same identity holds for all $\lambda\in\mathbb{C}$.
By taking $\lambda=2k\pi i$ with $k\in\mathbb{Z}\backslash\{0\},$
one concludes that 
\[
\int_{0}^{1}e^{2k\pi i\cdot x_{t}}dy_{t}=0,
\]
hence giving the desired integral property (\ref{eq:LineInt}).

\end{proof}

\section{A strengthened version of the LS conjecture}\label{sec:StrCj}

In this section, we prove a strengthened version of the LS conjecture: having infinite R.O.C. for the logarithmic signature \textit{over all sub-intervals of time} implies that the underlying path must live on a straight line.

\begin{thm}\label{thm:StrCj}Let $\gamma:[0,1]\rightarrow\mathbb{R}^d$ be a continuous path with bounded variation. Suppose that $\log S(\gamma)_{s,t}$ has infinite R.O.C. for all $[s,t]\subseteq[0,1]$. Then  $\gamma_{t}=\gamma_0 + f(t)\cdot\vec{v}$ for some real-valued function
$f:[0,1]\to\mathbb{R}$ and some fixed vector $\vec{v}\in\mathbb{R}^{d}$. 
\end{thm}

Our main strategy of proving Theorem \ref{thm:StrCj} can be summarised as follows. It is based on non-trivial applications of Theorem \ref{thm:GenLineInt} together with winding number considerations. 

\begin{enumerate}
\item We may assume $d=2$ and parametrise the path $\gamma$ at unit speed. Let $s$ be an arbitrary point at which $\gamma$ is differentiable and choose a time $T>s$ that is sufficiently close to $s$. We normalise (i.e. rescale and rotate) the path $\gamma|_{[s,T]}$ to obtain a new path $\beta$ which starts at $(0,0)$ and ends at $(1,0)$. Since $|\gamma'_s|=1$ and $T\approx s$, in the rescaled picture the new path $\beta$  stays inside a cone region $\{(x,y):x\in[0,1],|y|<cx\}$ for some small number $c$. 

\item For each $(x,y)\in\mathbb{R}^2$, let $w(x,y)$ denote the winding number of the path $\tilde{\beta}\triangleq\beta\sqcup\overleftarrow{{\rm e}_{1}}$ around the point $(x,y)$. Since the path $\beta$ satisfies the generalised integral condition (\ref{eq:GenLineInt}), one can show   that the function $(x,y)\mapsto w(x,y)$ depends only on $y$. As a result, $w(x,y)$ has to vanish identically (because for any $(x,y)$ inside the cone there is some $(x',y)$ at the same level which is outside the cone and the winding number of $\tilde{\beta}$ around that point is obviously zero). This further implies that $\beta$ is the line segment ${\rm e}_1$ \textit{in the weak sense} that the line integrals along $\beta$ and ${\rm e}_1$ are identical for any smooth one-form. Transferring back to the original path $\gamma$, one concludes that $\gamma|_{[s,T]}$ is a line segment in the above weak sense. Call this line $L$.

\item The path $\gamma_{[T,1]}$ has to be entirely contained in $L$ and we prove this by contradiction. If this is not true, there exists a time $t>T$ such that $\gamma_t\notin L$. By normalising the path $\gamma|_{[s,t]}$ to a new path $\zeta$ which starts from the origin and ends at $(1,0)$, the line $L$ is transformed to another line $L'$ which is not horizontal. Draw a tiny ball $B_1$ centered at the mid-point of $L'|_{[s,T]}$ and make a copy $B_2$ at the same vertical so that $B_2\cap L'|_{[s,T]}=\emptyset$. This is possible since $L'$ is not horizontal.

\item Choose an arbitrary one-form $\Psi$ supported on $B_1$ whose line integral along $L'|_{[s,T]}$ is non-zero. Modify its values on $B_2$ to obtain a new one-form $\Phi$ which satisfies the two conditions of Theorem \ref{thm:GenLineInt}. The conclusion of Point 2 and the construction of $\Phi$  implies that\[
\int_{s}^{t}\Phi(d\zeta_u)=\int_{s}^{T}\Psi(dL'_{u})\neq0.
\]This contradicts the conclusion of Theorem \ref{thm:GenLineInt} since the logarithmic signature of $\zeta$ has infinite R.O.C. Therefore, $\gamma|_{[T,1]}$ has to be contained in the line $L$. Since the differentiable point $s$ is  arbitrary and $T\approx s$, this implies that the entire path $\gamma$ lives on a single well-defined straight line. 

\end{enumerate}

The rest of this section is devetoed to the proof of Theorem \ref{thm:StrCj}.

\subsection{An application of the integral condition to winding number}

In this subsection, we  derive a simple application of the generalised line integral condition (\ref{eq:GenLineInt}) to the winding number. Such a property will play a key role in our proof of Theorem \ref{thm:StrCj}.

Let $\gamma:\left[0,1\right]\rightarrow\mathbb{C}$ be a continuous BV path such that
$\gamma_{0}=\gamma_{1}$.  Given $\left(x,y\right)\notin\gamma\left[0,1\right]$,
the \textit{winding number} of $\gamma$ around $\left(x,y\right)$ is defined by
\begin{equation}
\eta\left(\gamma,\left(x,y\right)\right)\triangleq\frac{1}{2\pi i}\int\frac{1}{\gamma_{t}-\left(x+yi\right)}d\gamma_{t}.\label{eq:WindingDefinition}
\end{equation}
The winding number of $\gamma=\left(\gamma^{1},\gamma^{2}\right):\left[0,1\right]\rightarrow\mathbb{R}^{2}$
is defined as the winding number of the path $t\mapsto\gamma_{t}^{1}+i\gamma_{t}^{2}$.
We will use the following properties of the winding number. 
\begin{enumerate}
\item If there is a simply-connected set $K$ such that $\gamma\left[0,1\right]\subseteq K$,
then $\eta\left(\gamma,\left(x,y\right)\right)=0$  for all $\left(x,y\right)\notin K$.
This is a consequence of Cauchy's theorem.

\item Let $R:\mathbb{R}^{2}\rightarrow\mathbb{R}^{2}$ be a rotation
 and $\lambda\in\mathbb{R}\backslash\left\{ 0\right\} $. Then
\[
\eta\left(\lambda R\left(\gamma\right),\left(x,y\right)\right)=\eta\left(\gamma,\lambda^{-1}R^{-1}\left(x,y\right)\right).
\]
This follows from the definition (\ref{eq:WindingDefinition}) of winding number. 

\item (Green's theorem for self-intersecting paths)
\begin{comment}
For a continuous, bounded variation path $\gamma=\left(x,y\right):\left[0,1\right]\rightarrow\mathbb{R}^{2}$
such that $\gamma_{0}=\gamma_{1}$, 
\end{comment}
For any smooth functions $f,g:\mathbb{R}^{2}\rightarrow\mathbb{R}$, one has
\[
\int_{\mathbb{R}^{2}}\left(\partial_{x}f\left(x,y\right)+\partial_{y}g\left(x,y\right)\right)\eta\left(\gamma,\left(x,y\right)\right)dxdy=\int fdy_{s}-gdx_{s}.
\]
In addition, the function
\[
(x,y)\rightarrow \eta\left(\gamma,\left(x,y\right)\right).
\]
is square integrable with respect to the Lebesgue measure on $\mathbb{R}^2$. This fact can be found in \cite[Theorem 15]{BNQ14} and references therein. 
\end{enumerate}

Let $\gamma:[0,1]\rightarrow\mathbb{R}^2$ be a continuous BV path such
that $\gamma_{0}=(0,0)$ and $\gamma_{1}=(1,0)$. 
Recall that $\tilde{\gamma}\triangleq\gamma\sqcup\overleftarrow{{\rm e}_{1}}$ and we parametrise it on $[0,1]$. More precisely, it is defined by
\[
\tilde{\gamma}_{t}\triangleq\begin{cases}
\gamma_{2t}, & t\in [0,1/2]\\
\gamma_{1}-\left(\gamma_{1}-\gamma_{0}\right)(2t-1), & t\in[1/2,1].
\end{cases}
\]
Suppose that the generalised line integral condition (\ref{eq:GenLineInt})
holds for $\gamma$. Then the same condition also holds for $\tilde{\gamma}$. This is due to the facts that (\ref{eq:GenLineInt}) 
holds for both $\gamma$ and $\beta\triangleq[t\mapsto\left(1-t,0\right)]$
and that 
\[
\int \Phi\left(d\tilde{\gamma}\right)=\int \Phi\left(d\gamma\right)+\int \Phi\left(d\beta\right).
\]
We now state the main application of (\ref{eq:GenLineInt}) to the winding number. 
\begin{lem}\label{lem:WindApp}
Let $\gamma$ be a continuous BV path such that $\gamma_{0}=\left(0,0\right)$,
$\gamma_{1}=\left(1,0\right)$ and 
\[
\gamma\left[0,1\right]\subseteq\left[0,1\right]\times\mathbb{R},
\]
with $\gamma_{t}\in\left\{ 1\right\} \times\mathbb{R}$ if and only
if $t=1$. Let $\eta$ be the winding number of $\tilde{\gamma}$
around the point $\left(x,y\right)$. Suppose that $\gamma$ satisfies the generalised line integral condition (\ref{eq:GenLineInt}). Then for $\left(x,y\right)\in\left(0,1\right)\times\mathbb{R}\backslash\tilde{\gamma}\left[0,1\right]$, one has
\[
\eta\left(\tilde{\gamma},\left(x,y\right)\right)=\int_{0}^{1}\eta\left(\tilde{\gamma},\left(v,y\right)\right)dv.
\]
\end{lem}

\begin{proof}
Let $g$ be a smooth function satisfying $g\left(x+1,y\right)=g\left(x,y\right)$ and define $f\left(x,y\right)\triangleq g\left(x,y\right)-\int_{0}^{1}g\left(t,y\right)dt$. Write $\gamma_t = (x_t,y_t)$.
According to the integral condition (\ref{eq:GenLineInt}), one has 
\[
\int_{0}^{1}g\left(x_{s},y_{s}\right)dx_{s}-\int_{0}^{1}\big(\int_{0}^{1}g(t,y_{s})dt\big)dx_{s}=0.
\]
It follows from Green's theorem for self-intersecting paths that 
\begin{equation}\label{eq:WindingPf1}
\int_{\mathbb{R}^{2}}\frac{\partial g}{\partial y}\left(x,y\right)\eta\left(\tilde{\gamma},\left(x,y\right)\right)dxdy=\int_{\mathbb{R}^{2}}\big(\int_{0}^{1}\frac{\partial g}{\partial y}\left(t,y\right)\mathrm{d}t\big)\eta\left(\tilde{\gamma},\left(x,y\right)\right)dxdy.
\end{equation}
Since $\tilde{\gamma}\left[0,1\right]\subseteq\left[0,1\right]\times\mathbb{R}$ and $\eta\left(\tilde{\gamma},\left(x,y\right)\right)=0$ for all $\left(x,y\right)\notin\left[0,1\right]\times\mathbb{R}$, the relation (\ref{eq:WindingPf1}) can be rewritten as 
\begin{align*}
\int_{\mathbb{R}}\int_{0}^{1}\frac{\partial g}{\partial y}\left(x,y\right)\eta\left(\tilde{\gamma},\left(x,y\right)\right)dxdy = \int_{\mathbb{R}}\int_{0}^{1}\frac{\partial g}{\partial y}\left(t,y\right)\big(\int_{0}^{1}\eta\left(\tilde{\gamma},\left(x,y\right)\right)\mathrm{d}x\big)dtdy,
\end{align*}
or equivalently,  
\begin{equation}\label{eq:WindingPf2}
\int_{\mathbb{R}}\int_{0}^{1}\frac{\partial g}{\partial y}\left(x,y\right)\big[\eta\left(\tilde{\gamma},\left(x,y\right)\right)-\int_{0}^{1}\eta\left(\tilde{\gamma},\left(t,y\right)\right)dt\big]dxdy=0.
\end{equation}
Let $h$ be a smooth function such that $h\left(x+1,y\right)= h\left(x,y\right).$
By applying the relation (\ref{eq:WindingPf2}) to the function $g\left(x,y\right)\triangleq\int_{0}^{y}h\left(x,t\right)dt$, one finds that 
\[
\int_{\mathbb{R}}\int_{0}^{1}h\left(x,y\right)\big[\eta\left(\tilde{\gamma},\left(x,y\right)\right)-\int_{0}^{1}\eta\left(\tilde{\gamma},\left(t,y\right)\right)dt\big]dxdy=0.
\]

Now fix any $\left(x_{0},y_{0}\right)\in\left(0,1\right)\times\mathbb{R}\backslash\tilde{\gamma}\left[0,1\right]$
and let $\left(h_{\varepsilon}\right)_{\varepsilon>0}$ be a standard mollifier. Define \[
\tilde{h}_{\varepsilon}(x,y)\triangleq h_{\varepsilon}((x,y)-(x_{0},y_{0}))
\]and modify it to be $1$-periodic in $x.$ Then one has 
\[
\int_{\mathbb{R}}\int_{0}^{1}h_{\varepsilon}\left(\left(x,y\right)-\left(x_{0},y_{0}\right)\right)\left(\eta\left(\tilde{\gamma},\left(x,y\right)\right)-\int_{0}^{1}\eta\left(\tilde{\gamma},\left(t,y\right)\right)dt\right)dxdy=0.
\]
By taking $\varepsilon\rightarrow0$, it follows that 
\[
\eta\left(\tilde{\gamma},\left(x_{0},y_{0}\right)\right)-\int_{0}^{1}\eta\left(\tilde{\gamma},\left(t,y_{0}\right)\right)dt=0.
\]The result follows since $(x_0,y_0)$ is arbitrary.
\end{proof}

\subsection{A zero winding lemma }

Due to the conditions in Theorem \ref{thm:GenLineInt}, it is
necessary to rescale and rotate a path $\gamma$ so that it satisfies $\gamma_{0}=\left(0,0\right)$
and $\gamma_{1}=\left(1,0\right)$. We first introduce such a normalising operation. 
\begin{defn}
\label{ofn} Let $\gamma:[0,T]\to\mathbb{R}^{2}$ be a BV path and let $s<t$ be two fixed times such that $\gamma_{s}\neq\gamma_{t}$. We define
the associated
\textit{normalisation operator} $\mathcal{A}_{s,t}$ by 
\[
\mathcal{A}_{s,t}(x)\triangleq\frac{1}{|\gamma_{t}-\gamma_{s}|}R_{s,t}\left(x-\gamma_{s}\right),\ \ \ x\in\mathbb{R}^{2}.
\]
Here $R_{s,t}:\mathbb{R}^{2}\rightarrow\mathbb{R}^{2}$ denotes the rotation
that maps $\gamma_{t}-\gamma_{s}$ to $\left(\left|\gamma_{t}-\gamma_{s}\right|,0\right)$. We use $\mathcal{A}_{s,t}\gamma$ to denote the path $[s,t]\ni u\mapsto{\cal A}_{s,t}(\gamma_{u})$. 

\end{defn}

In this subsection, we will prove the following key result. As usual, we use $\tilde{{\cal A}}_{s,t}\gamma$ to denote the path ${\cal A}_{s,t}\gamma\sqcup\overleftarrow{{\rm e_{1}}}$ (we also assume that it is parametrised on $[s,t]$). Recall that $\eta(\tilde{{\cal A}}_{s,t}\gamma,(x,y))$ is the winding number of $\tilde{{\cal A}}_{s,t}\gamma$ around the point $(x,y)$.
\begin{proposition}
\label{thm:WindingZero}Let $\gamma:\left[0,K\right]\rightarrow\mathbb{R}^{2}$
be a continuous BV path which is parametrised at unit speed. Suppose that $\log S(\gamma)_{u,v}$ has infinite R.O.C. on all time intervals $[u,v]\subseteq[0,K]$. 
Let $s$ be a given fixed time at which $\gamma$ is differentiable and let  $\varepsilon\in(0,1/6]$. Let $T\in(s,K]$ be chosen such that 
\begin{align}
\left|\gamma_{t}-\gamma_{s}-\gamma_{s}^{\prime}\left(t-s\right)\right| & \leqslant\varepsilon\left(t-s\right)\label{eq:Differentiable},\\
\Big|\Big|\frac{\gamma_{t}-\gamma_{s}}{t-s}\Big|-1\Big| & \leqslant\varepsilon\nonumber 
\end{align}for all $t\in (s,T]$ 
and that
$\left|\gamma_{t}-\gamma_{s}\right|<\left|\gamma_{T}-\gamma_{s}\right|$ for all $s\in(t,T)$.
Then one has
\[
\eta\big(\tilde{{\cal A}}_{s,T}\gamma,(x,y)\big)=0
\]
for all $(x,y)\in(0,1)\times\mathbb{R}\backslash(\tilde{{\cal A}}_{s,T}\gamma)[s,T].$ 

\end{proposition}

The proof of Prposition \ref{thm:WindingZero} relies on the following key observation which asserts that the path $\mathcal{A}_{s,T}\gamma$
lies inside a cone region.

\begin{lem}
\label{lem:Cone}Let $\gamma,s,T,\varepsilon$ be as in Proposition \ref{thm:WindingZero}.
Then
\[
\left(\mathcal{A}_{s,T}\gamma\right)\left[s,T\right]\subseteq\Big\{ \left(x,y\right):\left|y\right|<\frac{2\varepsilon}{1-3\varepsilon}x\Big\} .
\]
\end{lem}

\begin{proof}
By using (\ref{eq:Differentiable}), one has 
\[
\left|\frac{1}{\left|\gamma_{T}-\gamma_{s}\right|}R_{s,T}\left(\gamma_{t}-\gamma_{s}\right)-\frac{R_{s,T}\gamma_{s}^{\prime}}{\left|\gamma_{T}-\gamma_{s}\right|}\left(t-s\right)\right|\leqslant\varepsilon\frac{\left(t-s\right)}{\left|\gamma_{T}-\gamma_{s}\right|}
\]
and hence 
\[
\left|\mathcal{A}_{s,T}\left(\gamma_{t}\right)-\mathcal{A}_{s,T}(\gamma_{s}^{\prime})\left(t-s\right)\right|\leqslant\varepsilon\frac{\left(t-s\right)}{\left|\gamma_{T}-\gamma_{s}\right|}.
\]
In particular, by taking $t=T$ and dividing the equation by $T-s$, one obtains that
\[
\left|\frac{1}{T-s}\left(1,0\right)-\mathcal{A}_{s,T}(\gamma_{s}^{\prime})\right|\leqslant\varepsilon\frac{1}{\left|\gamma_{T}-\gamma_{s}\right|}.
\]
Therefore, 
\begin{align*}
 & \Big|\left(\mathcal{A}_{s,T}\gamma\right)_{t}-\left(t-s\right)\Big(\frac{1}{T-s},0\Big)\Big|\\
 & \leqslant\left|\left(\mathcal{A}_{s,T}\gamma\right)_{t}-\mathcal{A}_{s,T}(\gamma_{s}^{\prime})\left(t-s\right)\right|+\Big|\Big(\frac{1}{T-s}\left(1,0\right)-\left(\mathcal{A}_{s,T}\gamma\right)_{s}^{\prime}\Big)\left(t-s\right)\Big|\\
 & \leqslant\frac{2\varepsilon\left(t-s\right)}{\left(1-\varepsilon\right)\left(T-s\right)}.
\end{align*}
Let us use $x^{1}$ to denote the first coordinate of $x\in\mathbb{R}^{2}$
and similarly for $x^{2}$. Then one has
\begin{align*}
\Big|\left(\mathcal{A}_{s,T}\gamma\right)_{t}^{1}-\frac{t-s}{T-s}\Big| & \leqslant\frac{2\varepsilon\left(t-s\right)}{\left(1-\varepsilon\right)\left(T-s\right)},\\
\big|\left(\mathcal{A}_{s,T}\gamma\right)_{t}^{2}\big| & \leqslant\frac{2\varepsilon\left(t-s\right)}{\left(1-\varepsilon\right)\left(T-s\right)}.
\end{align*}
Note that for $\varepsilon\in(0,1/6]$, the first inequality
implies that
\[
t-s\leqslant\frac{\left(\mathcal{A}_{s,T}\gamma\right)_{t}^{1}}{1-\frac{2\varepsilon}{1-\varepsilon}}(T-s)
\]
and therefore 
\[
\left|\left(\mathcal{A}_{s,T}\gamma\right)_{t}^{2}\right|\leqslant\frac{2\varepsilon}{1-3\varepsilon}\left(\mathcal{A}_{s,T}\gamma\right)_{t}^{1}.
\]This shows that $\mathcal{A}_{s,T}\gamma$ lies in the desired cone region.
\end{proof}
\begin{proof}[Proof of Proposition \ref{thm:WindingZero}] By Lemma \ref{lem:Cone}, one has
\[
\big(\tilde{\mathcal{A}}_{s,T}\gamma\big)[s,T]\subseteq\Big\{ \left(x,y\right):|y|<\frac{2\varepsilon}{1-3\varepsilon}x\Big\} =: \mathcal{C} .
\]
As a consequence of Property 1 for the winding number,  
\begin{equation}\label{eq:ZeroWind}
\eta\big(\tilde{{\cal A}}_{s,T}\gamma,(x,y)\big)=0
\end{equation}for all $(x,y)\notin\mathcal{C}$. Now suppose that \[
(x,y)\in{\cal C}\bigcap\big((0,1)\times\mathbb{R}\backslash(\tilde{{\cal A}}_{s,T}\gamma)[s,T]\big).
\]There exists some $x'\in(0,x)$ such that $(x',y)\notin\mathcal{C}$. According to Lemma \ref{lem:WindApp}, one concludes that \[
\eta\big(\tilde{{\cal A}}_{s,T}\gamma,(x,y)\big)=\eta\big(\tilde{{\cal A}}_{s,T}\gamma,(x',y)\big)=0.
\]
The result of the proposition thus follows. 
\begin{comment}
such that $|y|\geqslant\frac{2\varepsilon}{1-3\varepsilon}x$. In particular, (\ref{eq:ZeroWind}) holds for any $(x,y)$ with $x\in(0,\rho)$ and $|y|\geqslant\frac{2\varepsilon\rho}{1-3\varepsilon}$, where $\rho>0$ is any given fixed number.

Since for $(x,y)\in\left(\left(0,1\right)\times\mathbb{R}\right)\backslash\tilde{\left(\mathcal{A}_{s,T}\gamma\right)}\left[s,T\right]$,
$\eta\left(\tilde{\left(\mathcal{A}_{s,T}\gamma\right)},\left(x,y\right)\right)$
is independent of $x$, we have that for all $x\in\left(0,1\right)$,
for all $\left|y\right|\geqslant\frac{2\varepsilon\rho}{1-3\varepsilon}$
and $(x,y)\notin\tilde{\left(\mathcal{A}_{s,T}\gamma\right)}\left[0,1\right]$,
\[
\eta\left(\tilde{\left(\mathcal{A}_{s,T}\gamma\right)},\left(x,y\right)\right)=0.
\]
As this holds for all $\rho>0$, we have that
\[
\eta\left(\tilde{\left(\mathcal{A}_{s,T}\gamma\right)},\left(x,y\right)\right)=0
\]
for all $\left(x,y\right)\in\left(0,1\right)\times\mathbb{R}\backslash\left(\tilde{\left(\mathcal{A}_{s,T}\gamma\right)}\left[0,1\right]\cup\left\{ 0\right\} \right)$.
\end{comment}
\end{proof}

The following result is an immediate application of Proposition \ref{thm:WindingZero}. It shows that infinite R.O.C. for the logarithmic signature implies that the path contains a line segment in a weak sense.
\begin{lem}
\label{lem:AnalyticStraightLine}Let $\gamma:\left[0,K\right]\rightarrow\mathbb{R}^{2}$
be a continuous BV path with unit speed parametrisation. Let $s$ be a point at which
$\gamma$ is differentiable. Let $T\in(s,K]$ be chosen such that \[
\big|\gamma_{t}-\gamma_{s}-\gamma'_{s}(t-s)\big|\leqslant\frac{1}{6}(t-s),\ \Big|\Big|\frac{\gamma_{t}-\gamma_{s}}{t-s}\Big|-1\Big|\leqslant\frac{1}{6}
\]for all $t\in(s,T]$ and
\[
\left|\gamma_{t}-\gamma_{s}\right|<\left|\gamma_{T}-\gamma_{s}\right|
\]for all $t\in(s,T)$. Suppose that $\log S\left(\gamma\right)$ has infinite R.O.C. on
the interval $\left[s,T\right]$. Then the following two statements hold true.

\vspace{2mm}\noindent (i) One has
\[
\eta\left(\tilde{\gamma},\left(x,y\right)\right)=0
\]for all $\left(x,y\right)\in\left(0,1\right)\times\mathbb{R}\backslash\tilde{\gamma}\left[s,T\right]$. Here $\tilde{\gamma}\triangleq\gamma|_{[s,T]}\sqcup(\gamma_{s}-\gamma_{T}).$

\vspace{2mm}\noindent (ii) Let $L$ be the line segment joining $\gamma_{s}$ to $\gamma_{T}$ (also parametrised on $[s,T]$).
Then for any smooth one-form $\Phi(x,y)=f(x,y)dx+g(x,y)dy$, one has 
\[
\int_{s}^{T}\Phi\left(d\gamma_{v}\right)=\int_{s}^{T}\Phi\left(dL_{v}\right).
\]
\end{lem}

\begin{proof}
(i) This follows from Proposition \ref{thm:WindingZero}, $\tilde{\mathcal{A}}_{s,T}\gamma=\mathcal{A}_{s,T}\tilde{\gamma}$ (up to reparametrisation)
, that $\mathcal{A}_{s,T}$ is a composition of rotation and scaling,
as well as Property 2 for the winding number which describes how the winding number
behaves under rotation and scaling. 

\vspace{2mm}\noindent (ii) According to Green's Theorem for self-intersecting paths,
\[
\int \Phi\left(d\tilde{\gamma}_{v}\right)=\int_{\mathbb{R}^{2}}\big(\frac{\partial f}{\partial x}+\frac{\partial g}{\partial y}\big)\eta\left(\tilde{\gamma},\left(x,y\right)\right)dxdy.
\]
It follows from Part (i) that 
$\int \Phi\left(d\tilde{\gamma}_{v}\right)=0$, or equivalently, 
\[
\int_s^T \Phi\left(d\gamma_{v}\right)=\int_s^T \Phi\left(dL_{v}\right).
\]This proves the desired claim.
\end{proof}

\subsection{Proof of Theorem \ref{thm:StrCj}}

In this subsection, we develop the proof of Theorem \ref{thm:StrCj}. We first state a technical lemma. 
\begin{lem}
\label{lem:NormalisedDist}Let
$\gamma:\left[\tau_{1},\tau_{2}\right]\rightarrow\mathbb{R}^{2}$
be a continuous BV path. There exist  $\delta_1,\delta_2,r\in(0,1)$ such that the following properties hold true. Let $a<b<d$ be elements
of $\left[\tau_{1},\tau_{2}\right]$ such that $\gamma_{a}\neq\gamma_{b}$ 
\[
\left|\gamma_{b}-\gamma_{c}\right|<\delta_{1}\left|\gamma_{b}-\gamma_{a}\right|.
\]for all $c\in[b,d]$.

\vspace{2mm}\noindent (i) One has
\begin{equation}\label{eq:TechLem1}
\left|\mathcal{A}_{a,d}(\gamma_{b})-\mathcal{A}_{a,d}(\gamma_{c})\right|<\frac{\delta_{1}}{1-\delta_{1}},\ \left|\mathcal{A}_{a,d}(\gamma_{c})-\left(1,0\right)\right|<\frac{2\delta_{1}}{1-\delta_{1}}
\end{equation}for all $c\in[b,d]$;

\vspace{2mm}\noindent (ii) One has 
\begin{equation}\label{eq:TechLem21}
\mathcal{A}_{a,d}(m_{a,b})\subseteq\Big(\frac{1}{2}-\frac{\delta_{1}/2}{1-\delta_{1}},\frac{1}{2}+\frac{\delta_{1}/2}{1-\delta_{1}}\Big)\times\mathbb{R}
\end{equation}and 
\begin{equation}\label{eq:TechLem22}
\overline{B\big({\cal A}_{a,d}(m_{a,b}),r\big)}\bigcup\overline{B\big({\cal A}_{a,d}(m_{a,b})-(\delta_{2},0),r\big)}\subseteq(0,1)\times\mathbb{R},
\end{equation}where $m_{a,b}\triangleq(\gamma_a+\gamma_b)/2$ and $B(x,r)$ denotes the ball centered at $x$ with radius $r$;

\vspace{2mm}\noindent (iii) The four sets \[
{\cal A}_{a,d}\gamma|_{[b,d]},\ {\cal A}_{a,d}\gamma|_{[b,d]}-(1,0),\ \overline{B\big({\cal A}_{a,d}(m_{a,b}),r\big)},\ \overline{B\big({\cal A}_{a,d}(m_{a,b})-(\delta_{2},0),r\big)}
\]are disjoint.

\vspace{2mm}\noindent (iv) Let $L^{a,b}$ denote the straight line joining $\mathcal{A}_{a,d}(\gamma_a)$ and $\mathcal{A}_{a,d}(\gamma_b)$. Suppose that $L^{a,b}$ is not horizontal. Then
\begin{equation}\label{eq:TechLem4}
\overline{B\big({\cal A}_{a,d}(m_{a,b})-(\delta_{2},0),r\big)}\bigcap L^{a,b}=\emptyset.
\end{equation}
\begin{comment}
let $L^{a,b}$ denote the straight line joining the points $\mathcal{A}_{a,d}\gamma_{u}$
and $\mathcal{A}_{a,d}\gamma_{v}$. Then there exists $\delta_{1},\delta_{2},r>0$
such that
\begin{align*}
\mathcal{A}_{a,d}\gamma|_{\left[b,d\right]}, & \mathcal{A}_{a,d}\gamma|_{\left[b,d\right]}-\left(1,0\right)^{T}\\
\overline{B\left(\mathcal{A}_{a,d}\left[\frac{1}{2}\left(\gamma_{a}+\gamma_{b}\right)\right],r\right)} & ,\overline{B\left(\mathcal{A}_{a,d}\left[\frac{1}{2}\left(\gamma_{a}+\gamma_{b}\right)\right]-\left(\delta_{2},0\right)^{T},r\right)}
\end{align*}
are disjoint and
\[
\overline{B\left(\mathcal{A}_{a,d}\left[\frac{1}{2}\left(\gamma_{a}+\gamma_{b}\right)\right],r\right)},\overline{B\left(\mathcal{A}_{a,d}\left[\frac{1}{2}\left(\gamma_{a}+\gamma_{b}\right)\right]-\left(\delta_{2},0\right)^{T},r\right)}\subseteq\left(0,1\right)\times\mathbb{R}.
\]
and if $L^{a,b}$ is not horizontal,
\begin{equation}
\overline{B\left(\mathcal{A}_{a,d}\left[\frac{1}{2}\left(\gamma_{a}+\gamma_{b}\right)\right]-\left(\delta_{2},0\right)^{T},r\right)}\cap L^{a,b}=\emptyset.\label{eq:disjointLineBall-1-1}
\end{equation}
\end{comment}
\end{lem}
\begin{figure}[H]   
\begin{center}   
\includegraphics[scale=0.21]{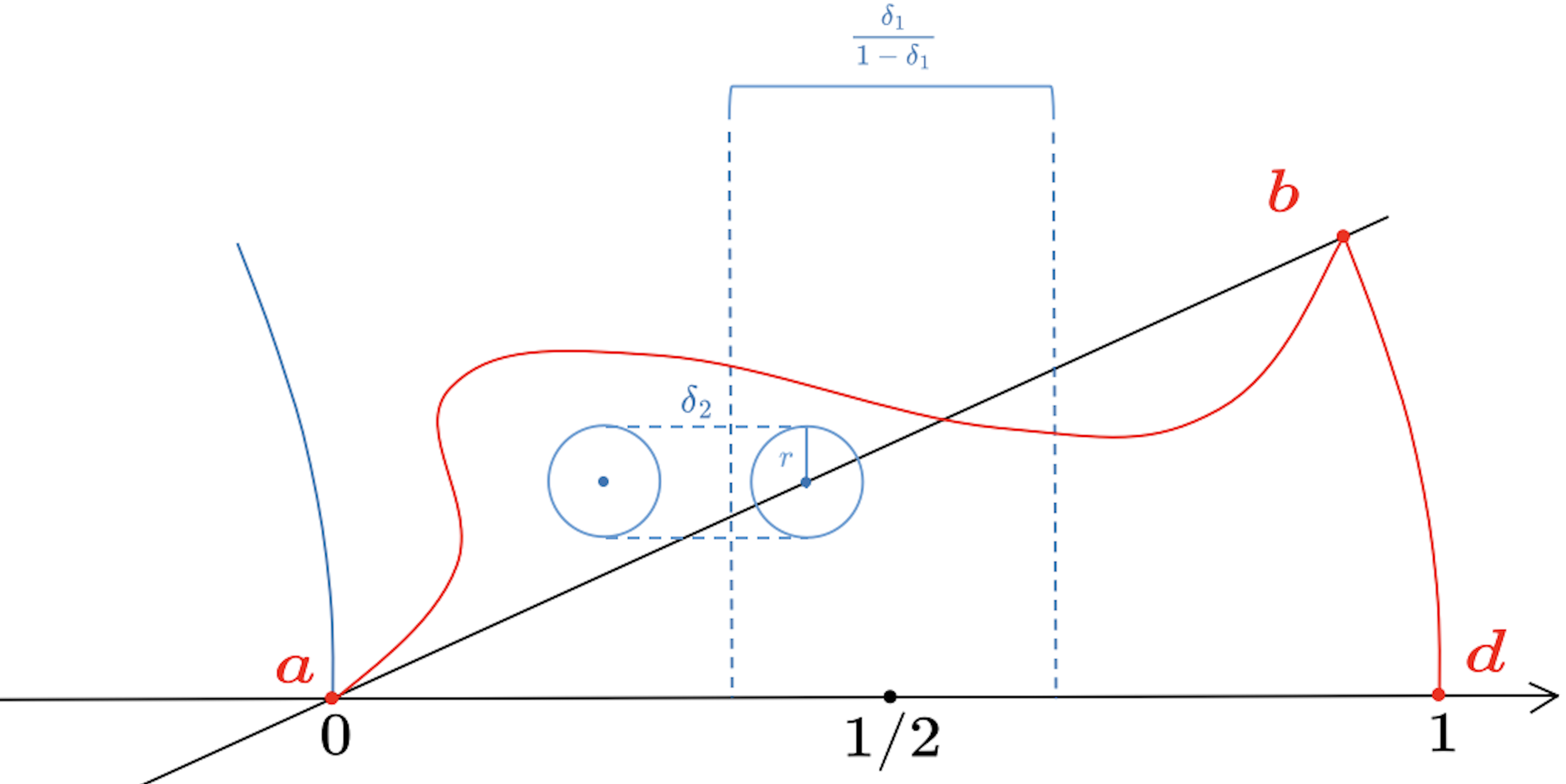}
%\caption{The path $\mathcal{A}_{a,d}\gamma$ is represented by the red curve.}
\end{center} 
\end{figure}
\begin{proof}
(i) Note that $\mathcal{A}_{a,d}\gamma$ is well-defined since \[
|\gamma_{d}-\gamma_{a}|\geqslant|\gamma_{b}-\gamma_{a}|-|\gamma_{b}-\gamma_{d}|\geqslant(1-\delta_{1})|\gamma_{b}-\gamma_{a}|>0.
\]Now for all $c\in[b,d]$, one has
\begin{align*}
\left|\mathcal{A}_{a,d}(\gamma_{b})-\mathcal{A}_{a,d}(\gamma_{c})\right|  =\frac{\left|\gamma_{b}-\gamma_{c}\right|}{\left|\gamma_{d}-\gamma_{a}\right|}
  <\frac{\delta_{1}\left|\gamma_{b}-\gamma_{a}\right|}{\left|\gamma_{b}-\gamma_{a}\right|-\left|\gamma_{d}-\gamma_{b}\right|}
  \leqslant\frac{\delta_{1}}{1-\delta_{1}}
\end{align*}
and
\begin{align*}
\left|\mathcal{A}_{a,d}(\gamma_{c})-\left(1,0\right)\right| & =\left|\mathcal{A}_{a,d}(\gamma_{c})-\mathcal{A}_{a,d}(\gamma_{d})\right|\\
 & \leqslant\left|\mathcal{A}_{a,d}(\gamma_{b})-\mathcal{A}_{a,d}(\gamma_{c})\right|+\left|\mathcal{A}_{a,d}(\gamma_{b})-\mathcal{A}_{a,d}(\gamma_{d})\right|
  <\frac{2\delta_{1}}{1-\delta_{1}}.
\end{align*} 
The desired estimate (\ref{eq:TechLem1}) thus follows.

\vspace{2mm}\noindent (ii) The relation (\ref{eq:TechLem21}) follows since 
\begin{align*}
\mathcal{A}_{a,d}(m_{a,b})= & \frac{1}{2}\big(\mathcal{A}_{a,d}(\gamma_{a})+\mathcal{A}_{a,d}(\gamma_{b})\big)\\
= & \frac{1}{2}\big(\mathcal{A}_{a,d}(\gamma_{b})-\mathcal{A}_{a,d}(\gamma_{d})\big)+\frac{1}{2}\mathcal{A}_{a,d}(\gamma_{d})\\
= & \frac{1}{2}\big(\mathcal{A}_{a,d}(\gamma_{b})-\mathcal{A}_{a,d}(\gamma_{d})\big)+\frac{1}{2}\left(1,0\right).
\end{align*}
and one also knows from (\ref{eq:TechLem1}) that 
\[
\big|\mathcal{A}_{a,d}(\gamma_{b})-\mathcal{A}_{a,d}(\gamma_{d})\big|\leqslant\frac{\delta_{1}}{1-\delta_{1}}.
\]
As a result of (\ref{eq:TechLem21}), one can obviously make (\ref{eq:TechLem22}) valid by choosing $\delta_1,\delta_2,r$ to be small enough.

\vspace{2mm}\noindent (iii) According to (i),  
\begin{align*}
\mathcal{A}_{a,d}\gamma|_{\left[b,d\right]} & \subseteq B\Big(\left(1,0\right),\frac{{\color{black}2}\delta_{1}}{1-\delta_{1}}\Big),\\
\mathcal{A}_{a,d}\gamma|_{\left[b,d\right]}-\left(1,0\right) & \subseteq B\Big(\left(0,0\right),\frac{{\color{black}2}\delta_{1}}{1-\delta_{1}}\Big),
\end{align*}
and by (ii),
\[
\mathcal{A}_{a,d}(m_{a,b})\subseteq\Big(\frac{1}{2}-\frac{\delta_{1}/2}{1-\delta_{1}},\frac{1}{2}+\frac{\delta_{1}/2}{1-\delta_{1}}\Big)\times\mathbb{R}.
\]
The  claim thus follows if one  takes $\delta_{1},\delta_{2},r$ to be sufficiently small so that the three sets
\[
\Big(\frac{1}{2}-\frac{\delta_{1}/2}{1-\delta_{1}},\frac{1}{2}+\frac{\delta_{1}/2}{1-\delta_{1}}\Big)\times\mathbb{R},\ B\Big(\left(1,0\right),\frac{{\color{black}2}\delta_{1}}{1-\delta_{1}}\Big),\ B\Big(\left(0,0\right),\frac{{\color{black}2}\delta_{1}}{1-\delta_{1}}\Big)
\]
are disjoint.

\vspace{2mm}\noindent (iv) Since $L^{a,b}$ is not horizontal and $\mathcal{A}_{a,d}(m_{a,b})\in L^{a,b}$, one knows that
\[
\mathcal{A}_{a,d}(m_{a,b})-\left(\delta_{2},0\right)\notin L^{a,b}.
\]
By further reducing $r$ if necessary, one can  ensure that the relation (\ref{eq:TechLem4}) holds.
\end{proof}

Next, by using Lemma \ref{lem:NormalisedDist} we prove a key lemma that is needed for the later proof of Theorem \ref{thm:StrCj}. 

\begin{lem}
\label{lem:KeyDimensionTwoLemma}Let $\gamma:\left[0,K\right]\rightarrow\mathbb{R}^{2}$
be a continuous BV path with unit speed parametrisation. Suppose that $\log S\left(\gamma\right)$ has infinite R.O.C. on all
sub-intervals $\left[u,v\right]\subseteq[0,K]$. Let $s,T$ be chosen fixed as in Lemma \ref{lem:AnalyticStraightLine}.
\begin{comment}
Let $s$ be a point at which
$\gamma$ is differentiable and let $T$ be such that for all $t$
such that $s<t\leqslant T$, 

\begin{align*}
\left|\gamma_{t}-\gamma_{s}-\gamma_{s}^{\prime}\left(t-s\right)\right| & \leqslant\frac{1}{6}\left(t-s\right)\\
\left|\left|\frac{\gamma_{t}-\gamma_{s}}{t-s}\right|-1\right| & \leqslant\frac{1}{6}.
\end{align*}
and for all $t$ such that $s<t<T$,
\[
\left|\gamma_{t}-\gamma_{s}\right|<\left|\gamma_{T}-\gamma_{s}\right|.
\]
\end{comment}
 Then one has
\[
\gamma\left[T,K\right]\subseteq\left\{ \gamma_{s}+\lambda(\gamma_{T}-\gamma_{s}):\lambda\in\mathbb{R}\right\} .
\]
\end{lem}

\begin{proof}
Let us denote the straight line $\{\gamma_s+\lambda(\gamma_T-\gamma_s):\lambda\in\mathbb{R}\}$ by $l$. Define
\[
T^{*}=\sup\left\{ t\in\left[T,K\right]:\gamma\left[T,t\right]\subseteq l \right\} .
\]
Suppose on the contrary that $T^*<K$. We will reach a contradiction by considering the following two cases.

\vspace{2mm}\noindent \underline{\textbf{Case 1:} $\gamma_{T^{*}}\neq\gamma_{s}$.}

\vspace{2mm}There exists $t>T^{*}$ such that $\gamma_{t}\notin l$
and $t$ is chosen to be close enough to $T^{*}$ so that
\[
\left|\gamma_{u}-\gamma_{T^{*}}\right|<\delta_{1}\left|\gamma_{T^{*}}-\gamma_{s}\right|
\]for all $u\in[T^*,t]$ (with $\delta_{1}$ to be specified later). 
By choosing $a=s,b=T^{*},d=t$ in Lemma \ref{lem:NormalisedDist},
one can find $\delta_{1},\delta_{2},r>0$ such that the four closed
sets
\[
{\cal A}_{s,t}\gamma|_{[T^{*},t]},\ \ {\cal A}_{s,t}\gamma|_{[T^{*},t]}-(1,0),\ \ \overline{B\big({\cal A}_{s,t}(m_{s,T^{*}}),r\big)},\ \ \overline{B\big({\cal A}_{s,t}(m_{s,T^{*}})-(\delta_{2},0),r\big)}
\]
are all disjoint, where $m_{s,T^*}\triangleq (\gamma_s+\gamma_{T^*})/2$. The same lemma shows that 
\[
\overline{B\big({\cal A}_{s,t}(m_{s,T^{*}}),r\big)}\bigcup\overline{B\big({\cal A}_{s,t}(m_{s,T^{*}})-(\delta_{2},0),r\big)}\subseteq(0,1)\times\mathbb{R},
\]
and since $\gamma_{t}\notin l$, $L^{s,T^{*}}$ cannot be horizontal so one also knows that 
\begin{equation}
\overline{B\big({\cal A}_{s,t}(m_{s,T^{*}})-(\delta_{2},0),r\big)}\bigcap L^{s,T^{*}}=\emptyset.
\end{equation}
Here $L^{s,T^*}$ denotes the line segment joining the origin to $\mathcal{A}_{s,t}(\gamma_{T^*})$. In what follows, we assume that $L^{s,T^*}$ is parametrised on $[s,T^*]$.

Let $\Psi$ be a smooth one-form supported
on $B(\mathcal{A}_{s,t}(m_{s,T^*}),r)$ 
such that 
\begin{equation}\label{eq:IntPsi}
\int_{s}^{T^{*}}\Psi\big(dL_{v}^{s,T^{*}}\big)=1.
\end{equation}Such a $\Psi$ clearly exists. 
Let us define
\[
\Phi\left(x,y\right)=\begin{cases}
\Psi\left(x,y\right), & \left(x,y\right)\in B\left(\mathcal{A}_{s,t}(m_{s,T^*}),r\right):=B_{1}\\
-\Psi\left(x+\delta_{2},y\right), & \left(x,y\right)\in B\big(\mathcal{A}_{s,t}(m_{s,T^*})-\left(\delta_{2},0\right),r\big):=B_{2},\\
0, & (x,y)\in\left(0,1\right)\times\mathbb{R}\backslash\left(B_{1}\cup B_{2}\right)\\
\Phi\left(x\  \mathrm{mod} 1,y\right), & \left(x,y\right)\notin\left(0,1\right)\times\mathbb{R}\backslash\left(B_{1}\cup B_{2}\right).
\end{cases}
\]Note that the smooth one-form $\Phi$ satisfies the two conditions in Theorem \ref{thm:GenLineInt}. 
One can now write
\begin{align*}
\int_{s}^{t}\Phi\left(d\left[\mathcal{A}_{s,t}\gamma\right]_v\right) & =\int_{s}^{T}\Phi\left(d\left[\mathcal{A}_{s,t}\gamma\right]_v\right)+\int_{T}^{T^{*}}\Phi\left(d\left[\mathcal{A}_{s,t}\gamma\right]_v\right)+\int_{T^{*}}^{t}\Phi\left(d\left[\mathcal{A}_{s,t}\gamma\right]_v\right).
\end{align*}
According to  Lemma \ref{lem:AnalyticStraightLine}, one has
\[
\int_{s}^{T}\Phi\left(d[\mathcal{A}_{s,t}\gamma]_{v}\right)=\int_{s}^{T}\Phi\big(dL_{v}^{s,T^*}\big).
\]
Since $\Phi=0$  on $\left(0,1\right)\times\mathbb{R}\backslash\left(B_{1}\cup B_{2}\right)$
and  $\mathcal{A}_{s,t}\gamma|_{\left[T^{*},t\right]}$ is disjoint
from $B_{1}\cup B_{2}$, one sees that 
\[
\int_{T^{*}}^{t}\Phi\left(d\left[\mathcal{A}_{s,t}\gamma\right]_v\right)=0.
\]
In addition, since $\gamma\left[T,T^{*}\right]\subseteq l$ one also knows that 
\[
\int_{T}^{T^{*}}\Phi\left(d\left[\mathcal{A}_{s,t}\gamma\right]_v\right)=\int \Phi\big(dL_{v}^{s,T^{*}}\big).
\]
Therefore, one concludes that 
\[
\int_{s}^{t}\Phi\left(d[\mathcal{A}_{s,t}\gamma]_{v}\right)=\int_s^{T^*} \Phi\big(dL_{v}^{s,T^{*}}\big).
\]
Note that $L^{s,T^{*}}$ is disjoint from $B_{2}$. By the construction of $\Phi$, the last integral is thus also equal to the same integral along $\Psi$ whose value is $1\neq0$ (cf. \ref{eq:IntPsi}).  
This leads to a contradiction to the generalised line integral condition (\ref{eq:GenLineInt}) in Theorem \ref{thm:GenLineInt}.
%\textbf{End of Case 1's proof}
\begin{figure}[H]   
\begin{center}   
\includegraphics[scale=0.21]{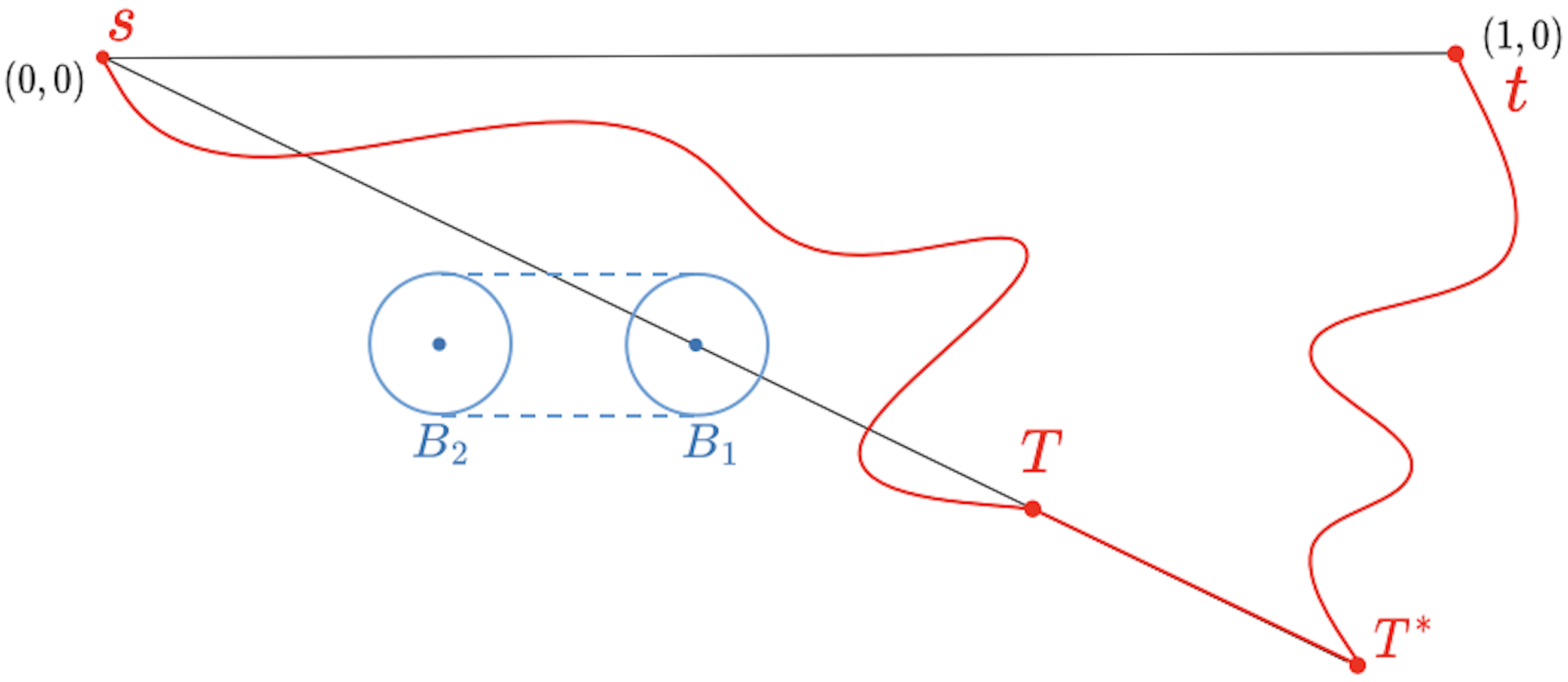}
\end{center} 
\end{figure}
\vspace{2mm}\noindent \underline{\textbf{Case 2: }$\gamma_{T^{*}}=\gamma_{s}$.}

\vspace{2mm} In this case, $\gamma_{T}\neq\gamma_{T^{*}}$ because $\gamma_{T}\neq\gamma_{s}$.
Then there exists $t>T^{*}$ such that $\gamma_{t}\notin l$
and $t$ is close enough to $T^{*}$ so that
\[
\left|\gamma_{u}-\gamma_{T^{*}}\right|<\delta_{1}\left|\gamma_{T^{*}}-\gamma_{T}\right|
\]for all $u\in[T^*,t]$ (with $\delta_{1}$ to be specified later). By choosing $a=T,b=T^{*},d=t$ in Lemma \ref{lem:NormalisedDist}, one can now proceed in exactly the same way as in Case 1 to construct a one-form $\Phi$ which satisfies the two conditions in Theorem \ref{thm:GenLineInt} but 
\[
\int_{T}^{t}\Phi(d[{\cal A}_{T,t}\gamma]_{v})=1\neq0.
\]This again contradicts the generalised integral condition (\ref{eq:GenLineInt}). 
\end{proof}

The proof of Theorem \ref{thm:StrCj} can be easily reduced to the two-dimensional situation which we first handle. 

\begin{thm}
\label{R2generalcase} Let $\gamma:[0,1]\to\mathbb{R}^{2}$
be a continuous BV path. Suppose that $\log S(\gamma)_{s,t}$ has infinite R.O.C. for all $s<t\in[0,1]$. Then  $\gamma_{t}=f(t)\cdot\vec{v}$ for some real-valued function
$f:[0,1]\to\mathbb{R}$ and some fixed vector $\vec{v}\in\mathbb{R}^{2}$. 
\end{thm}

\begin{proof}
By Appendix B in the arXiv version of \cite{BG23}, there exist a $K\in[0,\infty)$, a
non-decreasing function $L:\left[0,1\right]\rightarrow\left[0,K\right]$
and $\hat{\gamma}:\left[0,K\right]\rightarrow\mathbb{R}^{2}$ such
that $\hat{\gamma}_{L(\cdot)}=\gamma_\cdot$ and $\hat{\gamma}$ is Lipschitz continuous with unit speed. Since the signature is invariant under reparametrisation, one knows that
$\log S\left(\hat{\gamma}\right)_{s,t}$
 has infinite R.O.C. for all $s<t\in[0,K]$.
Let 
\[
\tau=\sup\left\{ t:\hat{\gamma}_{u}^{\prime}=0\;\text{for almost all }u\in\left[0,t\right]\right\} .
\]
Given $\varepsilon>0$, there exist a differentiable point $s\in[\tau,\tau+\varepsilon/2)$
and a sequence $\left(t_{n}\right)_{n=1}^{\infty}$ such that $t_{n}\downarrow s$
and $\hat{\gamma}_{s}\neq\hat{\gamma}_{t_{n}}$
for all $n$.
Let $\delta\in\left(0,\varepsilon/2\right)$ be chosen such
that \[
\big|\hat{\gamma}_{u}-\hat{\gamma}_{s}-\hat{\gamma}_{s}'(u-s)\big|\leqslant\frac{1}{6}(u-s),\ \Big|\Big|\frac{\hat{\gamma}_{u}-\hat{\gamma}_{s}}{u-s}\Big|-1\Big|\leqslant\frac{1}{6}
\]whenever $|u-s|<\delta$.
Take $N$ to be such that $\left|t_{N}-s\right|<\delta$ and let 
\[
T=\inf\left\{ u\geqslant s:\hat{\gamma}_{u}=\hat{\gamma}_{t_{N}}\right\} .
\]Clearly, the numbers $s,T$ satisfy the assumptions of Lemma \ref{lem:KeyDimensionTwoLemma}.
According to Lemma \ref{lem:KeyDimensionTwoLemma}, one knows that
\[
\hat{\gamma}\left[T,K\right]\subseteq\left\{ \hat{\gamma}_{s}+\lambda(\hat{\gamma}_{T}-\hat{\gamma}_{s}):\lambda\in\mathbb{R}\right\}=:l_\varepsilon 
\]
and in particular, $\hat{\gamma}[\tau+\varepsilon,K]\subseteq l_\varepsilon$. Since this is true for all $\varepsilon>0$, the straight lines $\{l_\varepsilon\}_{\varepsilon>0}$ are clearly consistent and thus all identical. Let us call this line $l$. It follows that $\hat{\gamma}[\tau,K]\subseteq l$.
Since $\hat{\gamma}_{t}^{\prime}=0$ for almost all $t\in\left[0,\tau\right]$,
one has $\hat{\gamma}_{t}=\hat{\gamma}_{\tau}$ for all $t\in\left[0,\tau\right]$. As a result, $\hat{\gamma}[0,K]\subseteq l$. This provides that $\hat{\gamma}$ (and thus $\gamma$) lives on a straight line.

\end{proof}

\begin{proof}[Proof of Theorem \ref{thm:StrCj}]
Suppose that $\gamma_{t}=(x_{t}^{1},x_{t}^{2},\cdots,x_{t}^{d})$ and assume that $\gamma_0 = \mathbf{0}$.
The conclusion is trivial if the image of $\gamma$ is a single point. Otherwise, let us  assume that $x_{t}^{1}$ is not identically
zero. For each $i\neq 1$, the logarithmic signature of the path  
$\gamma_{t}^{i}\triangleq(x_{t}^{1},x_{t}^{i})$
has infinite R.O.C. on $[s,t]$ for all $s<t\in(0,1)$. By Theorem \ref{R2generalcase}, one has $\gamma^i_t = f_i(t)(a_i,b_i)$ for some real-valued function $f_i(t)$ and some vector $(a_i,b_i)$ with $a\neq0$. It is clear that $x_t^1 = f_i(t)a_i$ and thus $\gamma_t^i = x_t^1 \cdot (1,c_i)$ where $c_i \triangleq b_i/a_i$. It follows that $\gamma_t = x_t^1 \cdot (1,c_2,\cdots,c_d)$. This completes the proof of the theorem. 
\end{proof}

\section{Second order integral identities}\label{sec:IterInt}

We continue to assume that $\gamma:[0,1]\rightarrow\mathbb{R}^2$ is a weakly geometric rough path satisfying the normalisation condition (\ref{eq:PathNorm}). In Section \ref{sec:LineInt}, we showed that if the logarithmic signature of $\gamma$ has infinite R.O.C., the path $\gamma$ must satisfy the line integral condition (\ref{eq:LineInt}). In this and the next sections, by using the method of Cartan's path development we will show that $\gamma$ has to satisfy \textit{an infinite system of iterated integral identities} where the relation (\ref{eq:LineInt}) appears to be the first level of them. 

To better illustrate the essential idea, it is helpful to first discuss the derivation of  \textit{second order} iterated integral identities. This is our main goal in the current section and the main result is stated as follows. 

\begin{thm}
\label{thm:DoubInt}Suppose that $\log S(\gamma)$  has infinite
R.O.C. Then the following identity
\begin{align}
 & (1-\cosh b)\int_{0<s<t<1}\big(e^{2k\pi i\cdot x_{s}+b(x_{t}-x_{s})}-e^{2k\pi i\cdot x_{t}+b(x_{s}-x_{t})}\big)dy_{s}dy_{t}\nonumber \\
 & \ \ \ +(\sinh b)\big(\int_{0}^{1}e^{bx_{s}}dy_{s}\big)\big(\int_{0}^{1}e^{(2k\pi i-b)x_{s}}dy_{s}\big)=0\label{eq:DoubInt}
\end{align}
holds true for all $k\in\mathbb{Z}\backslash\{0\}$ and $b\in\mathbb{C}$.
\end{thm}

\subsection{A second order extension of the line integral condition}

In contrast to the line integral condition (\ref{eq:LineInt}), Theorem
\ref{thm:DoubInt} yields a \textit{continuum} family of integral
identities from which one can extract various information. The following
result is a consequence of Theorem \ref{thm:DoubInt} which can be
viewed as a second order extension of (\ref{eq:LineInt}).
\begin{cor}\label{cor:DoubIntCond}
Under the assumption of Theorem \ref{thm:DoubInt}, one has
\begin{equation}
\int_{0<s<t<1}\big(e^{2\pi i(px_{s}+qx_{t})}-e^{2\pi i(px_{t}+qx_{s})}\big)dy_{s}dy_{t}=0\label{eq:DoubInt2}
\end{equation}
for all $p,q\in\mathbb{Z}\backslash\{0\}$ with $p+q\neq 0.$
\end{cor}
\begin{proof}
Let us denote 
\[
\varphi(b)\triangleq\int_{0<s<t<1}\big(e^{2k\pi i\cdot x_{s}+b(x_{t}-x_{s})}-e^{2k\pi i\cdot x_{t}+b(x_{s}-x_{t})}\big)dy_{s}dy_{t}
\]
and 
\[
\psi(b)\triangleq\big(\int_{0}^{1}e^{bx_{s}}dy_{s}\big)\big(\int_{0}^{1}e^{(2k\pi i-b)x_{s}}dy_{s}\big).
\]
Differentiating (\ref{eq:DoubInt}) with respect to $b$ gives that
\begin{equation}
-\sinh b\cdot\varphi(b)+(1-\cosh b)\varphi'(b)+\cosh b\cdot\psi(b)+\sinh b\cdot\psi'(b)=0.\label{eq:DoubInt1Der}
\end{equation}
By further differentiating (\ref{eq:DoubInt1Der}) and taking $b=2l\pi i$,
one obtains that
\[
-\varphi(2l\pi i)+2\psi'(2l\pi i)=0.
\]
Now suppose that $l\neq k$ and $l\neq0$. The line integral condition
(\ref{eq:LineInt}) yields that
\begin{align*}
\psi'(2l\pi i) & =\big(\int_{0}^{1}x_{s}e^{2l\pi i\cdot x_{s}}dy_{s}\big)\big(\int_{0}^{1}e^{2(k-l)\pi i\cdot x_{s}}dy_{s}\big)\\
 & \ \ \ -\big(\int_{0}^{1}e^{2l\pi i\cdot x_{s}}dy_{s}\big)\big(\int_{0}^{1}x_{s}e^{2(k-l)\pi i\cdot x_{s}}dy_{s}\big)=0.
\end{align*}
As a result, one concludes that 
\[
\varphi(2l\pi i)=\int_{0<s<t<1}\big(e^{2(k-l)\pi i\cdot x_{s}+2l\pi i\cdot x_{t}}-e^{2(k-l)\pi i\cdot x_{t}+2l\pi i\cdot x_{s}}\big)dy_{s}dy_{t}=0.
\]
This gives the desired identity (\ref{eq:DoubInt2}) with $p=k-l$
and $q=l$.
\end{proof}
\begin{rem}
The line integral identity (\ref{eq:LineInt}) is also a consequence
of Theorem \ref{thm:DoubInt}. Indeed, let $k$ be a nonzero even
integer in (\ref{eq:DoubInt}). By taking $b=k\pi i$ in the relation
(\ref{eq:DoubInt1Der}), one finds that
\[
0=\psi(k\pi i)=\big(\int_{0}^{1}e^{k\pi ix_{s}}dy_{s}\big)^{2}=0\iff\int_{0}^{1}e^{k\pi ix_{s}}dy_{s}=0.
\]
This is exactly the previous line integral identity (\ref{eq:LineInt})
since $k$ is even. 
\end{rem}
\begin{rem}\label{rem:DoubInt3}
According to (\ref{eq:LineInt}), one has 
\begin{align*}
\int_{0<s<t<1}\big(e^{2\pi i(px_{s}+qx_{t})}+e^{2\pi i(px_{t}+qx_{s})}\big)dy_{s}dy_{t} & =\big(\int_{0}^{1}e^{2\pi ipx_{s}}dy_{s}\big)\big(\int_{0}^{1}e^{2\pi iqx_{t}}dy_{t}\big)=0.
\end{align*}
In particular, the relation (\ref{eq:DoubInt2}) can also be rewritten
as 
\[
\int_{0<s<t<1}e^{2\pi i(px_{s}+qx_{t})}dy_{s}dy_{t}=0
\]
for all $p,q\in\mathbb{Z}\backslash\{0\}$ with $p+q\neq0$.

\end{rem}

\begin{rem}
The information encoded by the condition (\ref{eq:DoubInt}) is  larger than the one encoded in (\ref{eq:DoubInt2}). For instance, by further differentiating (\ref{eq:DoubInt1Der}) and
setting $b=0$ one obtains that 
\begin{equation}
-\varphi(0)+\psi(0)+\psi'(0)=0.\label{eq:NewIntPf}
\end{equation}
Note that $\psi(0)=0$ as a consequence of (\ref{eq:LineInt}). Therefore,
the relation (\ref{eq:NewIntPf}) becomes 
\[
-\int_{0<s<t<1}\big(e^{2k\pi i\cdot x_{s}}-e^{2k\pi i\cdot x_{t}}\big)dy_{s}dy_{t}-y_{1}\int_{0}^{1}x_{s}e^{2k\pi i\cdot x_{s}}dy_{s}=0.
\]
After simplification, this can be rewritten as the following identity:
\[
\int_{0}^{1}y_{t}e^{2k\pi i\cdot x_{t}}dy_{t}-y_{1}\int_{0}^{1}x_{t}e^{2k\pi i\cdot x_{t}}dy_{t}=0.
\]
This new identity cannot be implied by (\ref{eq:LineInt}) or (\ref{eq:DoubInt2}). One can also show by a similar type of calculation that \[
\int_{0<s<t<1}e^{\pi i(px_{s}+qx_{t})}dy_{s}dy_{t}=0
\]
for all odd integers $p,q$ with $p+q\neq0,$ as a consequence of Theorem \ref{thm:DoubInt}.
\end{rem}

\begin{example}\label{exam:Eight}
We use one example to illustrate the usefulness of the second order integral identity (\ref{eq:DoubInt2}). Consider the following piecewise linear path defined by a ``Figure Eight'' trajectory:
\begin{center}
\includegraphics[width=8cm]{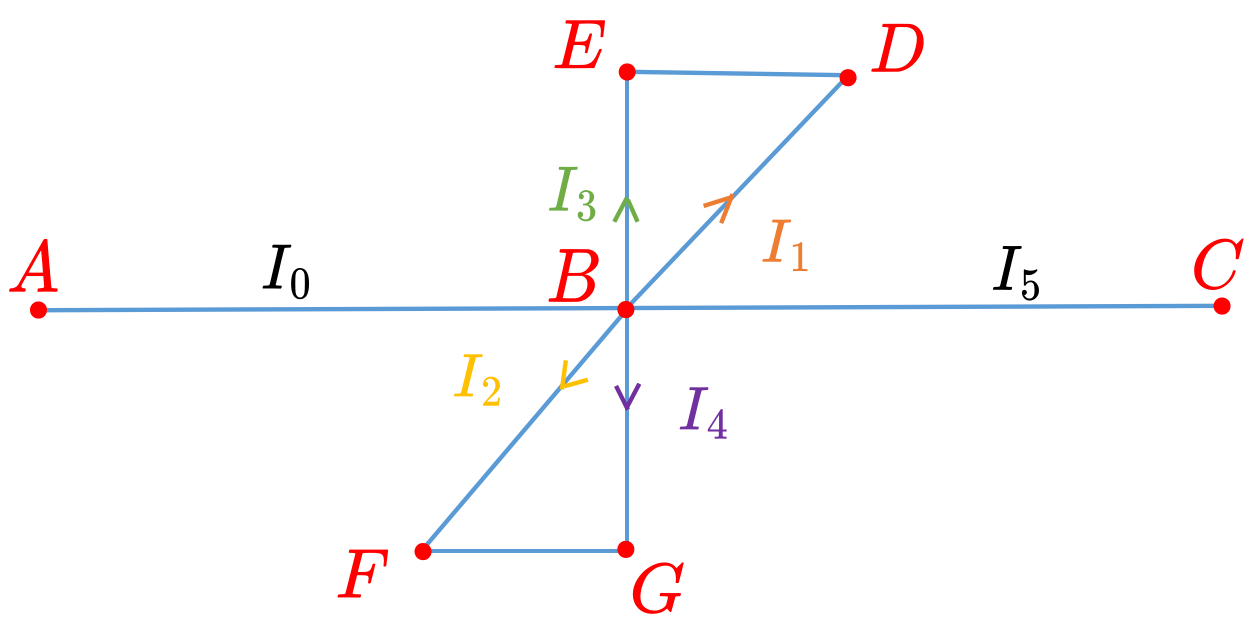}
\end{center}
Here $A$ is the origin, $B=(1/2,0)$, $C=(1,0)$ and

\begin{equation*}
\begin{aligned}
\overset{\longrightarrow}{BD}=\big(\frac{1}{4},\frac{1}{4}\big),\quad\overset{\longrightarrow}{BE}=\big(0,\frac{1}{4}\big),
\quad\overset{\longrightarrow}{BG}=\big(0,-\frac{1}{4}\big),\quad\overset{\longrightarrow}{BF}=\big(-\frac{1}{4},-\frac{1}{4}\big).
\end{aligned} 
\end{equation*}
The path $\gamma$ is defined by
\[
I_{0}\rightarrow I_{1}\rightarrow I_{2}\rightarrow I_{3}\rightarrow I_{4}\rightarrow I_{5},
\]where $I_1$ (respectively, $I_3$) denotes the loop $BDEB$ (respectively, $BEDB$) and $I_2$ (respectively, $I_4$) denotes the loop $BFGB$ (respectively, $BGFB$). Explicit calculation shows that \[
\int_{0<s<t<1}e^{2\pi ix_{s}+4\pi ix_{t}}dy_{s}dy_{t}\approx-0.05-0.08i\neq0.
\]Therefore, one concludes from Corollary \ref{cor:DoubIntCond} and Remark \ref{rem:DoubInt3} that $\log S(\gamma)$ has finite R.O.C. Note that this example \textit{cannot} be handled by Theorem \ref{thm:GenLineInt}; it is easily seen that $$\int_0^1 \Phi(d\gamma_t)=\int_{I_0\cup I_5}\Phi(d\gamma_t)=0$$ for any $\Phi$ satisfying the conditions in that theorem.

\end{example}

In the following subsections, we develop the proof of Theorem \ref{thm:DoubInt}. The key insight is to make use of path developments into \textit{complex semisimple} Lie algebras.

\subsection{Homogeneous projection formulae}

Our strategy relies on certain tensor projections as a starting point.
We first define these operations. Let $\{{\rm e}_{1},{\rm e}_{2}\}$
be the standard basis of $\mathbb{R}^{2}.$ 
\begin{defn}\label{defn:HomProj}
For each $m\geqslant1$, we define $D_{m}:T((\mathbb{R}^{2}))\rightarrow T((\mathbb{R}^{2}))$
to be linear operator induced by 
\[
D_{m}({\rm e}_{j_{1}}\otimes\cdots\otimes{\rm e}_{j_{n}})\triangleq\begin{cases}
{\rm e}_{j_{1}}\otimes\cdots\otimes{\rm e}_{j_{n}}, & \#\{l:j_{l}=2\}=m;\\
0, & \text{otherwise}.
\end{cases}
\]
In other words, $D_{m}$ projects a tensor series onto the sub-series
whose components have precisely $m$ of ${\rm e}_{2}$'s. 
\end{defn}
Suppose that $\gamma$ is a two-dimensional weakly geometric rough path whose
logarithmic signature $L\triangleq\log S(\gamma)$ has infinite R.O.C.
It is easy to see that $D_{m}L$ has infinite R.O.C. for all $m\geqslant1$.
As we will show, the main philosophy is that 
\[
D_{m}L\text{ has infinite R.O.C.}\implies\text{suitable }m\text{-th order iterated integral identities}.
\]
When $m=1,$ this gives back the line integral condition (\ref{eq:LineInt}).
Theorem \ref{thm:DoubInt} follows from the consideration of $D_{2}L$.
The higher order iterated integral identities we will derive in Section
\ref{sec:HighInt} are based on the consideration of general $D_{m}L$
($m\geqslant3$). 

In order to prove Theorem \ref{thm:DoubInt}, we first derive the
expressions of $D_{1}L$ and $D_{2}L$, leaving the more general situation
to Section \ref{sec:HighInt}. This is contained in the lemma below.
Recall that $\{B_{m}\}$ are the Bernoulli numbers arising from the
first Hausdorff series $H_{1}$ (cf. (\ref{eq:H1Formula})).
\begin{lem}\label{lem:D12Form}
The following formulae for $D_{1}L$ and $D_{2}L$ hold true:
\begin{equation}
D_{1}L=\sum_{m=0}^{\infty}\frac{B_{m}}{m!}{\rm ad}_{{\rm e}_{1}}^{m}\big(\int_{0}^{1}e^{x_{t}{\rm ad}_{{\rm e}_{1}}}({\rm e}_{2})dy_{t}\big)\label{eq:D1Form}
\end{equation}
and
\begin{align}
D_{2}L & =\frac{1}{2}\sum_{m=0}^{\infty}\frac{B_{m}}{m!}{\rm ad}_{{\rm e}_{1}}^{m}\big(\int_{0<s<t<1}\big[e^{x_{s}{\rm ad}_{{\rm e}_{1}}}({\rm e}_{2}),e^{x_{t}{\rm ad}_{{\rm e}_{1}}}({\rm e}_{2})\big]dy_{s}dy_{t}\big)\nonumber \\
 & \ \ \ +\frac{1}{2}\sum_{m,l\geqslant0}\frac{B_{m}B_{l}}{m!l!}\sum_{k=1}^{m}{\rm ad}_{{\rm e}_{1}}^{k-1}\big(\big[{\rm ad}_{{\rm e}_{1}}^{l}(\int_{0}^{1}e^{x_{s}{\rm ad}_{{\rm e}_{1}}}({\rm e}_{2})dy_{s}),{\rm ad}_{{\rm e}_{1}}^{m-k}(\int_{0}^{1}e^{x_{t}{\rm ad}_{{\rm e}_{1}}}({\rm e}_{2})dy_{t})\big]\big).\label{eq:D2Form}
\end{align}
\end{lem}
\begin{proof}
Let $\tilde{L}$ denote the logarithmic signature of $\tilde{\gamma}\triangleq\gamma\sqcup\overleftarrow{{\rm e}_{1}}$.
Recall from (\ref{eq:HausRelationTensor}) that
\[
L=B(\tilde{L},{\rm e}_{1})=\sum_{n=1}^{\infty}H_{n}(\tilde{L},{\rm e}_{1}).
\]
By applying projections $D_{1}$ and $D_{2}$ to both sides, one finds
that 
\begin{equation}
D_{1}L=H_{1}(D_{1}\tilde{L},{\rm e}_{1}),\ D_{2}L=H_{1}(D_{2}\tilde{L},{\rm e}_{1})+H_{2}(D_{1}\tilde{L},{\rm e}_{1}).\label{eq:D1D2}
\end{equation}
Here we used the fact that the first level of $\tilde{L}$ does not
contain the ${\rm e}_{1}$-component (due to the normalisation (\ref{eq:PathNorm})). 

We first compute $D_{1}\tilde{L}$ and $D_{2}\tilde{L}.$ Recall from
Lemma \ref{lem:AdDecomp} that 
\begin{align*}
e^{\tilde{L}} & =\sum_{n=0}^{\infty}\int_{0<t_{1}<\cdots<t_{n}<1}e^{x_{t_{1}}{\rm ad}_{{\rm e}_{1}}}({\rm e}_{2})\otimes\cdots\otimes e^{x_{t_{n}}{\rm ad}_{{\rm e}_{1}}}({\rm e}_{2})dy_{t_{1}}\cdots dy_{t_{n}}\\
 & =1+\tilde{L}+\frac{1}{2}\tilde{L}^{\otimes2}+\cdots=1+(D_{1}\tilde{L}+D_{2}\tilde{L}+\cdots)+\frac{1}{2}(D_{1}\tilde{L}+D_{2}\tilde{L}+\cdots)^{\otimes2}+\cdots.
\end{align*}
By applying $D_{i}$ ($i=1,2$) to both sides, one finds that 
\begin{equation}
D_{1}\tilde{L}=\int_{0}^{1}e^{x_{t}{\rm ad}_{{\rm e}_{1}}}({\rm e}_{2})dy_{t}\label{eq:TildD1Form}
\end{equation}
and 
\begin{align}
D_{2}\tilde{L} & =\int_{0<s<t<1}e^{x_{s}{\rm ad}_{{\rm e}_{1}}}({\rm e}_{2})\otimes e^{x_{t}{\rm ad}_{{\rm e}_{1}}}({\rm e}_{2})dy_{s}dy_{t}-\frac{1}{2}D_{1}\tilde{L}\otimes D_{1}\tilde{L}\nonumber \\
 & =\frac{1}{2}\int_{0<s<t<1}\big[e^{x_{s}{\rm ad}_{{\rm e}_{1}}}({\rm e}_{2}),e^{x_{t}{\rm ad}_{{\rm e}_{1}}}({\rm e}_{2})\big]dy_{s}dy_{t}.\label{eq:TildD2Form}
\end{align}
Next, we compute the Hausdorff series $H_{1}$ and $H_{2}.$ According
to the definition (\ref{eq:H1Formula}) of $H_{1}$, one has 
\begin{equation}
H_{1}(D_{i}\tilde{L},{\rm e}_{1})=\sum_{m=0}^{\infty}\frac{B_{m}}{m!}{\rm ad}_{{\rm e}_{1}}^{m}(D_{m}\tilde{L})\ \ \ (i=1,2).\label{eq:H1Di}
\end{equation}
Respectively, the formula (\ref{eq:HnFormula}) for $H_{n}$ (with
$n=2$) yields that 
\begin{align}
H_{2}(D_{1}\tilde{L},{\rm e}_{1}) & =\frac{1}{2}\big(H_{1}\frac{\partial}{\partial{\rm e}_{1}}\big)H_{1}(D_{1}\tilde{L},{\rm e}_{1})\nonumber \\
 & =\frac{1}{2}\sum_{m=0}^{\infty}\frac{B_{m}}{m!}\big(H_{1}\frac{\partial}{\partial{\rm e}_{1}}\big)\big({\rm ad}_{{\rm e}_{1}}^{m}(D_{1}\tilde{L})\big)\nonumber \\
 & =\frac{1}{2}\sum_{m=0}^{\infty}\frac{B_{m}}{m!}\sum_{k=1}^{m}{\rm ad}_{{\rm e}_{1}}^{k-1}\circ{\rm ad}_{H_{1}(D_{1}\tilde{L},{\rm e}_{1})}\circ{\rm ad}_{{\rm e}_{1}}^{m-k}(D_{1}\tilde{L})\nonumber \\
 & =\frac{1}{2}\sum_{m,l\geqslant0}\frac{B_{m}B_{l}}{m!l!}\sum_{k=1}^{m}{\rm ad}_{{\rm e}_{1}}^{k-1}\big(\big[{\rm ad}_{{\rm e}_{1}}^{l}(D_{1}\tilde{L}),{\rm ad}_{{\rm e}_{1}}^{m-k}(D_{1}\tilde{L})\big]\big).\label{eq:H2D1}
\end{align}
After substituting the expressions (\ref{eq:TildD1Form}) and (\ref{eq:TildD2Form})
into (\ref{eq:H1Di}) and (\ref{eq:H2D1}), the lemma follows from
the relation (\ref{eq:D1D2}).
\end{proof}

\subsubsection*{Recapturing Theorem \ref{thm:LineInt} through the $D_{1}$-projection}

The formula (\ref{eq:D1Form}) easily recovers Theorem \ref{thm:LineInt}.
Indeed, by taking the development
\begin{equation}
F_{\lambda}:{\rm e}_{1}\mapsto A,{\rm e}_{2}\mapsto D\ \text{with Lie structure }[A,D]=\lambda D.\label{eq:DevD1}
\end{equation}
The relation (\ref{eq:D1Form}) yields
\[
\hat{F}_{\lambda}\big(D_{1}L\big)=\sum_{m=0}^{\infty}\frac{B_{m}}{m!}\lambda^{m}\big(\int_{0}^{1}e^{\lambda x_{t}}dy_{t}D\big)=\frac{\lambda}{e^{\lambda}-1}\times\big(\int_{0}^{1}e^{\lambda x_{t}}dy_{t}\big)D.
\]
Since the left hand side defines an entire function in $\lambda\in\mathbb{C}$,
one must have
\begin{equation}
\int_{0}^{1}e^{2k\pi i\cdot x_{t}}dy_{t}=0\ \forall k\in\mathbb{Z}\backslash\{0\}.\label{eq:LineIntCond}
\end{equation}
Of course, this argument is essentially the same as our earlier proof
of Theorem \ref{thm:LineInt}.

An interesting point is that in the $D_{1}L$ case, one can achieve
more by showing that the line integral condition (\ref{eq:LineInt})
is also \textit{sufficient} for $D_{1}L$ to have infinite R.O.C. Indeed, one
can rewrite (\ref{eq:D1Form}) as 
\[
D_{1}L=\sum_{m,p=0}^{\infty}\frac{B_{m}}{m!p!}{\rm ad}_{{\rm e}_{1}}^{m+p}({\rm e}_{2})\int_{0}^{1}x_{t}^{p}dy_{t}=\sum_{N=0}^{\infty}\big(\sum_{m=0}^{N}\frac{B_{m}}{m!(N-m)!}\int_{0}^{1}x_{t}^{N-m}dy_{t}\big){\rm ad}_{{\rm e}_{1}}^{N}({\rm e}_{2}).
\]
Since the Lie polynomials ${\rm ad}_{{\rm e}_{1}}^{N}({\rm e}_{2})$
($N\geqslant0$) all have different degrees and
\[
C_{1}\|{\rm ad}_{{\rm e}_{1}}^{N}({\rm e}_{2})\|\leqslant2^{N},
\]
one has for every $\lambda>0$ that
\begin{align}
F(\lambda) & \triangleq\sum_{N=0}^{\infty}\|\pi_{N+1}(D_{1}L)\|\lambda^{N}\nonumber \\
 & \leqslant\sum_{N=0}^{\infty}\big|\sum_{m=0}^{N}\frac{B_{m}}{m!(N-m)!}\int_{0}^{1}x_{t}^{N-m}dy_{t}\big|\lambda^{N}2^{N}:=\sum_{N=0}^{\infty}|c_{N}|(2\lambda)^{N}.\label{eq:D1Series}
\end{align}
Observe that 
\[
\sum_{N=0}^{\infty}c_{N}(2\lambda)^{N}=\big(\sum_{m}\frac{B_{m}}{m!}(2\lambda)^{m}\big)\big(\sum_{p}\frac{1}{p!}\int_{0}^{1}x_{t}^{p}dy_{t}\big)=\frac{2\lambda}{e^{2\lambda-1}}\times\int_{0}^{1}e^{2\lambda x_{t}}dy_{t}.
\]
Now suppose that the line integral condition (\ref{eq:LineInt}) holds.
Then the function 
\[
z\mapsto\frac{z}{e^{z}-1}\times\int_{0}^{1}e^{zx_{t}}dy_{t}
\]
is entire. As a result, the power series $\underset{N\geqslant0}{\sum}c_{N}z^{N}$
has infinite R.O.C. and so does $\underset{N\geqslant0}{\sum}|c_{N}|z^{N}$.
In view of (\ref{eq:D1Series}), this implies that $D_{1}L$ has infinite
R.O.C. 

To summarise, one has obtained the following neat result. 
\begin{proposition}
Let $\gamma:[0,1]\rightarrow\mathbb{R}^{2}$ be a weakly
geometric rough path satisfying the normalisation condition (\ref{eq:PathNorm}). Then
$D_{1}\log S(\gamma)$ has infinite R.O.C. if and only if $\int_{0}^{1}e^{2k\pi ix_{t}}dy_{t}=0$
for all nonzero integers $k$. 
\end{proposition}

\subsection{Complex semisimple Lie algebras}

In order to prove Theorem \ref{thm:DoubInt}, one needs to investigate
the $D_{2}$-projection. Our method relies on path developments into
a particular type of Lie algebras: \textit{complex semisimple Lie algebras}.
We first provide a quick review
on relevant concepts. 
\begin{defn}
A (finite dimensional) complex Lie algebra $\mathfrak{g}$ is said
to be \textit{semisimple} if it can be decomposed into a direct sum
$\mathfrak{g}\cong\mathfrak{g}_{1}\oplus\cdots\oplus\mathfrak{g}_{r}$,
where each $\mathfrak{g}_{i}$ is a \textit{simple} Lie subalgebra
in the sense that it does not contain non-trivial proper ideals.
\end{defn}
A key concept in semisimple Lie theory is the notion of Cartan subalgebras.
Let $\mathfrak{g}$ be a complex semisimple Lie algebra.
\begin{defn}
A \textit{Cartan subalgebra} $\mathfrak{h}$ of $\mathfrak{g}$ is
a subspace which satisfies the following two properties:

\vspace{2mm}\noindent (i) $\mathfrak{h}$ is a maximal abelian subalgebra;

\vspace{2mm}\noindent (ii) ${\rm ad}_{H}\in{\rm End}(\mathfrak{g})$ is diagonalisable (over
$\mathbb{C}$) for each $H\in\mathfrak{h}$.
\end{defn}
It is well known that a Cartan subalgebra always exists and is unique
up to conjugation in $\mathfrak{g}$. Let $\mathfrak{h}$ be a Cartan
subalgebra of $\mathfrak{g}$. Given any representation $\rho:\mathfrak{g}\rightarrow{\rm End}(W)$
over a finite dimensional complex vector space $W,$ all elements
of $\mathfrak{h}$ are simultaneously diagonalisable when they are
viewed as linear transformations over $W.$ As a result, $W$ admits
a decomposition into common $\mathfrak{h}$-eigenspaces. To
be precise, a complex linear functional $\mu\in\mathfrak{h}^{*}$
is said to be a \textit{weight} for $\rho$ if the subspace 
\begin{equation}
W^{\mu}\triangleq\{w\in W:\rho(H)(w)=\mu(H)w\ \forall h\in\mathfrak{h}\}\label{eq:WeightSpace}
\end{equation}
is non-trivial. There are at most finitely many weights for $\rho$
due to finite dimensionality and their collection is denoted as $\Pi(\rho).$
The space $W$ admits a decomposition (simultaneous diagonalisation)
\[
W=\bigoplus_{\mu\in\Pi(\rho)}W^{\mu},
\]
where for each $H\in\mathfrak{h}$, $W^{\mu}$ is an eigenspace of
$\rho(H)$ with eigenvalue $\mu(H)$ ($\mu\in\Pi(\rho)$).

The above general consideration, being applied to the adjoint representation
${\rm ad}:\mathfrak{g}\rightarrow{\rm End}(\mathfrak{g}),$ gives
the so-called \textit{root space decomposition }of $\mathfrak{g}.$
Let $\alpha\in\mathfrak{h}^{*}.$ Similar to (\ref{eq:WeightSpace})
we define the subspace 
\[
\mathfrak{g}^{\alpha}\triangleq\{X\in\mathfrak{g}:{\rm ad}_{H}(X)=\alpha(H)X\ \forall H\in\mathfrak{h}\}.
\]
It is readily checked that $\mathfrak{g}^{0}=\mathfrak{h}$ and $[\mathfrak{g}^{\alpha},\mathfrak{g}^{\beta}]\subseteq\mathfrak{g}^{\alpha+\beta}$
for all $\alpha,\beta\in\mathfrak{h}^{*}.$ A complex linear functional
$\alpha\in\mathfrak{h}^{*}$ is called a \textit{root} of $\mathfrak{g}$
with respect to $\mathfrak{h}$ if it is a weight for the adjoint
representation (i.e. if $\mathfrak{g}^{\alpha}\neq\{0\}$). In this
case, $\mathfrak{g}^{\alpha}$ is called the \textit{root space} associated
with $\alpha.$ It is clear that there are at most finitely many roots
of $\mathfrak{g}$.

\begin{thm}[Root Space Decomposition] Let $\Delta$ be the set of
all nonzero roots with respect to a given Cartan subalgebra $\mathfrak{h}.$
Then $\mathfrak{g}$ can be written as the direct sum 
\begin{equation}
\mathfrak{g}=\mathfrak{h}+\sum_{\alpha\in\Delta}\mathfrak{g}^{\alpha}.\label{eq:RootDecomp}
\end{equation}
In this decomposition, one has $\dim\mathfrak{g}^{\alpha}=1$ for
each $\alpha\in\Delta$. Moreover, if $\alpha,\beta,\alpha+\beta\in\Delta$
then $[\mathfrak{g}^{\alpha},\mathfrak{g}^{\beta}]=\mathfrak{g}^{\alpha+\beta}.$

\end{thm}

The following basic example will play an essential role in the proof
of Theorem \ref{thm:DoubInt} (and more significantly, of Theorem \ref{thm:IterIntCond} below).
\begin{example}
\label{exa:SlnC}Let $\mathfrak{g}=\mathfrak{sl}_{n}(\mathbb{C})$
be the Lie algebra of $n\times n$ matrices over $\mathbb{C}$ with
zero trace. A Cartan subalgebra $\mathfrak{h}$ can be taken as the
subspace of diagonal matrices in $\mathfrak{g}.$ One has $\dim\mathfrak{h}=n-1$.
For each $1\leqslant i\leqslant n$, let $\mu_{i}\in\mathfrak{h}^{*}$
be the linear functional defined by taking the $i$-th diagonal entry
of $H\in\mathfrak{h}$. Then the set of nonzero roots is given by
\[
\Delta=\{\lambda_{ij}\triangleq\mu_{i}-\mu_{j}:1\leqslant i\neq j\leqslant n\}.
\]
Respectively, the root spaces are given by
\[
\mathfrak{g}^{\lambda_{ij}}=\mathbb{C}\cdot E_{ij},\ \ \ (i\neq j)
\]
where $E_{ij}$ denotes the matrix whose $(i,j)$-entry is one and
all other entries are zero.
\end{example}

\subsection{Proof of Theorem \ref{thm:DoubInt} through $\mathfrak{sl}_{3}(\mathbb{C})$-development}

A nice feature about the root space decomposition is that the adjoint
actions by elements of the Cartan subalgebra are simultaneously diagonalised
into root spaces (eigenspaces). In particular, if one applies a path
development into a complex semisimple Lie algebra with ${\rm e}_{1}$
being mapped into a Cartan element, the formulae (\ref{eq:D1Form},
\ref{eq:D2Form}) for $D_{2}L$ could potentially be projected along
suitable root spaces to yield \textit{scalar} equations similar to
the one obtained in (\ref{eq:FormalHaus}). This will allow one to
treat $D_{2}L$ as a scalar entire function and to obtain suitable
integral conditions in a similar fashion as in the proof of Theorem
\ref{thm:LineInt}.

To make this idea precise, let $\mathfrak{g}$ be a complex semisimple
Lie algebra with root space decomposition (\ref{eq:RootDecomp}).
We consider a path development $F:\mathbb{C}^{2}\rightarrow\mathfrak{g}$
induced by 
\[
{\rm e}_{1}\mapsto A,\ {\rm e}_{2}\mapsto D
\]
 where $A,D$ are chosen to be such that
\[
A\in\mathfrak{h},\ D=\sum_{\alpha\in\Delta}c_{\alpha}E_{\alpha}\in\sum_{\alpha\in\Delta}\mathfrak{g}^{\alpha}.
\]
Here $E_{\alpha}$ is a generator of $\mathfrak{g}^{\alpha}$ (recall
that $\dim\mathfrak{g}^{\alpha}=1$) and $c_{\alpha}\in\mathbb{C}$.
Under such a development, one has 
\[
\hat{F}\big(\big[e^{x_{s}{\rm ad}_{{\rm e}_{1}}}({\rm e}_{2}),e^{x_{t}{\rm ad}_{{\rm e}_{1}}}({\rm e}_{2})\big]\big)=\sum_{\alpha,\beta\in\Delta}c_{\alpha}c_{\beta}e^{\alpha(A)x_{s}+\beta(A)x_{t}}[E_{\alpha},E_{\beta}].
\]
A crucial point is that $[E_{\alpha},E_{\beta}]\neq0$ if $\alpha+\beta\in\Delta$
and one can also make $\alpha(A),\beta(A)$ arbitrary by varying $A$. 

\begin{proof}[Proof of Theorem \ref{thm:DoubInt}]

As the simplest non-trivial example, in what follows we choose $\mathfrak{g}=\mathfrak{sl}_{3}(\mathbb{C})$.
According to Example \ref{exa:SlnC}, the Cartan subalgebra $\mathfrak{h}$
consists of diagonal matrices in $\mathfrak{g}$ and the roots are
given by 
\[
\lambda_{ij}(H)=H_{ii}-H_{jj},\ \ \ H\in\mathfrak{h}.\ \ \ (1\leqslant i\neq j\leqslant3)
\]
Let $a,b$ be two given complex numbers. We construct a path development
$F_{a,b}:\mathbb{C}^{2}\rightarrow\mathfrak{g}$ by specifying 
\[
{\rm e}_{1}\mapsto A\triangleq\left(\begin{array}{ccc}
\frac{2a+b}{3} & 0 & 0\\
0 & \frac{b-a}{3} & 0\\
0 & 0 & -\frac{a+2b}{3}
\end{array}\right),\ {\rm e}_{2}\mapsto D\triangleq E_{12}+E_{23}=\left(\begin{array}{ccc}
0 & 1 & 0\\
0 & 0 & 1\\
0 & 0 & 0
\end{array}\right).
\]
In particular, one has 
\[
D\in\mathfrak{g}^{\lambda_{12}}+\mathfrak{g}^{\lambda_{2,3}};\ \lambda_{12}(A)=a,\ \lambda_{23}(A)=b;\ [E_{12},E_{23}]=E_{13}\in\mathfrak{g}^{\lambda_{13}}.
\]
It follows that
\begin{align*}
\hat{F}_{a,b}\big(\big[e^{x_{s}{\rm ad}_{{\rm e}_{1}}}({\rm e}_{2}),e^{x_{t}{\rm ad}_{{\rm e}_{1}}}({\rm e}_{2})\big]\big) & =\big[e^{ax_{s}}E_{12}+e^{bx_{s}}E_{23},e^{ax_{t}}E_{12}+e^{bx_{t}}E_{23}\big]\\
 & =\big(e^{ax_{s}+bx_{t}}-e^{ax_{t}+bx_{s}}\big)E_{13}.
\end{align*}
According to the formula (\ref{eq:H1Di}) for $H_{1}(D_{2}\tilde{L},{\rm e}_{1})$,
one finds that
\begin{align*}
\hat{F}_{a,b}\big(H_{1}(D_{2}\tilde{L},{\rm e}_{1})\big) & =\frac{1}{2}\sum_{m=0}^{\infty}\frac{B_{m}}{m!}{\rm ad}_{A}^{m}\big(\int_{0<s<t<1}\big(e^{ax_{s}+bx_{t}}-e^{ax_{t}+bx_{s}}\big)dy_{s}dy_{t}\big)E_{13}\\
 & =\frac{1}{2}\sum_{m=0}^{\infty}\frac{B_{m}}{m!}(a+b)^{m}\int_{0<s<t<1}\big(e^{ax_{s}+bx_{t}}-e^{ax_{t}+bx_{s}}\big)dy_{s}dy_{t}\cdot E_{13}\\
 & =\frac{1}{2}\cdot\frac{a+b}{e^{a+b}-1}\cdot\int_{0<s<t<1}\big(e^{ax_{s}+bx_{t}}-e^{ax_{t}+bx_{s}}\big)dy_{s}dy_{t}\cdot E_{13}.
\end{align*}
Similarly, by using the formula (\ref{eq:H2D1}) for $H_{2}(D_{1}\tilde{L},{\rm e}_{1})$
one has 
\begin{align*}
\hat{F}_{a,b}\big(H_{2}(D_{1}\tilde{L},{\rm e}_{1})\big)= & \frac{1}{2}\int_{0}^{1}\int_{0}^{1}dy_{s}dy_{t}\times\sum_{m,l\geqslant0}\frac{B_{m}B_{l}}{m!l!}\cdot\\
 & \ \ \ \sum_{k=1}^{m}{\rm ad}_{A}^{k-1}\big[a^{l}e^{ax_{s}}E_{12}+b^{l}e^{bx_{s}}E_{23},a^{m-k}e^{ax_{t}}E_{12}+b^{m-k}e^{bx_{t}}E_{23}\big]\\
= & \frac{1}{2}\int_{0}^{1}\int_{0}^{1}dy_{s}dy_{t}\times\sum_{m,l\geqslant0}\frac{B_{m}B_{l}}{m!l!}\\
 & \ \ \ \times\sum_{k=1}^{m}\big(a^{l}b^{m-k}e^{ax_{s}+bx_{t}}-a^{m-k}b^{l}e^{ax_{t}+bx_{s}}\big)(a+b)^{k-1}\cdot E_{13}.
\end{align*}
After evaluating the summations explicitly, one finds that 
\begin{align*}
\hat{F}_{a,b}\big(H_{2}(D_{1}\tilde{L},{\rm e}_{1})\big) & =\frac{1}{2}\big(\frac{(a+b)(e^{b}-e^{a})}{(e^{a+b}-1)(e^{a}-1)(e^{b}-1)}-\frac{b-a}{(e^{a}-1)(e^{b}-1)}\big)\\
 & \ \ \ \times\big(\int_{0}^{1}e^{ax_{t}}dy_{t}\big)\big(\int_{0}^{1}e^{bx_{t}}dy_{t}\big)E_{1,3}.
\end{align*}
Consequently, 
\begin{align*}
\hat{F}_{a,b}(D_{2}L) & =\hat{F}_{a,b}(H_{1}(D_{2}\tilde{L},{\rm e}_{1}))+\hat{F}_{a,b}(H_{2}(D_{1}\tilde{L},{\rm e}_{1}))\\
 & =\frac{1}{2}E_{1,3}\times\big(\frac{a+b}{e^{a+b}-1}\cdot\int_{0<s<t<1}\big(e^{ax_{s}+bx_{t}}-e^{ax_{t}+bx_{s}}\big)dy_{s}dy_{t}\\
 & \ \ \ +\big(\frac{(a+b)(e^{b}-e^{a})}{(e^{a+b}-1)(e^{a}-1)(e^{b}-1)}-\frac{b-a}{(e^{a}-1)(e^{b}-1)}\big(\int_{0}^{1}e^{ax_{t}}dy_{t}\big)\big(\int_{0}^{1}e^{bx_{t}}dy_{t}\big)\big).
\end{align*}
By multiplying $(e^{a}-1)(e^{b}-1)(e^{a+b}-1)$ on both sides, one
obtains that 
\begin{align}
 & (e^{a}-1)(e^{b}-1)(e^{a+b}-1)\hat{F}_{a,b}(D_{2}L)\nonumber \\
 & =\frac{1}{2}E_{13}\times\big((a+b)(e^{a}-1)(e^{b}-1)\int_{0<s<t<1}\big(e^{ax_{s}+bx_{t}}-e^{ax_{t}+bx_{s}}\big)dy_{s}dy_{t}\nonumber \\
 & \ \ \ +\big((a+b)(e^{b}-e^{a})-(b-a)(e^{a+b}-1)\big)\big(\int_{0}^{1}e^{ax_{t}}dy_{t}\big)\big(\int_{0}^{1}e^{bx_{t}}dy_{t}\big)\big).\label{eq:abId}
\end{align}

Under the assumption that $D_{2}L$ has infinite R.O.C., the left
hand side of (\ref{eq:abId}) defines an entire function in $(a,b)\in\mathbb{C}^{2}$.
Let $k\in\mathbb{Z}\backslash\{0\}$ be given fixed and let $a,b\in\mathbb{C}$
satisfy $a+b=2k\pi i$. Then (\ref{eq:abId}) becomes
\begin{align*}
0 & =2k\pi iE_{13}\times\big((1-\cosh b)\int_{0<s<t<1}\big(e^{2k\pi i\cdot x_{s}+b(x_{t}-x_{s})}-e^{2k\pi i\cdot x_{t}+b(x_{s}-x_{t})}\big)dy_{s}dy_{t}\\
 & \ \ \ +(\sinh b)\big(\int_{0}^{1}e^{(2k\pi i-b)x_{t}}dy_{t}\big)\big(\int_{0}^{1}e^{bx_{t}}dy_{t}\big)\big).
\end{align*}
Since $k\neq0,$ one concludes the desired integral identity (\ref{eq:DoubInt}).
This completes the proof of Theorem \ref{thm:DoubInt}.
\end{proof}

\begin{rem}
It is certainly possible to use other types of complex semisimple Lie algebras in the above analysis. We are not sure if the use of other root structures would lead to new integral identities that are not covered by (\ref{eq:DoubInt}). 
\end{rem}

\section{Higher order iterated integral identities}\label{sec:HighInt}

In this section, we establish the higher order counterpart of Corollary
\ref{cor:DoubIntCond}. To state our main theorem, we first introduce
the following definitions. 
\begin{defn}
A finite sequence of numbers $(a_{1},\cdots,a_{m})$ is said to be
\textit{non-degenerate }if it does not contain zero consecutive sums,
i.e. if
\[
a_{j}+a_{j+1}+\cdots+a_{k}\neq0
\]
for all $1\leqslant j\leqslant k\leqslant m$. 
\end{defn}
\begin{defn}
\label{def:Sm}Let $\gamma:[0,1]\rightarrow\mathbb{R}^{2}$ be a weakly geometric
rough path. For each $m\geqslant1$, we define an analytic function
$S_{m}^{\gamma}:\mathbb{C}^{m}\rightarrow\mathbb{C}$ in the following
way. Given $a_{1},\cdots,a_{m}\in\mathbb{C}$, let $B:[0,1]\rightarrow\mathbb{C}^{m}$
be the path defined by 
\[
B_{t}\triangleq\sum_{j=1}^{m}\big(\int_{0}^{t}e^{a_{j}x_{s}}dy_{s}\big){\rm e}_{j}
\]
where $\{{\rm e}_{1},\cdots,{\rm e}_{m}\}$ denotes the standard basis
of $\mathbb{C}^{m}$. We define $S_{m}^{\gamma}(a_{1},\cdots,a_{m})$
to be the coefficient of the logarithmic signature of $B$ with respect
to the tensor ${\rm e}_{1}\otimes\cdots\otimes{\rm e}_{m}$, i.e.
\[
S_{m}^{\gamma}(a_{1},\cdots,a_{m})\triangleq\big(\log S(B)\big)^{(1,\cdots,m)}.
\]
\end{defn}
\begin{example}
By explicit calculation, one finds that 
\begin{align*}
&S_{1}^{\gamma}(a)=\int_{0}^{1}e^{ax_{s}}dy_{s},\ S_{2}^{\gamma}(a,b)=\frac{1}{2}\int_{0<s<t<1}\big(e^{ax_{s}+bx_{t}}-e^{bx_{s}+ax_{t}}\big)dy_{s}dy_{t}
\end{align*}
and
\begin{align*}
S_{3}^{\gamma}(a,b,c) & =\int_{0<s<t<r<1}\big(\frac{1}{3}e^{ax_{s}+bx_{t}+cx_{r}}-\frac{1}{6}e^{ax_{s}+cx_{t}+bx_{r}}-\frac{1}{6}e^{bx_{s}+ax_{t}+cx_{r}}\\
 & \ \ \ -\frac{1}{6}e^{bx_{s}+cx_{t}+ax_{r}}-\frac{1}{6}e^{cx_{s}+ax_{t}+bx_{r}}+\frac{1}{3}e^{cx_{s}+bx_{t}+ax_{r}}\big)dy_{s}dy_{t}dy_{r}.
\end{align*}
\end{example}
\begin{rem}
One can consider a general $d$-dimensional path 
\[
B_{t}\triangleq\sum_{i=1}^{d}\big(\int_{0}^{t}e^{p_{i}x_{s}}dy_{s}\big){\rm e}_{i}
\]
It is not difficult to see that the logarithmic signature coefficient
of $B$ with respect to ${\rm e}_{i_{1}}\otimes\cdots\otimes{\rm e}_{i_{m}}$
is given by $S_{m}^{\gamma}(p_{i_{1}},\cdots,p_{i_{m}})$.
\end{rem}
Our main theorem of this section is stated as follows. As before,
we assume that $\gamma:[0,1]\rightarrow\mathbb{R}^{2}$ is a weakly geometric
rough path satisfying the normalisation condition (\ref{eq:PathNorm}). 
\begin{thm}
\label{thm:IterIntCond}Suppose that $\log S(\gamma)$ has infinite
R.O.C. Then one has 
\[
S_{m}^{\gamma}(a_{1},\cdots,a_{m})=0
\]
for all $m\geqslant1$ and all non-degenerate sequences $(a_{1},\cdots,a_{m})$
satisfying $a_{j}\in2\pi i\mathbb{Z}$ for each $j$. 
\end{thm}
\begin{rem}
When $m=1,2$, Theorem \ref{thm:IterIntCond} reduces to Theorem \ref{thm:LineInt}
and Corollary \ref{cor:DoubIntCond} respectively. It is possible
to establish higher order versions of the stronger Theorem \ref{thm:DoubInt}.
However, since the general formulae become significantly more involved
we decide not to pursue this generality. 
\end{rem}
Inspired by the second order case, our strategy for proving Theorem
\ref{thm:IterIntCond} will be based on path developments into $\mathfrak{sl}_{m+1}(\mathbb{C})$.
The general spirit is not-so-different from the second order argument.
However, the underlying algebraic structure becomes much subtler and
the argument involves several non-trivial combinatorial considerations. 

In the following subsections, we develop the proof of Theorem \ref{thm:IterIntCond}
in a precise mathematical way. We will continue to use the notation
introduced in Section \ref{sec:IterInt}.

\subsection{A Chen-Strichartz type formula for $D_{m}\tilde{L}$ }

Recall that $\tilde{L}$ is the logarithmic signature of the path
$\tilde{\gamma}\triangleq\gamma\sqcup\overleftarrow{{\rm e}}_{1}$
and $D_{m}$ is the projection operator defined in Definition \ref{defn:HomProj}. As a starting
point, we first derive an explicit formula for $D_{m}\tilde{L}$ which
generalises (\ref{eq:TildD1Form}, \ref{eq:TildD2Form}). Throughout the rest, we use $\Delta_{m}$ to denote the standard simplex  $0<t_{1}<\cdots<t_{m}<1$. 
\begin{proposition}
\label{prop:DmTilLForm}For any $m\geqslant1,$ one has
\begin{align}
D_{m}\tilde{L}= & \sum_{\sigma\in\textcolor{black}{{\cal S}_{m}}}\frac{(-1)^{e(\sigma)}}{m^{2}{m-1 \choose e(\sigma)}}\int_{\Delta_{m}}\big[e^{x_{t_{\sigma(1)}}{\rm ad}_{{\rm e}_{1}}}({\rm e}_{2}),\nonumber \\
 & \ \ \ \big[\cdots\big[e^{x_{t_{\sigma(m-1)}}{\rm ad}_{{\rm e}_{1}}}({\rm e}_{2}),e^{x_{t_{\sigma(m)}}{\rm ad}_{{\rm e}_{1}}}({\rm e}_{2})\big]\cdots\big]\big]dy_{t_{1}}\cdots dy_{t_{m}}.\label{eq:DmTilLForm}
\end{align}
Here we define
\[
e(\sigma)\triangleq\#\{j=1,\cdots,m-1:\sigma(j)>\sigma(j+1)\}
\]
for each permutation $\sigma\in{\cal S}_{m}.$
\end{proposition}
\begin{proof}
Consider the path 
\[
\Gamma_{t}\triangleq\int_{0}^{t}e^{x_{s}{\rm ad}_{{\rm e}_{1}}}({\rm e}_{2})dy_{s}\in W\triangleq T((\mathbb{C}^{2})).
\]
Its signature is given by 
\[
S(\Gamma)={\bf 1}+\sum_{m=1}^{\infty}\int_{\Delta_{m}}d\Gamma_{t_{1}}\boxtimes\cdots\boxtimes d\Gamma_{t_{m}}\in T((W)),
\]
where $\boxtimes$ denotes another tensor product that is independent
of $\otimes.$ According to the Chen-Strichartz formula (cf. \cite[Theorem 3.27]{Gen21}),
one has 
\begin{align*}
\log S(\Gamma)= & \sum_{m=1}^{\infty}\sum_{\sigma\in{\cal S}_{m}}\frac{(-1)^{e(\sigma)}}{m^{2}{m-1 \choose e(\sigma)}}\int_{\Delta_{m}}\llbracket e^{x_{t_{\sigma(1)}}{\rm ad}_{{\rm e}_{1}}}({\rm e}_{2}),\llbracket\cdots\\
 & \ \ \ \llbracket e^{x_{t_{\sigma(m-1)}}{\rm ad}_{{\rm e}_{1}}}({\rm e}_{2}),e^{x_{t_{\sigma(m)}}{\rm ad}_{{\rm e}_{1}}}({\rm e}_{2})\rrbracket\cdots\rrbracket\rrbracket dy_{t_{1}}\cdots dy_{t_{m}}\in{\cal L}((W)),
\end{align*}
where $\llbracket\cdot,\cdot\rrbracket$ denotes the commutator for
$\boxtimes$. Now recall from Lemma \ref{lem:AdDecomp} that the signature of $\tilde{\gamma}$ is
given by the formula \ref{eq:SigAdRep} (without the last $e^{\rm e_1}$-term). By applying the algebra homomorphism $\boxtimes\mapsto\otimes$
as well as the projection $D_{m}$ to the above relation, the left
hand side becomes $D_{m}\tilde{L}$ and the right hand side is precisely
(\ref{eq:DmTilLForm}).
\end{proof}

\subsection{$\mathfrak{sl}_{m+1}(\mathbb{C})$-development of the logarithmic
signature}

Our next step is to show that the logarithmic signature coefficients
$S_{m}^{\gamma}(a_{1},\cdots,a_{m})$ (cf. Definition \ref{def:Sm})
can be realised through suitable path developments. Let us formulate
this fact in a slightly more general setting. Suppose that $p_{1},\cdots,p_{d}\in\mathbb{C}$
are given fixed numbers. We define the path 
\begin{equation}
B_{t}\triangleq\sum_{j=1}^{d}\big(\int_{0}^{t}e^{p_{j}x_{s}}dy_{s}\big){\rm e}_{j}\in\mathbb{C}^{d}.\label{eq:BPath}
\end{equation}
Let $m\geqslant1$ and $I=(i_{1},\cdots,i_{m})\in\{1,\cdots,d\}^{m}$
be a given fixed word. We consider the path development $F:\mathbb{C}^{2}\rightarrow\mathfrak{sl}_{m+1}(\mathbb{C})$
(cf. Example \ref{exa:SlnC} for the relevant notation) induced by 
\begin{equation}
F_{I}({\rm e}_{1})\triangleq A\in\mathfrak{h},\ F_{I}({\rm e}_{2})\triangleq E_{12}+E_{23}+\cdots+E_{m,m+1},\label{eq:FI}
\end{equation}
where the Cartan element $A$ is chosen to satisfy
\begin{equation}
[A,E_{k,k+1}]=p_{i_{k}}E_{k,k+1},\ \ \ k=1,\cdots,m.\label{eq:DevGen}
\end{equation}
It is a simple linear algebra exercise to see that such an $F_{I}$
exists. Recall from Section \ref{sec:Cartan} that the induced homomorphisms at the tensor
and Lie series levels are both denoted as $\hat{F}_{I}.$ The main
result for this part stated as follows.
\begin{proposition}
\label{prop:DmTilLDevForm}One has 
\begin{equation}
\hat{F}_{I}(D_{m}\tilde{L})=(\log S(B))^{I}E_{1,m+1},\label{eq:DmTilLDevForm}
\end{equation}
where $(\cdot)^{I}$ denotes the coefficient of a tensor series over
$\mathbb{C}^{d}$ with respect to the monomial ${\rm e}_{i_{1}}\otimes\cdots\otimes{\rm e}_{i_{m}}.$
\end{proposition}
Our proof of Proposition \ref{prop:DmTilLDevForm}, which has a combinatorial
nature, reveals a surprising connection between two very different
Lie structures ($\mathfrak{sl}_{m+1}(\mathbb{C})$ and free Lie algebra).
As we will see, they are connected through a shuffle product relation.
We first introduce some notation that is needed for the proof and
then establish two key lemmas connecting the two sides with a specific
shuffle product relation. The identity (\ref{eq:DmTilLDevForm}) will
thus follow easily. 

\subsubsection*{Some notation}

(i) Suppose that $(w_{1},\cdots,w_{m})$ and $(w_{m+1},\cdots,w_{m+n})$
are two given words. We define their \textit{shuffle product} by
\begin{align*}
(w_{1},\cdots,w_{m})\shuffle(w_{m+1},\cdots,w_{m+n}) & \triangleq\sum_{\sigma\in{\cal P}(m,n)}(w_{\sigma^{-1}(1)},\cdots,w_{\sigma^{-1}(m+n)}).
\end{align*}
By definition, it is obvious
that 
\begin{equation}
a{\bf w}\shuffle b{\bf u}=a({\bf w}\shuffle b{\bf u})+b(a{\bf w}\shuffle{\bf u})\label{eq:ShufRel}
\end{equation}
for all letters $a,b$ and words ${\bf w},{\bf u}$. 

\vspace{2mm}\noindent (ii) Let $a_{j}\in\mathbb{C}$ be given numbers.
The notation $(e^{a_{p}},\cdots,e^{a_{q}})$ simply represents a word
and for given times $s_{j}$ we define 
\[
(e^{a_{p}},\cdots,e^{a_{q}})(s_{p},\cdots,s_{q})\triangleq\exp\big(a_{p}x_{s_{p}}+\cdots+a_{q}x_{s_{q}}\big).
\]
The tensor product between two words is simply defined by concatenation.
The Lie bracket $[e^{a},e^{b}]$ is defined by 
\[
[e^{a},e^{b}]\triangleq(e^{a},e^{b})-(e^{b},e^{a}).
\]
When acting on a pair of times $(s,t)$, one has 
\[
[e^{a},e^{b}](s,t)\triangleq(e^{a},e^{b})(s,t)-(e^{b},e^{a})(s,t)=e^{ax_{s}+bx_{t}}-e^{bx_{s}+ax_{t}}.
\]

\vspace{2mm}\noindent (iii) Let $a_{1},\cdots,a_{m}\in\mathbb{C}$
be given fixed. We introduce the word notation
\[
(e^{a_{p}},\stackrel{\nearrow}{\cdots},e^{a_{q}})\triangleq\begin{cases}
(e^{a_{p}},e^{a_{p+1}},\cdots,e^{a_{q}}), & \text{if }p\leqslant q;\\
{\bf 1}, & \text{if }p=q+1;\\
{\bf 0}, & \text{if }p>q,
\end{cases}
\]
and 
\[
(e^{a_{p}},\stackrel{\searrow}{\cdots},e^{a_{q}})\triangleq\begin{cases}
(e^{a_{p}},e^{a_{p-1}},\cdots,e^{a_{q}}), & \text{if }p\geqslant q;\\
{\bf 1}, & \text{if }p=q-1;\\
{\bf 0}, & \text{if }p<q.
\end{cases}
\]
Here ${\bf 1}$ means the multiplicative unit (for both concatenation
and shuffle product) and ${\bf 0}$ means the empty word whose (tensor
or shuffle) product with any other word is still ${\bf 0}.$ In other
words, conventionally one has 
\[
{\bf 1}\shuffle{\bf w}={\bf w}\shuffle\textcolor{black}{\mathbf{1}}={\bf w},\ {\bf 0}\shuffle{\bf w}={\bf w}\shuffle{\bf 0}={\bf 0}.
\]

\subsubsection*{Expression of $\hat{F}_{I}(D_{m}\tilde{L})$ in terms of a shuffle
identity}

The first key lemma is related to computing the development of a generic
term in the formula (\ref{eq:DmTilLForm}) for $D_{m}\tilde{L}$.
Recall that the development $F_{I}$ is defined by (\ref{eq:FI})
where $A$ is chosen according to (\ref{eq:DevGen}) and $a_{j}=p_{i_{j}}$. 
\begin{lem}
\label{lem:DevShuf}For any $k=1,\cdots,m-1$, one has 
\begin{align}
 & \big[e^{x_{t_{\sigma(m-k)}}{\rm ad}_{A}}(F_{I}({\rm e}_{2})),\big[\cdots\big[e^{x_{t_{\sigma(m-1)}}{\rm ad}_{A}}(F_{I}({\rm e}_{2})),e^{x_{t_{\sigma(m)}}{\rm ad}_{A}}(F_{I}({\rm e}_{2}))\big]\big]\big]\nonumber \\
 & =\sum_{j}\sum_{i=1}^{m-1}(-1)^{j+k-i-1}\big(((e^{a_{j}},\stackrel{\nearrow}{\cdots},e^{a_{i-1}})\shuffle(e^{a_{j+k}},\stackrel{\searrow}{\cdots},e^{a_{i+2}}))\nonumber \\
 & \ \ \ \ \ \ \otimes[e^{a_{i}},e^{a_{i+1}}]\big)(t_{\sigma(m-k)},\cdots,t_{\sigma(m)})E_{j,j+k+1}.\label{eq:DevShuf}
\end{align}
Here the Lie bracket is taken in $\mathfrak{sl}_{m+1}(\mathbb{C})$
and we adopt the convention that $E_{p,q}=0$ if $p,q\leqslant0$
or $\geqslant m+2$. 
\end{lem}
\begin{proof}
We prove the claim by induction on $k.$ Let us denote the right hand
side of (\ref{eq:DevShuf}) by $G(k)$. For the base step $k=1$,
by the definition of $F_{I}$ one has 
\begin{align}
 & \big[e^{x_{t_{\sigma(m-1)}}{\rm ad}_{A}}(F_{I}({\rm e}_{2})),e^{x_{t_{\sigma(m)}}{\rm ad}_{A}}(F_{I}({\rm e}_{2}))\big]\nonumber \\
 & =\big[e^{x_{t_{\sigma(m-1)}}{\rm ad}_{A}}(E_{12}+\cdots+E_{m,m+1}),e^{x_{t_{\sigma(m)}}{\rm ad}_{A}}(E_{12}+\cdots+E_{m,m+1})\big]\nonumber \\
 & =\big[\sum_{j=1}^{m}e^{a_{j}x_{t_{\sigma(m-1)}}}E_{j,j+1},\sum_{k=1}^{m}e^{a_{k}x_{t_{\sigma(m)}}}E_{k,k+1}\big].\label{eq:G1}
\end{align}
By using the explicit commutator relation
\begin{equation}
[E_{ij},E_{kl}]=\delta_{jk}E_{il}-\delta_{il}E_{kj},\label{eq:SLComm}
\end{equation}
one easily finds that the right hand side of (\ref{eq:G1}) equals
\begin{align*}
 & \sum_{i=1}^{m-1}\big(e^{a_{i}x_{t_{\sigma(m-1)}}}e^{a_{i+1}x_{t_{\sigma(m)}}}-e^{a_{i+1}x_{t_{\sigma(m-1)}}}e^{a_{i}x_{t_{\sigma(m)}}}\big)E_{i,i+2}\\
 & =\sum_{i=1}^{m-1}[e^{a_{i}},e^{a_{i+1}}](t_{\sigma(m-1)},t_{\sigma(m)})E_{i,i+2}=G(1).
\end{align*}
This concludes the base step.

Now suppose that the claim is true for $k\leqslant n.$ By using the
induction hypothesis, one can write 
\begin{align*}
 & \big[e^{x_{t_{\sigma(m-n-1)}{\rm ad}_{A}}}F_{I}({\rm e}_{2}),\big[e^{x_{t_{\sigma(m-n)}}{\rm ad}_{A}}(F_{I}({\rm e}_{2}))\\
 & \ \ \ \ \ \ \ \ \ \cdots\big[e^{x_{t_{\sigma(m-1)}}{\rm ad}_{A}}(F_{I}({\rm e}_{2})),e^{x_{t_{\sigma(m)}}{\rm ad}_{A}}(F_{I}({\rm e}_{2}))\big]\big]\big]\\
 & =\sum_{j}\sum_{i=1}^{m-1}(-1)^{j+n-i-1}\big[e^{a_{1}x_{t_{\sigma(m-n-1)}}}E_{12}+\cdots+e^{a_{m}x_{t_{\sigma(m-n-1)}}}E_{m,m+1},\\
 & \ \ \ \ \ \ \ \ \ \big(((e^{a_{j}},\stackrel{\nearrow}{\cdots},e^{a_{i-1}})\shuffle(e^{a_{j+n}},\stackrel{\searrow}{\cdots},e^{a_{i+2}}))\otimes[e^{a_{i}},e^{a_{i+1}}]\big)\\
 & \ \ \ \ \ \ \ \ \ (t_{\sigma(m-n)},\cdots,t_{\sigma(m)})E_{j,j+n+1}\big].
\end{align*}
Again by applying the commutator relation (\ref{eq:SLComm}) and adjusting
the $j$-index, the above expression is equal to 
\begin{align*}
 & \sum_{j}\sum_{i=1}^{m-1}(-1)^{j+n-i}\big(e^{a_{j}}\otimes((e^{a_{j+1}},\stackrel{\nearrow}{\cdots},e^{a_{i-1}})\shuffle(e^{a_{j+n+1}},\stackrel{\searrow}{\cdots},e^{a_{i+2}}))\otimes[e^{a_{i}},e^{a_{i+1}}]\big)\\
 & \ \ \ \ \ \ +\big(e^{a_{j+n+1}}\otimes((e^{a_{j}},\stackrel{\nearrow}{\cdots},e^{a_{i-1}})\shuffle(e^{a_{j+n}},\stackrel{\searrow}{\cdots},e^{a_{i+2}}))\otimes[e^{a_{i}},e^{a_{i+1}}]\big)\\
 & \ \ \ \ \ \ (t_{\sigma(m-n-1)},\cdots,t_{\sigma(m)})E_{j,j+n+2}\\
 & =\sum_{j}\sum_{i=1}^{m-1}(-1)^{j+n-i}\big(((e^{a_{j}},\stackrel{\nearrow}{\cdots},e^{a_{i-1}})\shuffle(e^{a_{j+n+1}},\stackrel{\searrow}{\cdots},e^{a_{i+2}}))\otimes[e^{a_{i}},e^{a_{i+1}}]\big)\\
 & \ \ \ \ \ \ (t_{\sigma(m-n-1)},\cdots,t_{\sigma(m)})E_{j,j+n+2}=G(n+1),
\end{align*}
where the first equality follows from the relation (\ref{eq:ShufRel}).
This completes the induction step. 
\end{proof}

\subsubsection*{Computation of $(\log S(B))^{I}$ }

Before stating the second key lemma, we derive a general formula for
basic Lie elements which may  be of independent interest. We consider
the tensor algebra $T((\mathbb{C}^{d}))$. Let $J=(j_{1},\cdots,j_{m})$
be a given word. We define 
\[
{\rm e}_{[J]}\triangleq[{\rm e}_{j_{1}},[{\rm e}_{j_{2}},\cdots[{\rm e}_{j_{m-1}},{\rm e}_{j_{m}}]]],\ \text{e}_{J}\triangleq{\rm e}_{j_{1}}\otimes\cdots\otimes{\rm e}_{j_{m}}
\]
respectively. Given a word $K=(k_{1},\cdots,k_{r})\subseteq(1,\cdots,m),$
we denote 
\[
J_{K}\triangleq(j_{k_{1}},\cdots,j_{k_{r}}),\ \overleftarrow{J_{K}}\triangleq(j_{k_{r}},\cdots,j_{k_{1}}).
\]
As a convention, we also set ${\rm e}_{\emptyset}\triangleq{\bf 1}$
and $\overleftarrow{\emptyset}\triangleq\emptyset$.
\begin{lem}
\label{lem:LieMon}For any word $J=(j_{1},\cdots,j_{m}),$ one has
\begin{equation}
{\rm e}_{[J]}=\sum_{K\subseteq(1,\cdots,m-2)}(-1)^{|K|}{\rm e}_{J\backslash(J_{K},j_{m-1},j_{m})}\otimes[{\rm e}_{j_{m-1}},{\rm e}_{j_{m}}]\otimes{\rm e}_{\overleftarrow{J_{K}}}.\label{eq:LieMon}
\end{equation}
\end{lem}
\begin{proof}
We prove the claim by induction. The case when $m=2$ is obvious.
Suppose that the claim is true for any word with length $\leqslant m-1$
and let $J=(j_{1},\cdots,j_{m})$. By the induction hypothesis, one
has
\begin{align*}
{\rm e}_{[J]}= & [{\rm e}_{j_{1}},{\rm e}_{[J']}]\ \ \ \ \ \ (J'\triangleq J\backslash\{j_{1}\})\\
= & \sum_{K\subseteq(2,\cdots,m-2)}(-1)^{|K|}{\rm e}_{(j_{1},J')\backslash(J'_{K},j_{m-1},j_{m})}\otimes[{\rm e}_{j_{m-1}},{\rm e}_{j_{m}}]\otimes{\rm e}_{\overleftarrow{J'_{K}}}\\
 & \ \ \ -\sum_{K\subseteq(2,\cdots,m-2)}(-1)^{|K|}{\rm e}_{J'\backslash(J'_{K},j_{m-1},j_{m})}\otimes[{\rm e}_{j_{m-1}},{\rm e}_{j_{m}}]\otimes{\rm e}_{\overleftarrow{(j_{1},J'_{K})}}\\
= & \sum_{L\subseteq(1,\cdots,m-2)}(-1)^{|L|}{\rm e}_{J\backslash(J_{L},j_{m-1},j_{m})}\otimes[{\rm e}_{j_{m-1}},{\rm e}_{j_{m}}]\otimes{\rm e}_{\overleftarrow{J_{L}}}
\end{align*}
The last equality follows from the observation that the two sums in
the second last equality correspond to the cases $1\notin L$ and
$1\in L$ respectively. 
\end{proof}
We are now able to establish the second key lemma for the proof of
Proposition \ref{prop:DmTilLDevForm}. This connects the logarithmic
signature coefficient $(\log S(B))^{I}$ with the same kind of shuffle
product identity appearing in Lemma \ref{lem:DevShuf}. Recall that
$B_{t}$ is the path in $\mathbb{C}^{d}$ defined by (\ref{eq:BPath})
with given $p_{1},\cdots,p_{d}\in\mathbb{C}$. To ease notation, we
write 
\[
dB_{{\bf t}_{\sigma}}^{J}\triangleq dB_{t_{\sigma(1)}}^{j_{1}}\cdots dB_{t_{\sigma(m)}}^{j_{m}}.
\]
We denote $\pi_{I}:T((\mathbb{C}^{d}))\rightarrow\mathbb{C}$ as the
projection map which extracts the coefficient of a tensor with respect
to ${\rm e}_{I}$.
\begin{lem}
\label{lem:PiIShuf}For any word $I=(i_{1},\cdots,i_{m})$, one has
\begin{align}
 & \pi_{I}\big(\sum_{J:|J|=m}\int_{\Delta_{m}}dB_{{\bf t}_{\sigma}}^{J}{\rm e}_{[J]}\big)\nonumber \\
 & =\sum_{i=1}^{m-1}(-1)^{m-1-i}\int_{\Delta_{m}}\big(((e^{a_{1}},\stackrel{\nearrow}{\cdots},e^{a_{i-1}})\shuffle(e^{a_{m}},\stackrel{\searrow}{\cdots},e^{a_{i+2}}))\nonumber \\
 & \ \ \ \ \ \ \ \ \ \otimes[e^{a_{i}},e^{a_{i+1}}]\big)(t_{\sigma(1)},\cdots,t_{\sigma(m)})dy_{t_{1}}\cdots dy_{t_{m}},\label{eq:PiIShuf}
\end{align}
where $a_{j}\triangleq p_{i_{j}}$ ($j=1,\cdots,m$).
\end{lem}
\begin{proof}
According to Lemma \ref{lem:LieMon}, one has 
\begin{equation}\label{eq:PIEJ}
\pi_{I}({\rm e}_{[J]}){\rm e}_{I}=\sum_{\textcolor{black}{K\subseteq(1,\cdots,m-2)}}(-1)^{|K|}\pi_{I}\big({\rm e}_{J\backslash(J_{K},j_{m-1},j_{m})}\otimes[{\rm e}_{j_{m-1}},{\rm e}_{j_{m}}]\otimes{\rm e}_{\overleftarrow{J_{K}}}\big).
\end{equation}
Let \textcolor{black}{$K=(k_{1},\cdots,k_{r})\subseteq(1,\cdots,m-2)$} be given fixed.
For any word $J$ with length $m$, one has \[\pi_{I}\big({\rm e}_{J\backslash(J_{K},j_{m-1},j_{m})}\otimes[{\rm e}_{j_{m-1}},{\rm e}_{j_{m}}]\otimes{\rm e}_{\overleftarrow{J_{K}}}\big)\neq 0\]
if and only if 
\[
(J\backslash(J_{K},j_{m-1},j_{m}),j_{m-1},j_{m},\overleftarrow{J_{K}})=I
\]
or
\[
(J\backslash(J_{K},j_{m-1},j_{m}),j_{m},j_{m-1},\overleftarrow{J_{K}})=I.
\]
The first (respectively, second) case comes with a plus (respectively,
minus) sign arising from the Lie bracket. As a result, the word $J$
in the first case is uniquely determined by 
\begin{equation}
\begin{array}{cccccccccccc}
J= & ( & \cdots & i_{m} & \cdots & i_{m-1} & \cdots & i_{m-r+1} & \cdots & i_{m-r-1} & i_{m-r} & )\\
 &  & \cdots & k_{1} & \cdots & k_{2} & \cdots & k_{r} & \cdots & m-1 & m
\end{array}\label{eq:JWord}
\end{equation}
where the ``$\cdots$'' positions are filled by $(i_{1},\cdots,i_{m-r-2})$
in its natural order. Respectively, the word $J$ in the second case
is determined by swapping $i_{m-r-1},i_{m-r}$ in the last two positions
in (\ref{eq:JWord}). We denote these two uniquely determined words
by $J_{1}(K)$ and $J_{2}(K)$ respectively. 

It follows from the above computation that 
\begin{align*}
\sum_{J:|J|=m}\int_{\Delta_{m}}dB_{{\bf t}_{\sigma}}^{J}\pi_{I}({\rm e}_{[J]}) & =\sum_{K\subseteq(1,\cdots,m-2)}(-1)^{|K|}\int_{\Delta_{m}}\big(dB_{{\bf t}_{\sigma}}^{J_{1}(K)}-dB_{{\bf t}_{\sigma}}^{J_{2}(K)}\big).
\end{align*}
From the explicit shapes of $J_{1}(K),J_{2}(K)$, it is not hard to
see that the above expression is precisely equal to 
\begin{align*}
\sum_{r=0}^{m-2}(-1)^{r}\int_{\Delta_{m}}\big(((e^{a_{1}}, & \stackrel{\nearrow}{\cdots},e^{a_{m-r-2}})\shuffle(e^{a_{m}},\stackrel{\searrow}{\cdots},e^{a_{m-r+1}}))\\
 & \otimes[e^{a_{m-r-1}},e^{a_{m-r}}]\big)(t_{\sigma(1)},\cdots,t_{\sigma(m)})dy_{t_{1}}\cdots dy_{t_{m}}.
\end{align*}
The desired relation follows from the change of indices $i\triangleq m-1-r$.
\end{proof}
We are now in a position to finish the proof of Proposition \ref{prop:DmTilLDevForm}.

\begin{proof}[Proof of Proposition \ref{prop:DmTilLDevForm}]

According to Proposition \ref{prop:DmTilLForm} and Lemma \ref{lem:DevShuf}
with $k=m-1$, one has 
\begin{align}
\hat{F}_{I}(D_{m}\tilde{L})= & \sum_{\sigma\in{\cal S}_{m}}\frac{(-1)^{e(\sigma)}}{m^{2}{m-1 \choose e(\sigma)}}\sum_{i=1}^{m-1}(-1)^{m-i-1}\nonumber \\
 & \ \ \ \int_{\Delta_{m}}\big(((e^{a_{1}},\stackrel{\nearrow}{\cdots},e^{a_{i-1}})\shuffle(e^{a_{m}},\stackrel{\searrow}{\cdots},e^{a_{i+2}}))\nonumber \\
 & \ \ \ \otimes[e^{a_{i}},e^{a_{i+1}}]\big)(t_{\sigma(1)},\cdots,t_{\sigma(m)})dy_{t_{1}}\cdots dy_{t_{m}}E_{1,m+1}.\label{eq:FDmPf}
\end{align}
On the other hand, the Chen-Strichartz formula gives that 
\begin{equation}
\log S(B)=\sum_{m=1}^{\infty}\sum_{\sigma\in{\cal S}_{m}}\frac{(-1)^{e(\sigma)}}{m^{2}{m-1 \choose e(\sigma)}}\sum_{J:|J|=m}\int_{\Delta_{m}}dB_{{\bf t}_{\sigma}}^{J}{\rm e}_{[J]}.\label{eq:LogSB}
\end{equation}
It follows from (\ref{eq:LogSB}) and Lemma \ref{lem:PiIShuf} that
the right hand side of (\ref{eq:FDmPf}) is precisely $(\log S(B))^{I}E_{1,m+1}$.
This completes the proof of the proposition.

\end{proof}

We conclude this part by presenting a general formula for the development
of $D_{k}\tilde{L}$, which will be important to us later on. Recall
that $S_{m}^{\gamma}$ is the analytic function defined by Definition
\ref{def:Sm}. Throughout the rest, we will omit the superscript $\gamma$
for simplicity. 
\begin{cor}
\label{cor:FDkSk}Let $a_{1},\cdots,a_{m}\in\mathbb{C}$ be given
fixed numbers. Consider the $\mathfrak{sl}_{m+1}(\mathbb{C})$-development
defined by (\ref{eq:FI}) where $A\in\mathfrak{h}$ is chosen to be
such that $[A,E_{k,k+1}]=a_{k}E_{k,k+1}$ for all $k=1,\cdots,m.$
Then one has 
\begin{equation}
\hat{F}(D_{k}\tilde{L})=\sum_{j=1}^{m-k+1}S_{k}(a_{j},\cdots,a_{j+k-1})E_{j,j+k}\label{eq:FDkSk}
\end{equation}
for all $k=1,\cdots,m.$
\end{cor}
\begin{proof}
The case when $k=m$ is just Proposition \ref{prop:DmTilLForm} applied
to the path 
\[
B_{t}\triangleq\big(\int_{0}^{t}e^{a_{1}x_{s}}dy_{s},\cdots,\int_{0}^{t}e^{a_{m}x_{s}}dy_{s}\big)\in\mathbb{C}^{m}
\]
with the word $I=(1,\cdots,m)$. For a general $k$, the same argument
as the proof of Lemma \ref{lem:DevShuf} gives that 
\begin{align}
\hat{F}(D_{k}\tilde{L})= & \sum_{\sigma\in{\cal S}_{k}}\frac{(-1)^{e(\sigma)}}{k^{2}{k-1 \choose e(\sigma)}}\sum_{j=1}^{m-k+1}\sum_{i=1}^{m-1}(-1)^{j+k-i}\nonumber \\
 & \ \ \ \int_{\Delta_{k}}\big(((e^{a_{j}},\stackrel{\nearrow}{\cdots},e^{a_{i-1}})\shuffle(e^{a_{j+k-1}},\stackrel{\searrow}{\cdots},e^{a_{i+2}}))\nonumber \\
 & \ \ \ \otimes[e^{a_{i}},e^{a_{i+1}}]\big)(t_{\sigma(1)},\cdots,t_{\sigma(k)})dy_{t_{1}}\cdots dy_{t_{k}}E_{j,j+k}.\label{eq:FDkPf}
\end{align}
In view of Lemma \ref{lem:PiIShuf}, after a change of indices $l\triangleq i-j+1$
the right hand side of (\ref{eq:FDkPf}) is precisely 
\[
\sum_{j=1}^{m-k+1}S_{k}(a_{j},\cdots,a_{j+k-1})E_{j,j+k}.
\]
The relation (\ref{eq:FDkSk}) thus follows. 
\end{proof}

\subsection{Proof of Theorem \ref{thm:IterIntCond}}

Our proof of Theorem \ref{thm:IterIntCond} is based on a key lemma
regarding the analyticity of path developments for the Hausdorff series.
Since its proof has a rather involved combinatorial nature, we decide to only state this analyticity lemma here and then use it to complete the proof
of Theorem \ref{thm:IterIntCond}. In Section \ref{subsec:KeyLem}
below, we will give the proof of this key lemma. 

Recall that the $n$-th Hausdorff series $H_{n}$ is defined by (\ref{eq:HnFormula}).
In our context, we need to consider a more specific series which is
defined in terms of the homogeneous projections of $\tilde{L}$.
\begin{defn}
Let $n\geqslant1$ and $k_{1},\cdots,k_{n}\geqslant1$ be given fixed.
We define 
\[
H_{n}(k_{1},\cdots,k_{n})\triangleq\frac{1}{n!}\big(H_{1}(k_{1})\partial_{{\rm e}_{1}}\big)\circ\cdots\circ\big(H_{1}(k_{n})\partial_{{\rm e}_{1}}\big)({\rm e}_{1}),
\]
where we used the shorthand notation $H_{1}(k_{i})\triangleq H_{1}(D_{k_{i}}\tilde{L},{\rm e}_{1})$.
It is clear that $H_{n}(k_{1},\cdots,k_{n})$ depends only on the
vector $K\triangleq(k_{1},\cdots,k_{n})$ and we will thus also
use the alternative shorthand notation $H_{K}$.
\end{defn}
\begin{rem}
In the above definition, we view $D_{k_{i}}\tilde{L}$ as a fixed
symbol and the derivation $H_{1}(k_{j})\partial_{{\rm e}_{1}}$ does
not act on it. As a result, $H_{n}(k_{1},\cdots,k_{n})$ is expressed
as a formal Lie series over the $n+1$ independent symbols $\{{\rm e}_{1},D_{k_{1}}\tilde{L},\cdots,D_{k_{n}}\tilde{L}\}$.
It is then regarded as a Lie series over $\mathbb{R}^{2}$ through
the substitution $D_{k_{i}}\tilde{L}\in{\cal L}((\langle{\rm e}_{1},{\rm e}_{2}\rangle))$.
\end{rem}
\begin{lem}
\label{lem:AnaLem}Let $N\geqslant2$ be a given fixed integer. Suppose
that $S_{n}(a_{1},\cdots,a_{n})=0$ for all $n\leqslant N-1$ and
all non-degenerate sequences $(a_{1},\cdots,a_{n})$ with $a_{j}\in2\pi i\mathbb{Z}$.
Let $(b_{1}^{*},\cdots,b_{N}^{*})$ be a fixed non-degenerate sequence
with $b_{j}\in2\pi i\mathbb{Z}$. Given $w\in\mathbb{C}$, we define
the path development $F_{w}:\mathbb{C}^{2}\rightarrow\mathfrak{sl}_{N+1}(\mathbb{C})$
by 
\begin{equation}
F_{w}({\rm e}_{1})\triangleq A_{w},\ F_{w}({\rm e}_{2})\triangleq E_{12}+\cdots+E_{N,N+1},\label{eq:Fw}
\end{equation}
where $A_{w}\in\mathfrak{h}$ is chosen to satisfy 
\begin{equation}
[A_{w},E_{k,k+1}]=(b_{k}^{*}+w)E_{k,k+1}\ \ \ (k=1,\cdots,N).\label{eq:FwLieRel}
\end{equation}
Then for any vector $K=(k_{1},\cdots,k_{n})$ satisfying $|K|=N,n\geqslant2$ or $|K|\leqslant N-1,n\geqslant1$ and any $j=1,\cdots,N,$ there exists a meromorphic function $\Psi_{K,j}:\mathbb{C}^{|K|}\rightarrow\mathbb{C}$
(i.e. the quotient of two holomorphic functions) in $|K|$ complex variables,
such that 
\[
\hat{F}_{w}(H_{K})=\sum_{j}\Psi_{K,j}(b_{j}^{*}+w,\cdots,b_{j+|K|-1}^{*}+w)E_{j,j+|K|}
\]
and the function $w\mapsto\Psi_{K,j}(b_{j}^{*}+w,\cdots,b_{j+|K|-1}^{*}+w)$
is analytic in a neighbourhood of $w=0$ for all $(K,j)$. Here $\hat{F}_{w}$
is the induced Lie homomorphism on ${\cal L}((\mathbb{C}^{2}))$ (cf.
(\ref{eq:HatF})).
\end{lem}
%
\begin{comment}
\textcolor{purple}{(Sheng: I suggest that in Lemma 5.15, given that $S_n(a_1,\cdots,a_n)=0$ for all $n\leq N-1$, we make the conclusion for $|K|=N$, $n\geq 2$ or $|K|\leq N-1$, $n\geq1$, but in the induction argument performed in \eqref{eq:FSymProd}, we only focus on $|K|=N$, $n\geq2$. Because (a) for the case $|K|\leq N-1$, $n\geq2$, it is from the induction hypothesis directly, (b) for the case $|K|\leq N-1$, $n=1$, we have $\hat{F}_{w}(H_{K})=\sum_iS^i_{|K|}E_{i,i+|K|}=0$ due to the condition on $S_n$. Perhaps elaborate on the logic when we apply induction argument before \eqref{eq:FSymProd}, since I was very confused by the target of induction when I read that part firstly.)}
\end{comment}

Now we prove Theorem \ref{thm:IterIntCond} presuming
the correctness of Lemma \ref{lem:AnaLem}.

\begin{proof}[Proof of Theorem \ref{thm:IterIntCond}]

We argue by induction on the length $m$ of the sequence $(a_{1},\cdots,a_{m})$.
The cases when $m=1,2$ are just Theorem \ref{thm:LineInt} and Corollary
\ref{cor:DoubIntCond} respectively. Suppose that the claim is true
for all non-degenerate sequences with length $\leqslant N-1.$ Let
$(b_{1}^{*},\cdots,b_{N}^{*})$ be a non-degenerate sequence with
$b_{j}^{*}\in2\pi i\mathbb{Z}$ for all $j$. Our goal is to show
that $S_{N}(b_{1}^{*},\cdots,b_{N}^{*})=0$. 

First of all, by using the degree $N$ version of (\ref{eq:D1D2})
one finds that 
\begin{equation}
D_{N}L=D_{N}\big(\sum_{n=1}^{\infty}H_{n}(\tilde{L},{\rm e}_{1})\big)=H_{1}(N)+\sum_{n=2}^{N}\sum_{k_{1}+\cdots+k_{n}=N}H_{n}(k_{1},\cdots,k_{n}).\label{eq:DN}
\end{equation}
Given $w\in\mathbb{C},$ let $F_{w}:\mathbb{C}^{2}\rightarrow\mathfrak{sl}_{N+1}(\mathbb{C})$
denote the path development defined by (\ref{eq:Fw}) under the condition
(\ref{eq:FwLieRel}). By applying $\hat{F}_{w}$ to both sides of
(\ref{eq:DN}), one finds that 
\begin{align*}
\hat{F}_{w}(D_{N}L)= & H_{1}(\hat{F}_{w}(D_{N}\tilde{L}),A)+\sum_{n=2}^{N}\sum_{k_{1}+\cdots+k_{n}=N}\hat{F}_{w}(H_{n}(k_{1},\cdots,k_{n}))\\
= & S_{N}(b_{1}^{*}+w,\cdots,b_{N}^{*}+w)\phi(b_{1}^{*}+\cdots+b_{N}^{*}+Nw)E_{1,N+1}\\
 & \ \ \ +\sum_{n=2}^{N}\sum_{k_{1}+\cdots+k_{n}=N}\hat{F}_{w}(H_{n}(k_{1},\cdots,k_{n})).
\end{align*}
The second equality follows from Corollary \ref{cor:FDkSk} and the
definition (\ref{eq:H1Formula}) of $H_{1}$, where $\phi(z)\triangleq\frac{z}{e^{z}-1}$.
Since the induction hypothesis holds for all non-degenerate sequences
with length $<N$, one can apply Lemma \ref{lem:AnaLem} to conclude
that the function 
\[
w\mapsto\hat{F}_{w}(H_{n}(k_{1},\cdots,k_{n}))
\]
is analytic near $w=0$. Note that $w\mapsto\hat{F}_{w}(D_{N}L)$
is also an analytic function (on the whole space $\mathbb{C}$ due
to the infinite R.O.C. for $D_{N}L$). It follows that the function
\[
w\mapsto S_{N}(b_{1}^{*}+w,\cdots,b_{N}^{*}+w)\phi(b_{1}^{*}+\cdots+b_{N}^{*}+Nw)
\]
must be analytic near $w=0$. Recall that $b^{*}\triangleq b_{1}^{*}+\cdots+b_{N}^{*}$
is a nonzero integer multiple of $2\pi i$ and is thus a pole of $\phi$.
As a consequence, the function $w\mapsto S_{N}(b_{1}^{*}+w,\cdots,b_{N}^{*}+w)$
must vanish at $w=0$. In other words, one has $S_{N}(b_{1}^{*},\cdots,b_{N}^{*})=0$,
which completes the induction step. 

\end{proof}

\subsection{\label{subsec:KeyLem}The analyticity lemma }

It remains to prove the analyticity lemma which will be the main task of this subsection. 

\subsubsection{A recursive formula for the Hausdorff series}

Our strategy relies on induction on the total degree $|K|$ of the
vector $K=(k_{1},\cdots,k_{n}).$ For this purpose, we shall derive
a recursive formula for computing $H_{n}(k_{1},\cdots,k_{n})$ in
terms of ``symmetrised products'' of $H_{m}$'s ($m<n$). We first define
such type of products. 
\begin{defn}
\label{def:SymProd}Let $A_{1},\cdots,A_{r},{\rm e}_{1},v$ be given
symbols. We define the linear operator 
\[
A_{1}\partial_{{\rm e}_{1}}\hat{\otimes}_{{\rm s}}\cdots\hat{\otimes}_{{\rm s}}A_{r}\partial_{{\rm e}_{1}}:T((\langle{\rm e}_{1},v\rangle))\rightarrow T((\langle{\rm e}_{1},v,A_{1},\cdots,A_{r}\rangle))
\]
in the following way. Let $\mathfrak{f}_{1}\otimes\cdots\otimes\mathfrak{f}_{N}$
be a given monomial where $\mathfrak{f}_{i}={\rm e}_{1}$ or $v$.
We first define
\begin{align*}
 & A_{1}\partial_{{\rm e}_{1}}\hat{\otimes}_{{\rm }}\cdots\hat{\otimes}A_{r}\partial_{{\rm e}_{1}}(\mathfrak{f}_{1}\otimes\cdots\otimes\mathfrak{f}_{N})\\
 & \triangleq\sum\mathfrak{f}_{1}\otimes\cdots\otimes\mathfrak{f}_{i_{1}-1}\otimes A_{1}\otimes\mathfrak{f}_{i_{1}+1}\otimes\cdots\otimes\mathfrak{f}_{i_{r}-1}\otimes A_{r}\otimes\mathfrak{f}_{i_{r}+1}\otimes\cdots\otimes\mathfrak{f}_{N},
\end{align*}
where the summation is taken over all $r$-subsets $\{i_{1},\cdots,i_{r}\}$
with $\mathfrak{f}_{i_{j}}={\rm e}_{1}$. In other words, the action
is given by replacing $r$ ${\rm e}_{1}$'s by $(A_{1},\cdots,A_{r})$
and summing over all such possibilities. The definition is extended
to the tensor algebra $T((\langle{\rm e}_{1},v\rangle))$ by linearity.
We then set 
\[
A_{1}\partial_{{\rm e}_{1}}\hat{\otimes}_{{\rm s}}\cdots\hat{\otimes}_{{\rm s}}A_{r}\partial_{{\rm e}_{1}}\triangleq\sum_{\sigma\in{\cal S}_{r}}A_{\sigma(1)}\partial_{{\rm e}_{1}}\hat{\otimes}_{{\rm }}\cdots\hat{\otimes}A_{\sigma(r)}\partial_{{\rm e}_{1}}.
\] 
\end{defn}
The following simple property of the operator $A_{1}\partial_{{\rm e}_{1}}\hat{\otimes}_{{\rm s}}\cdots\hat{\otimes}_{{\rm s}}A_{r}\partial_{{\rm e}_{1}}$
will be useful to us. 
\begin{lem}
\label{lem:SymAd}One has 
\begin{align*}
 & A_{1}\partial_{{\rm e}_{1}}\hat{\otimes}_{{\rm s}}\cdots\hat{\otimes}_{{\rm s}}A_{r}\partial_{{\rm e}_{1}}({\rm ad}_{{\rm e}_{1}}^{l}(v))\\
 & =\sum_{\sigma\in{\cal S}_{r}}\sum_{\substack{\substack{\xi_{1}+\cdots+\xi_{r+1}}
=l-r}
}{\rm ad}_{{\rm e}_{1}}^{\xi_{1}}\circ{\rm ad}_{A_{\sigma(1)}}\circ{\rm ad}_{{\rm e}_{1}}^{\xi_{2}}\circ\cdots\circ{\rm ad}_{{\rm e}_{1}}^{\xi_{r}}\circ{\rm ad}_{A_{\sigma(r)}}\circ{\rm ad}_{{\rm e}_{1}}^{\xi_{r+1}}(v)
\end{align*}
for any $l\geqslant r\geqslant1$.
\end{lem}
\begin{proof}
This follows by induction on $r$ based on Definition \ref{def:SymProd}
(one first treats the case when $l=r$ by induction and then the case
when $l\geqslant r$ follows by another step of induction from cases
$(l-1,r)$ and $(l-1,r-1)$). 
\end{proof}
To get some feeling about the shape of $H_{n}(k_{1},\cdots,k_{n})$,
we first look at the small-$n$ cases.
\begin{example}
We have seen the computation for $H_{2}$ in (\ref{eq:H2D1}). Let
us consider the case when $n=3$. By definition, one has 
\begin{align*}
 & H_{3}(k_{1},k_{2},k_{3})\\
 & =\frac{1}{3!}\big(H_{1}(k_{1})\partial_{{\rm e}_{1}}\big)\big(\sum_{l}\frac{B_{l}}{l!}(H_{1}(k_{2})\partial_{{\rm e}_{1}})({\rm ad}_{{\rm e}_{1}}^{l}(D_{k_{3}}\tilde{L}))\big)\\
 & =\frac{1}{3!}\big(H_{1}(k_{1})\partial_{{\rm e}_{1}}\big)\big(\sum_{l}\frac{B_{l}}{l!}\sum_{\xi_{1}+\xi_{2}=l-1}{\rm ad}_{{\rm e}_{1}}^{\xi_{1}}\circ{\rm ad}_{H_{1}(k_{2})}\circ{\rm ad}_{{\rm e}_{1}}^{\xi_{2}}(D_{k_{3}}\tilde{L})\big)\\
 & =\frac{1}{3!}\sum_{\tau\in{\cal S}_{2}}\sum_{l}\frac{B_{l}}{l!}\sum_{\xi_{1}+\xi_{2}+\xi_{3}=l-2}{\rm ad}_{{\rm e}_{1}}^{\xi_{1}}\circ{\rm ad}_{H_{1}(k_{\tau(1)})\circ}{\rm ad}_{{\rm e}_{1}}^{\xi_{2}}\circ{\rm ad}_{H_{1}(k_{\tau(2)})}\circ{\rm ad}_{{\rm e}_{1}}^{\xi_{3}}(D_{k_{3}}\tilde{L})\\
 & \ \ \ +\frac{2!}{3!}\sum_{l}\frac{B_{l}}{l!}\sum_{\xi_{1}+\xi_{2}=l-1}{\rm ad}_{{\rm e}_{1}}^{\xi_{1}}\circ{\rm ad}_{H_{2}(k_{1},k_{2})}\circ{\rm ad}_{{\rm e}_{1}}^{\xi_{2}}(D_{k_{3}}\tilde{L}).
\end{align*}
A crucial observation is that the above expression can be rewritten
in terms of the symmetrised tensor product in Definition \ref{def:SymProd}.
In fact, according to Lemma \ref{lem:SymAd} one has 
\begin{align*}
H_{3}(k_{1},k_{2},k_{3})=\frac{1}{3!} & \big((H_{1}(k_{1})\partial_{{\rm e}_{1}})\hat{\otimes}_{{\rm s}}(H_{1}(k_{2})\partial_{{\rm e}_{1}})\big)\big(H_{1}(k_{3})\big)\\
 & \ \ \ +\frac{2!}{3!}\big(H_{2}(k_{1},k_{2})\partial_{{\rm e}_{1}}\big)(H_{1}(k_{3})).
\end{align*}
Our aim is to generalise the above formula to arbitrary $H_{n}(k_{1},\cdots,k_{n})$.
\end{example}
To state the recursive formula for general $H_{n}(k_{1},\cdots,k_{n})$,
we need to introduce one more definition to ease notation.
\begin{defn}
\label{def:HPDef}Let $P=\{I_{1},\cdots,I_{r}\}$ be a given (unordered)
partition of $\{1,\cdots,n\}$ (i.e. $I_{p}\neq\emptyset,$ $I_{p}\cap I_{q}=\emptyset$
for $p\neq q$ and $\cup I_{p}=\{1,\cdots,n\}$). We define 
\[
\hat{H}_{P}(k_{1},\cdots,k_{n})\triangleq H_{K_{1}}\partial_{{\rm e}_{1}}\hat{\otimes}_{{\rm s}}\cdots\hat{\otimes}_{{\rm s}}H_{K_{r}}\partial_{{\rm e}_{1}},
\]
where $I_{p}=(i_{1}^{p}<\cdots<i_{l_{p}}^{p})$ and $K_{p}\triangleq(k_{i_{1}^{p}},\cdots,k_{i_{l_{p}}^{p}})$
($p=1,\cdots,r$). 
\end{defn}
The main recursive formula for $H_{n}(k_{1},\cdots,k_{n})$  is stated
as follows. 
\begin{proposition}
\label{prop:HnRecur}For any $n\geqslant1$ and $k_{1},\cdots,k_{n}\geqslant1$,
one has 
\[
H_{n}(k_{1},\cdots,k_{n})=\sum_{P=\{I_{1},\cdots,I_{r}\}}\frac{l_{1}!\cdots l_{r}!}{n!}\hat{H}_{P}(k_{1},\cdots,k_{n-1})(H_{1}(k_{n})),
\]
where the summation is taken over all (unordered) partitions $\{I_{1},\cdots,I_{r}\}$
of the set $\{1,\cdots,n-1\}$ and $l_{s}$ denotes the cardinality
of $I_{s}$ ($s=1,\cdots,r$). 
\end{proposition}
\begin{proof}
We prove the claim by induction on $n$. To simplify notation, we
get rid of the factorials by setting $\bar{H}_{n}\triangleq n!H_{n}$
and $\hat{\bar{H}}_{P}\triangleq l_{1}!\cdots l_{r}!\hat{H}_{P}.$
The base case $n=1$ is clear. Suppose that the claim is true for
$\bar{H}_{n-1}$. By the definition of $H_{n}(k_{1},\cdots,k_{n})$
and the induction hypothesis, one has 
\begin{align}
\bar{H}_{n}(k_{1},\cdots,k_{n}) & =(H_{1}(k_{1})\partial_{{\rm e}_{1}})\big(\bar{H}_{n-1}(k_{2},\cdots,k_{n})\big)\nonumber \\
 & =(H_{1}(k_{1})\partial_{{\rm e}_{1}})\big[\sum_{P}\big(\hat{\bar{H}}_{P}(k_{2},\cdots,k_{n-1})(H_{1}(k_{n}))\big)\big],\label{eq:HnRecFormPf1}
\end{align}
where the summation is taken over all partitions of $\{2,\cdots,n-1\}$. 

Given any such partition $P=\{I_{1},\cdots,I_{r}\}$, by using Lemma
\ref{lem:SymAd} (and the notation in Definition \ref{def:HPDef})
one can write
\begin{align*}
 & \hat{\bar{H}}_{P}(k_{2},\cdots,k_{n-1})(H_{1}(k_{n}))\\
 & =\sum_{\tau\in{\cal S}_{r}}\sum_{l}\frac{B_{l}}{l!}\sum_{\xi_{1}+\cdots+\xi_{r+1}=l-r}{\rm ad}_{{\rm e}_{1}}^{\xi_{1}}\circ{\rm ad}_{\bar{H}_{K_{\tau(1)}}}\circ{\rm ad}_{{\rm e}_{1}}^{\xi_{2}}\circ\\
 & \ \ \ \ \ \ \ \ \ \cdots\circ{\rm ad}_{\bar{H}_{K_{\tau(r)}}}\circ{\rm ad}_{{\rm e}_{1}}^{\xi_{r+1}}(D_{k_{n}}\tilde{L}).
\end{align*}
As a derivation, the outer $H_{1}(k_{1})\partial_{{\rm e}_{1}}$ action
in (\ref{eq:HnRecFormPf1}) will either apply to the ${\rm ad}_{{\rm e}_{1}}$'s
or to the $\bar{H}_{K_{\tau(i)}}$'s. The former case yields 
\[
\hat{\bar{H}}_{\{1\}\oplus P}(k_{1},\cdots,k_{n-1})(H_{1}(k_{n}))
\]
where $\{1\}\oplus P$ is the partition $\{\{1\},I_{1},\cdots,I_{r}\}$
of $\{1,\cdots,n-1\}$, while the latter case yields 
\[
\sum_{s=1}^{r}\hat{\bar{H}}_{\{I_{1},\cdots,\{1\}\cup I_{s},\cdots,I_{r}\}}(k_{1},\cdots,k_{n-1})(H_{1}(k_{n}))
\]
where $\{I_{1},\cdots,\{1\}\cup I_{s},\cdots,I_{r}\}$ is the partition
of $\{1,\cdots,n-1\}$ obtained by adding the element ``$1$'' into
$I_{s}$. As a consequence, one finds that 
\begin{align*}
\bar{H}_{n}(k_{1},\cdots,k_{n})= & \sum_{P=\{I_{1},\cdots,I_{r}\}}\big[\hat{\bar{H}}_{\{1\}\oplus P}(k_{1},\cdots,k_{n-1})(H_{1}(k_{n}))\\
 & \ \ \ +\sum_{s=1}^{r}\hat{\bar{H}}_{\{I_{1},\cdots,\{1\}\cup I_{s},\cdots,I_{r}\}}(k_{1},\cdots,k_{n-1})(H_{1}(k_{n}))\big].
\end{align*}
The right hand side is exactly the sum of $\hat{\bar{H}}_{Q}(k_{1},\cdots,k_{n-1})(H_{1}(k_{n}))$
over all possible partitions of $\{1,\cdots,n-1\}$. Indeed, any such
partition $Q$ corresponds to a partition $P$ of $\{2,\cdots,n-1\}$
together with a specific way of adding the element ``$1$'' (either
outside $P$ or into one of the members of $P$) and vice versa. This
completes the induction step. 
\end{proof}

\subsubsection{A combinatorial identity }

The following combinatorial identity plays a key role in the proof.
It allows one to express certain quotients defined by non-consecutive sums of a sequence in terms of consecutive sums. 
\begin{lem}
\label{lem:ComLem}Let $0\leqslant s\leqslant R$ be given fixed and
let $\{c_{k}\}_{1\leqslant k\leqslant R+1}$ be a given sequence of
numbers. Then one has 
\begin{equation}
\sum_{\eta\in{\cal T}(s,R-s)}\frac{1}{c_{\eta^{-1}(2)}^{\eta^{-1}(2)}c_{\eta^{-1}(2)}^{\eta^{-1}(3)}\cdots c_{\eta^{-1}(2)}^{\eta^{-1}(R+1)}}=\big(\prod_{k=1}^{s}c_{k}^{s}\big)^{-1}\big(\prod_{k=s+2}^{R+1}c_{s+2}^{k}\big)^{-1}.\label{eq:ComLem}
\end{equation}
Here ${\cal T}(s,R-s)$ denotes the collection of permutations $\eta\in{\cal S}_{R+1}$
such that 
\begin{equation}
\eta(s+1)=1;\ \eta(s)<\cdots<\eta(1);\ \eta(s+2)<\cdots<\eta(R+1).\label{eq:T}
\end{equation}
We also denote 
\begin{equation}
c_{\eta^{-1}(2)}^{\eta^{-1}(k)}\triangleq c_{\eta^{-1}(2)}+c_{\eta^{-1}(3)}+\cdots+c_{\eta^{-1}(k)},\ c_{p}^{q}\triangleq c_{p}+c_{p+1}+\cdots+c_{q},\label{eq:CSum}
\end{equation}
and conventionally we set $c_{1}^{0}=c_{R+2}^{R+1}=1$.
\end{lem}
\begin{proof}
The case when $R=2,3$ can be checked explicitly. Suppose that the
claim is true for $R-1$. Let $0\leqslant s\leqslant R$ and $\{c_{k}\}_{1\leqslant k\leqslant R+1}$
be given . By the definition of $\eta\in{\cal T}(s,R-s),$ either
$\eta(1)=R+1$ or $\eta(R+1)=R+1.$ One can thus write
\[
\sum_{\eta\in{\cal T}(s,R-s)}\frac{1}{c_{\eta^{-1}(2)}^{\eta^{-1}(2)}c_{\eta^{-1}(2)}^{\eta^{-1}(3)}\cdots c_{\eta^{-1}(2)}^{\eta^{-1}(R+1)}}=\sum_{\eta:\eta(1)=R+1}+\sum_{\eta:\eta(R+1)=R+1}.
\]
According to the induction hypothesis,
\[
\sum_{\eta:\eta(1)=R+1}\frac{1}{c_{\eta^{-1}(2)}^{\eta^{-1}(2)}c_{\eta^{-1}(2)}^{\eta^{-1}(3)}\cdots c_{\eta^{-1}(2)}^{\eta^{-1}(R+1)}}=\frac{1}{(c_{1}^{s}+c_{s+2}^{R+1})}\times\big(\prod_{k=2}^{s}c_{k}^{s}\big)^{-1}\big(\prod_{k=s+2}^{R+1}c_{s+2}^{k}\big)^{-1}
\]
and similarly
\[
\sum_{\eta:\eta(1)=R+1}\frac{1}{c_{\eta^{-1}(2)}^{\eta^{-1}(2)}c_{\eta^{-1}(2)}^{\eta^{-1}(3)}\cdots c_{\eta^{-1}(2)}^{\eta^{-1}(R+1)}}=\frac{1}{(c_{1}^{s}+c_{s+2}^{R+1})}\times\big(\prod_{k=1}^{s}c_{k}^{s}\big)^{-1}\big(\prod_{k=s+2}^{R}c_{s+2}^{k}\big)^{-1}.
\]
It is easily seen that the sum of the above two expressions is equal
to the right hand side of (\ref{eq:ComLem}).
\end{proof}

\subsubsection{Proof of Lemma \ref{lem:AnaLem}}

We are now in a position to develop the precise proof of Lemma \ref{lem:AnaLem}.
We first recall the standing assumptions of the lemma which will be
imposed throughout the rest of this subsection. 

\vspace{2mm}\noindent \textbf{Assumption (${\bf A}_{N-1}$)}. Let
$N\geqslant2$ be given fixed. We assume that $S_{m}(a_{1},\cdots,a_{m})=0$
for all $m\leqslant N-1$ and all non-degenerate sequences $(a_{1},\cdots,a_{m})$
with $a_{j}\in2\pi i\mathbb{Z}$. 

\vspace{2mm}\noindent Now let $(b_{1}^{*},\cdots,b_{N}^{*})$ be
a given fixed non-degenerate sequence with $b_{j}^{*}\in2\pi i\mathbb{Z}$.
For each $w\in\mathbb{C}$, we consider the path development $F_{w}$
into $\mathfrak{sl}_{N+1}(\mathbb{C})$ defined by (\ref{eq:Fw},
\ref{eq:FwLieRel}). Recall that $\hat{F}_{w}$ is the induced Lie
homomorphism $\hat{F}_{w}$ on ${\cal L}((\mathbb{C}^{2}))$.

We are going to prove Lemma \ref{lem:AnaLem} by induction on the
degree $|K|=k_{1}+\cdots+k_{n}$ of the vector $K=(k_{1},\cdots,k_{n}).$
To be precise, we assume as the induction hypothesis that for any
vector $K$ with degree $|K|<k\leqslant N$ and $j=1,\cdots,N$, there
exists a meromorphic function $\Psi_{K,j}$ in $|K|$ variables such
that 

\begin{equation}
\hat{F}_{w}(H_{K})=\sum_{j}\Psi_{K,j}(b_{j}^{*}+w,\cdots,b_{j+|K|-1}^{*}+w)E_{j,j+|K|}.\label{eq:IndHypo}
\end{equation}
In addition, the function 
\[
w\mapsto\Psi_{K,j}(b_{j}^{*}+w,\cdots,b_{j+|K|-1}^{*}+w)
\]
is analytic near $w=0$. We want to prove the same assertion for any vector
$K=(k_1,\cdots,k_n)$ of degree $|K|=k$. Here we should emphasise that $n\geqslant 2$ if $k=N$ and $n\geqslant 1$ if $k<N$. 

We first consider the easier case when $n=1$ and $k<N$. According to Corollary \ref{cor:FDkSk} and the $H_1$-formula (\ref{eq:H1Formula}), one has 
\begin{align*}
\hat{F}_{w}(H_{1}(k)) & =\sum_{l}\frac{B_{l}}{l!}{\rm ad}_{A}^{l}\big(\hat{F}_{w}(D_{k}\tilde{L})\big)\\
 & =\sum_{j}S_{k}(b_{j}^{*}+w,\cdots,b_{j+k-1}^{*}+w)\sum_{l}\frac{B_{l}}{l!}(b_{j}^{*}+\cdots+b_{j+k-1}^{*}+kw)^{l}E_{j,j+k}\\
 & =\sum_{j}S_{k}(b_{j}^{*}+w,\cdots,b_{j+k-1}^{*}+w)\phi(b_{j}^{*}+\cdots+b_{j+k-1}^{*}+kw)E_{j,j+k}.
\end{align*}Since $k<N$, one knows from Assumption ($\mathbf{A}_{N-1}$) that $S_k(b_j^*,\cdots,b^*_{j+k-1})=0$. On the other hand, $w=0$ is a simple pole of the function $\phi(b_{j}^{*}+\cdots+b_{j+k-1}^{*}+kw)$. As a result, the function\[
\Psi_{(k),j}(w)\triangleq S_{k}(b_{j}^{*}+w,\cdots,b_{j+k-1}^{*}+w)\phi(b_{j}^{*}+\cdots+b_{j+k-1}^{*}+kw)
\]is analytic at $w=0$. This proves the assertion in the current case.

\vspace{2mm} We now focus on the more difficult case when $n\geqslant 2$ ($k\leqslant N$). According to Proposition \ref{prop:HnRecur}, it suffices
to prove the following claim.

\vspace{2mm}\noindent \textit{Claim}. Let $K_{1},\cdots,K_{r}$ be
given vectors (of positive integers) and let $n_{0}\geqslant1$ be
such that $|K_{r}|+\cdots+|K_{1}|+n_{0}=k\leqslant N$. Then one can
write 
\begin{align}
 & \hat{F}_{w}\big(H_{K_{r}}\partial_{{\rm e}_{1}}\hat{\otimes}_{{\rm s}}\cdots\hat{\otimes}_{{\rm s}}H_{K_{1}}\partial_{{\rm e}_{1}}(H_{1}(n_{0}))\big)\nonumber \\
 & =\sum_{j}\Phi_{K_{1},\cdots,K_{r},n_{0},j}(b_{j}^{*}+w,\cdots,b_{j+k-1}^{*}+w)E_{j,j+k},\label{eq:FSymProd}
\end{align}
where the $\Phi_{K_{1},\cdots,K_{r},n_{0},j}$'s are suitable meromorphic
functions in $k$ variables whose one-dimensional reductions
\[
w\mapsto\Phi_{K_{1},\cdots,K_{r},n_{0},j}(b_{j}^{*}+w,\cdots,b_{j+k-1}^{*}+w)
\]
are all analytic near $w=0$.

\vspace{2mm} By using Lemma \ref{lem:SymAd}, the $H_{1}$-formula
(\ref{eq:H1Formula}) and Corollary \ref{cor:FDkSk}, one can first
rewrite the left hand side of (\ref{eq:FSymProd}) as 
\begin{align}
 & \hat{F}_{w}\big(H_{K_{r}}\partial_{{\rm e}_{1}}\hat{\otimes}_{{\rm s}}\cdots\hat{\otimes}_{{\rm s}}H_{K_{1}}\partial_{{\rm e}_{1}}(H_{1}(n_{0}))\big)\nonumber \\
 & =\sum_{l=0}^{\infty}\frac{B_{l}}{l!}\sum_{\sigma\in{\cal S}_{r}}\sum_{\xi_{0}+\xi_{1}+\cdots+\xi_{r}=l-r}{\rm ad}_{A}^{\xi_{r}}\circ{\rm ad}_{\hat{F}_{w}(H_{K_{\sigma(r)}})}\circ{\rm ad}_{A}^{\xi_{r-1}}\nonumber \\
 & \ \ \ \ \ \ \circ\cdots\circ{\rm ad}_{A}^{\xi_{1}}\circ{\rm ad}_{\hat{F}_{w}(H_{K_{\sigma(1)}})}\circ{\rm ad}_{A}^{\xi_{0}}\big(\hat{F}_{w}(D_{n_{0}}\tilde{L})\big).\label{eq:MainSeries}
\end{align}
The main effort is to evaluate the above summation over the $\xi_{i}$'s
by using the induction hypothesis on the $\hat{F}_{w}(H_{K_{\sigma(j)}})$'s
(note that $|K_{\sigma(j)}|<k$). We first state the key lemma for
this purpose and then explain the heavy notation involved. 
\begin{lem}
\label{lem:MainSeriesKeyLem}Let $L_{1},\cdots,L_{r}$ be given vectors
of positive integers and let $n_{0}\geqslant1$. Suppose that $|L_{1}|+\cdots+|L_{r}|+n_{0}=k\leqslant N.$
Then one has 
\begin{align}
 & \sum_{\xi_{0}+\xi_{1}+\cdots+\xi_{r}=l-r}{\rm ad}_{A}^{\xi_{r}}\circ{\rm ad}_{\hat{F}_{w}(H_{L_{r}})}\circ\cdots\circ{\rm ad}_{A}^{\xi_{1}}\circ{\rm ad}_{\hat{F}_{w}(H_{L_{1}})}\circ{\rm ad}_{A}^{\xi_{0}}\big(\hat{F}_{w}(D_{n_{0}}\tilde{L})\big)\nonumber \\
 & =\sum_{i}\sum_{\mu=1}^{r+1}\Theta_{L_{\mu-1},\cdots,L_{1},n_{0}}^{i}\cdot(b_{i}^{i+n_{0}^{\mu-1}-1})^{l+1-\mu}\nonumber \\
 & \ \ \ \ \ \ \ \ \ \sum_{I\in{\cal W}_{\mu,r}}\frac{(-1)^{\varepsilon(I)}\Psi_{I,n_{0}^{\mu-1}}^{{\bf L}_{{\bf p},{\bf q}}^{\mu,r}}\cdot}{b_{(i_{\mu}),n_{0}^{\mu-1}}^{{\bf n}_{{\bf p},{\bf q}}^{\mu,r}}b_{(i_{\mu},i_{\mu+1}),n_{0}^{\mu-1}}^{{\bf n}_{{\bf p},{\bf q}}^{\mu,r}}\cdots b_{I,n_{0}^{\mu-1}}^{{\bf n}_{{\bf p},{\bf q}}^{\mu,r}}}E_{i-|{\bf n}_{{\bf p}}^{\mu,r}|,i+n_{0}^{\mu-1}+|{\bf n}_{{\bf q}}^{\mu,r}|}.\label{eq:MainSeriesKeyLem}
\end{align}
Here \textcolor{black}{$\mathcal{W}_{\mu,r}$} denotes the set of words $I=(i_{\mu},i_{\mu+1},\cdots,i_{r})\in\{\pm1\}^{r-\mu+1}$. Given such a word $I$, the vector $\mathbf{p}$
denotes the subword of $(\mu,\cdots,r)$ which records the locations
of $-1$'s in $I$ and ${\bf q}\triangleq(\mu,\cdots,r)\backslash{\bf p}$.
The quantity $\varepsilon(I)$ denotes the length of ${\bf p}$ (i.e.
number of $-1$'s in $I$). The function $\Theta_{L_{\mu-1},\cdots,L_{1},n_{0}}^{i}$
is defined by 
\[
\Theta_{L_{\mu-1},\cdots,L_{1},n_{0}}^{i}\triangleq\Theta_{L_{\mu-1},\cdots,L_{1},n_{0},i}(b_{i},b_{i+1},\cdots,b_{i+k-1}),
\]
where $\Theta_{L_{\mu-1},\cdots,L_{1},n_{0},i}$ is a suitable meromorphic
function in $m$ complex variables ($m\triangleq n_{0}+|L_{1}|+\cdots+|L_{\mu-1}|$)
whose shape depends on $(L_{1},\cdots,L_{\mu-1}),n_{0}$ and $i$.
The family $\{\Theta_{L_{1},\cdots,L_{\mu-1},n_{0}}^{i}\}$ of functions
satisfy the following property: for any integer $i$, $1\leqslant\nu\leqslant r$
and any sequence of vectors ${\bf V}=(V_{1},\cdots,V_{r})$, under
the one-dimensional reduction $b_{j}\triangleq b_{j}^{*}+w$ the function
\begin{equation}
w\mapsto\sum_{\sigma\in{\cal S}_{\nu}}\Theta_{V_{\sigma(\nu)},\cdots,V_{\sigma(1)},n_{0}}^{i}\label{eq:ThetaSym}
\end{equation}
is analytic near $w=0$ and vanishes at $w=0$. When $\mu=1,$ the
function $\Theta_{L_{\mu-1},\cdots,L_{1},n_{0}}^{i}$ depends only
on $n_{0}$ and $i$. As a convention, if $\mu=r+1$ the summation
$\sum_{I\in{\cal W}(\mu,r)}$ in (\ref{lem:MainSeriesKeyLem}) is
set to be one and $|{\bf n}_{{\bf p}}^{\mu,r}|=|{\bf n}_{{\bf q}}^{\mu,r}|\triangleq0$
. 
\end{lem}
\begin{rem}
The functions $\Theta_{V_{\nu},\cdots,V_{1},n_{0}}^{i}$ may be singular
at $w=0$; it is their symmetrisation (\ref{eq:ThetaSym}) which will
eliminate the possible singularity. 
\end{rem}

\subsubsection*{Explanation of notation}

Here we explain the notation involved in the expression (\ref{eq:MainSeriesKeyLem}).

\vspace{2mm}\noindent (i) Let ${\bf L}\triangleq(L_{1},\cdots,L_{r})$
and $n_{i}\triangleq|L_{i}|$ ($1\leqslant i\leqslant r$). Given
$0\leqslant i\leqslant j\leqslant r,$ we set $n_{i}^{j}\triangleq n_{i}+n_{i+1}+\cdots+n_{j}$.
For each $1\leqslant\mu\leqslant r$, we set ${\bf n}^{\mu,r}\triangleq(n_{\mu},\cdots,n_{r})$.
Let ${\bf p}=(p_{1},\cdots,p_{l})$ be a given subword of ${\bf n}^{\mu,r}$
and ${\bf q}\triangleq{\bf n}^{\mu,r}\backslash{\bf p}=(q_{1},\cdots,q_{m})$.
Elements in ${\bf p},{\bf q}$ are arranged in the natural order.
We write 
\[
{\bf n}_{{\bf p}}^{\mu,r}\triangleq(n_{p_{1}},\cdots,n_{p_{l}}),\ {\bf n}_{{\bf q}}^{\mu,r}\triangleq(n_{q_{1}},\cdots,n_{q_{m}})
\]
respectively and set $|{\bf n}_{{\bf p}}^{\mu,r}|\triangleq n_{p_{1}}+\cdots+n_{p_{l}}.$

\vspace{2mm}\noindent (ii) Given $i\leqslant j$, we define $b_{i}^{j}\triangleq b_{i}+b_{i+1}+\cdots+b_{j}.$
We always use $\Psi_{L,j}$ to denote a function in $|L|$ variables
whose shape depends on a given vector $L$ and a number $j$. We write
\[
\Psi_{L}^{j}\triangleq\Psi_{L,j}(b_{j},b_{j+1},\cdots,b_{j+|K|-1}).
\]
Here $b_{j}\triangleq b_{j}^{*}+w$ is viewed as a function of $w$. 

\vspace{2mm}\noindent (iii) Let ${\bf L}$, $n_{i}$, $\mu$, ${\bf p},{\bf q}$
be given as in Part (i). Let $I=(i_{\mu},i_{\mu+1},\cdots,i_{r})$ be a
given word with $i_{j}=\pm1$. Denote $s$ (respectively, $t$) as
the number of $-1$'s (respectively, $1$'s) in $I$. Suppose that
$s\leqslant l$ and $t\leqslant m$. We define 
\begin{equation}
{\bf b}_{I,n_{0}^{\mu-1}}^{{\bf n}_{{\bf p},{\bf q}}^{\mu,r},}\triangleq b_{i-n_{p_{1}}-\cdots-n_{p_{s}}}^{i-1}+b_{i+n_{0}^{\mu-1}}^{i+n_{0}^{\mu-1}+n_{q_{1}}+\cdots+n_{q_{t}}-1}.\label{eq:bI}
\end{equation}
and
\begin{align}
\Psi_{I,n_{0}^{\mu-1}}^{{\bf L}_{{\bf p},{\bf q}}^{\mu,r}}\triangleq & \Psi_{L_{p_{1}}}^{i-n_{p_{1}}}\Psi_{L_{p_{2}}}^{i-n_{p_{1}}-n_{p_{2}}}\cdots\Psi_{L_{p_{s}}}^{i-n_{p_{1}}-\cdots-n_{p_{s}}}\nonumber \\
 & \ \ \ \times\Psi_{L_{q_{1}}}^{i+n_{0}^{\mu-1}}\Psi_{L_{q_{2}}}^{i+n_{0}^{\mu-1}+q_{1}}\cdots\Psi_{L_{q_{t}}}^{i+n_{0}^{\mu-1}+q_{1}+\cdots+q_{t-1}}.\label{eq:PsiI}
\end{align}

\begin{rem}
The superscripts in ${\bf b}$ and $\Psi$ mean that the movement
steps are taken from the word ${\bf n}^{\mu,r}$, where the indices
${\bf p}$ and ${\bf q}$ record backward and forward index movements
with sizes $n_{p_{j}}$ and $n_{q_{k}}$ respectively. The subscript
$\mu$ means that the index gap between the initial backward and forward
movements is $n_{0}^{\mu}(=n_{0}+n_{1}+\cdots+n_{\mu})$. In other
words, the backward movement starts from index $i-1$ and the forward
movement starts from $i+n_{0}^{\mu}$. In the above definitions, $i$
is fixed and is not reflected in the notation.
\end{rem}
\begin{rem}
\label{rem:IDep}The quantities ${\bf b}_{I,n_{0}^{\mu-1}}^{{\bf n}_{{\bf p},{\bf q}}^{\mu,r},}$
and $\Psi_{I,n_{0}^{\mu-1}}^{{\bf L}_{{\bf p},{\bf q}}^{\mu,r}}$
depend on the word $I$ only through the numbers $s,t$ (numbers of
$\mp1$'s respectively).
\end{rem}

\subsubsection*{Proof of Lemma \ref{lem:MainSeriesKeyLem}}

Consider the following slightly more general claim which depends on
the positive integer $r$. 

\vspace{2mm}\noindent Claim \textbf{P}($r$): The following representation
holds 
\begin{align}
 & \sum_{\xi_{0}+\xi_{1}+\cdots+\xi_{r}=v}{\rm ad}_{A}^{\xi_{r}}\circ{\rm ad}_{\hat{F}_{w}(H_{L_{r}})}\circ\cdots\circ{\rm ad}_{A}^{\xi_{1}}\circ{\rm ad}_{\hat{F}_{w}(H_{L_{1}})}\circ{\rm ad}_{A}^{\xi_{0}}\big(\hat{F}_{w}(D_{n_{0}}\tilde{L})\big)\nonumber \\
 & =\sum_{i}\sum_{\mu=1}^{r+1}\Theta_{L_{\mu-1},\cdots,L_{1},n_{0}}^{i}\cdot(b_{i}^{i+n_{0}^{\mu-1}-1})^{v+r+1-\mu}\nonumber \\
 & \ \ \ \ \ \ \ \ \ \ \ \ \sum_{I\in{\cal W}_{\mu,r}}\frac{(-1)^{\varepsilon(I)}\Psi_{I,n_{0}^{\mu-1}}^{{\bf L}_{{\bf p},{\bf q}}^{\mu,r}}\cdot}{b_{(i_{\mu}),n_{0}^{\mu-1}}^{{\bf n}_{{\bf p},{\bf q}}^{\mu,r}}b_{(i_{\mu},i_{\mu+1}),n_{0}^{\mu-1}}^{{\bf n}_{{\bf p},{\bf q}}^{\mu,r}}\cdots b_{I,n_{0}^{\mu-1}}^{{\bf n}_{{\bf p},{\bf q}}^{\mu,r}}}E_{i-|{\bf n}_{{\bf p}}^{\mu,r}|,i+n_{0}^{\mu-1}+|{\bf n}_{{\bf q}}^{\mu,r}|}\label{eq:Pr}
\end{align}
for any sequence of vectors ${\bf L}=(L_{1},\cdots,L_{r})$ and
$v\in\mathbb{N}$. Here the functions $\{\Theta_{L_{\mu-1},\cdots,L_{1},n_{0}}^{i}\}$
satisfy the properties stated in Lemma \ref{lem:MainSeriesKeyLem}. 

\vspace{2mm}\noindent Note that the result of Lemma \ref{lem:MainSeriesKeyLem}
follows immediately by taking $v=l-r$. We now prove the claim ${\bf P}(r)$
by induction on $r.$ 

\vspace{2mm}\noindent \textit{\uline{Step 1: Base case.}}

\vspace{2mm} For the base case $r=1$, one first recalls from Corollary
\ref{cor:FDkSk} that 
\[
\hat{F}_{w}(D_{n_{0}}\tilde{L})=\sum_{i}S_{n_{0}}^{i}E_{i,i+n_{0}}
\]
where $S_{n_{0}}^{i}\triangleq S_{n_{0}}(b_{i},b_{i+1},\cdots,b_{i+n_{0}-1}).$
Therefore, one has
\begin{equation}
{\rm ad}_{A}^{\xi_{0}}(\hat{F}_{w}(D_{n_{0}}\tilde{L}))=\sum_{i}S_{n_{0}}^{i}(b_{i}^{i+n_{0}-1})^{\xi_{0}}E_{i,i+n_{0}}.\label{eq:InitialAd}
\end{equation}
Next, recall that the assumption (\ref{eq:IndHypo}) is valid since
we are under the induction step for proving of the main Lemma \ref{lem:AnaLem}
(all the $L_{j}$'s are assumed to satisfy $|L_{j}|<k$ where $k$
is defined in the Claim (\ref{eq:FSymProd})). In particular, one
has 

\begin{equation}
\hat{F}_{w}(H_{L_{1}})=\sum_{j}\Psi_{L_{1}}^{j}E_{j,j+n_{1}}\label{eq:IndHypKeyLem}
\end{equation}
where $\Psi_{L_{1}}^{j}$ are analytic near $w=0$. It follows from
(\ref{eq:InitialAd}, \ref{eq:IndHypKeyLem}) that 
\begin{align*}
{\rm ad}_{\hat{F}_{w}(H_{L_{1}})}\circ{\rm ad}_{A}^{\xi_{0}}\big(\hat{F}_{w}(D_{n_{0}}\tilde{L})\big) & =\sum_{i}S_{n_{0}}^{i}\cdot(b_{i}^{i+n_{0}-1})^{\xi_{0}}\sum_{j}\Psi_{L_{1}}^{j}[E_{j,j+n_{1}},E_{i,i+n_{0}}].
\end{align*}
There are at most two $j$-cases to make the above Lie bracket nonzero:
$j=i-n_{1}$ and $j=i+n_{0}$. The former case is interpreted as a
``backward movement'' while the latter is a ``forward movement''.
In other words, one can write
\[
{\rm ad}_{\hat{F}_{w}(H_{L_{1}})}\circ{\rm ad}_{A}^{\xi_{0}}\big(\hat{F}_{w}(D_{n_{0}}\tilde{L})\big)=\sum_{i}S_{n_{0}}^{i}(b_{i}^{i+n_{0}-1})^{\xi_{0}}\big(\Psi_{L_{1}}^{i-n_{1}}E_{i-n_{1},i+n_{0}}-\Psi_{L_{1}}^{i+n_{0}}E_{i,i+n_{0}+n_{1}}\big).
\]
The additional ${\rm ad}_{A}^{\xi_{1}}$-action yields that
\begin{align*}
{\rm ad}_{A}^{\xi_{1}}\circ{\rm ad}_{\hat{F}_{w}(H_{L_{1}})}\circ{\rm ad}_{A}^{\xi_{0}}\big(\hat{F}_{w}(D_{n_{0}}\tilde{L})\big)=\sum_{i}S_{n_{0}}^{i} & (b_{i}^{i+n_{0}-1})^{\xi_{0}}\big[\Psi_{L_{1}}^{i-n_{1}}\cdot(b_{i-n_{1}}^{i+n_{0}-1})^{\xi_{1}}E_{i-n_{1},i+n_{0}}\\
 & \ \ \ -\Psi_{L_{1}}^{i+n_{0}}\cdot(b_{i}^{i+n_{0}+n_{1}-1})^{\xi_{1}}E_{i,i+n_{0}+n_{1}}\big].
\end{align*}
Now we perform the summation over $\xi_{0}+\xi_{1}=v$ by using the
elementary formula
\begin{equation}
\sum_{\xi_{0}+\xi_{1}=v}x^{\xi_{0}}y^{\xi_{1}}=\frac{y^{v+1}-x^{v+1}}{y-x}.\label{eq:SumProd}
\end{equation}
This gives that 
\begin{align}
 & \sum_{\substack{\xi_{0}+\xi_{1}=v}
}{\rm ad}_{A}^{\xi_{1}}\circ{\rm ad}_{\hat{F}_{w}(H_{L_{1}})}\circ{\rm ad}_{A}^{\xi_{0}}\big(F(D_{n_{0}}\tilde{L})\big)\nonumber \\
 & =\sum_{i}S_{n_{0}}^{i}\frac{\Psi_{L_{1}}^{i-n_{1}}}{b_{i-n_{1}}^{i-1}}\big[(b_{i-n_{1}}^{i+n_{0}-1})^{v+1}-(b_{i}^{i+n_{0}-1})^{v+1}\big]E_{i-n_{1},i+n_{0}}\nonumber \\
 & \ \ \ -\sum_{i}S_{n_{0}}^{i}\frac{\Psi_{L_{1}}^{i+n_{0}}}{b_{i+n_{0}}^{i+n_{0}+n_{1}-1}}\big[(b_{i}^{i+n_{0}+n_{1}-1})^{v+1}-(b_{i}^{i+n_{0}-1})^{v+1}\big]E_{i,i+n_{0}+n_{1}}\nonumber \\
 & =\sum_{i}S_{n_{0}}^{i}(b_{i}^{i+n_{0}-1})^{v+1}\big(-\frac{\Psi_{L_{1}}^{i-n_{1}}}{b_{i-n_{1}}^{i-1}}E_{i-n_{1},i+n_{0}}+\frac{\Psi_{L_{1}}^{i+n_{0}}}{b_{i+n_{0}}^{i+n_{0}+n_{1}-1}}E_{i,i+n_{0}+n_{1}}\big)\nonumber \\
 & \ \ \ +\sum_{i}\big[\big(S_{n_{0}}^{i+n_{1}}\frac{\Psi_{L_{1}}^{i}}{b_{i}^{i+n_{1}-1}}-S_{n_{0}}^{i}\frac{\Psi_{L_{1}}^{i+n_{0}}}{b_{i+n_{0}}^{i+n_{0}+n_{1}-1}}\big)(b_{i}^{i+n_{0}+n_{1}-1})^{v+1}\big]E_{i,i+n_{0}+n_{1}},\label{eq:S1Action}
\end{align}
where the last line follows from a change of indices $i-n_{1}\leftrightarrow i$.
The desired identity (\ref{eq:Pr}) thus follows by taking $\Theta_{n_{0}}^{i}\triangleq S_{n_{0}}^{i}$
($\mu=1$) and 
\begin{equation}
\Theta_{L_{1},n_{0}}^{i}\triangleq\big(S_{n_{0}}^{i+n_{1}}\frac{\Psi_{L_{1}}^{i}}{b_{i}^{i+n_{1}-1}}-S_{n_{0}}^{i}\frac{\Psi_{L_{1}}^{i+n_{0}}}{b_{i+n_{0}}^{i+n_{0}+n_{1}-1}}\big)(b_{i}^{i+n_{0}+n_{1}-1}).\label{eq:Theta1}
\end{equation}
The function $\Theta_{L_{1},n_{0}}^{i}$ is analytic near $w=0$ since
the sequence $(b_{1}^{*},\cdots,b_{N}^{*})$ is non-degenerate (the
denominators appearing in (\ref{eq:Theta1}) are nonzero near $w=0$).
It is clear that both $\Theta_{n_{0}}^{i}$ and $\Theta_{L_{1},n_{0}}^{i}$
vanishes at $w=0$ due to the presence of the $S_{n_{0}}$-terms (cf.
Assumption (${\bf A}_{N-1}$)).

\vspace{2mm}\noindent \textit{\uline{Step 2: Evaluation of (\mbox{\ref{eq:Pr}})
based on induction hypothesis.}}

\vspace{2mm} Now suppose that Claim ${\bf P}(r-1)$ is true and we
want to prove Claim ${\bf P}(r)$. Let $(L_{1},\cdots,L_{r})$ be
a given sequence of vectors and let $v\geqslant0.$ By applying (\ref{eq:IndHypo})
and Claim $\mathbf{P}(r-1)$ to $v-\xi_{r},$ one finds that 
\begin{align}
 & \sum_{\xi_{0}+\cdots+\xi_{r-1}+\xi_{r}=v}{\rm ad}_{A}^{\xi_{r}}\circ{\rm ad}_{\hat{F}_{w}(H_{L_{r}})}\circ{\rm ad}_{A}^{\xi_{r-1}}\circ\cdots\circ{\rm ad}_{\hat{F}_{w}(H_{L_{1}})}\circ{\rm ad}_{A}^{\xi_{0}}\big(\hat{F}_{w}(D_{n_{0}}\tilde{L})\big)\nonumber \\
 & =\sum_{\xi_{r}=0}^{v}{\rm ad}_{A}^{\xi_{r}}\big(\big[\sum_{j}\Psi_{L_{r}}^{j}E_{j,j+n_{r}},\sum_{\xi_{0}+\cdots+\xi_{r-1}=v-\xi_{r}}{\rm ad}_{A}^{\xi_{r-1}}\circ{\rm ad}_{\hat{F}_{w}(H_{L_{r-1}})}\circ\nonumber \\
 & \ \ \ \ \ \ \ \ \ \ \ \ \ \ \ \ \ \ \cdots\circ{\rm ad}_{A}^{\xi_{1}}\circ{\rm ad}_{\hat{F}_{w}(H_{L_{1}})}\circ{\rm ad}_{A}^{\xi_{0}}\big(\hat{F}_{w}(D_{n_{0}}\tilde{L})\big)\big]\big)\nonumber \\
 & =\sum_{\xi_{r}=0}^{v}\sum_{i}\sum_{\mu=1}^{r}\Theta_{L_{\mu-1},\cdots,L_{1},n_{0}}^{i}\cdot\sum_{I\in{\cal W}_{\mu,r-1}}\frac{(-1)^{\varepsilon(I)}\Psi_{I,n_{0}^{\mu-1}}^{{\bf L}_{{\bf p},{\bf q}}^{\mu,r-1}}\cdot\Psi_{L_{r}}^{i-|{\bf n}_{{\bf p}}^{\mu,r-1}|-n_{r}}}{b_{(i_{\mu}),n_{0}^{\mu-1}}^{{\bf n}_{{\bf p},{\bf q}}^{\mu,r-1}}\cdots b_{I,n_{0}^{\mu-1}}^{{\bf n}_{{\bf p},{\bf q}}^{\mu,r-1}}}\nonumber \\
 & \ \ \ \times\big(b_{i}^{i+n_{0}^{\mu-1}-1}\big)^{v-\xi_{r}+r-\mu}\big(b_{i-|{\bf n}_{{\bf p}}^{\mu,r-1}|-n_{r}}^{i+n_{0}^{\mu-1}+|{\bf n}_{{\bf q}}^{\mu,r-1}|-1}\big)^{\xi_{r}}E_{i-|{\bf n}_{{\bf p}}^{\mu,r-1}|-n_{r},i+n_{0}^{\mu-1}+|{\bf n}_{{\bf q}}^{\mu,r-1}|}\label{eq:Pf1}\\
 & \ -\sum_{\xi_{r}=0}^{v}\sum_{i}\sum_{\mu=1}^{r}\Theta_{L_{\mu-1},\cdots,L_{1},n_{0}}^{i}\cdot\sum_{I\in{\cal W}_{\mu,r-1}}\frac{(-1)^{\varepsilon(I)}\Psi_{I,n_{0}^{\mu-1}}^{{\bf L}_{{\bf p},{\bf q}}^{\mu,r-1}}\cdot\Psi_{L_{r}}^{i+n_{0}^{\mu-1}+|{\bf n}_{{\bf q}}^{\mu,r-1}|}}{b_{(i_{\mu}),n_{0}^{\mu-1}}^{{\bf n}_{{\bf p},{\bf q}}^{\mu,r-1}}\cdots b_{I,n_{0}^{\mu-1}}^{{\bf n}_{{\bf p},{\bf q}}^{\mu,r-1}}}\nonumber \\
 & \ \ \ \times\big(b_{i}^{i+n_{0}^{\mu-1}-1}\big)^{v-\xi_{r}+r-\mu}\big(b_{i-|{\bf n}_{{\bf p}}^{\mu,r-1}|}^{i+n_{0}^{\mu-1}+|{\bf n}_{{\bf q}}^{\mu,r-1}|+n_{r}-1}\big)^{\xi_{r}}E_{i-|{\bf n}_{{\bf p}}^{\mu,r-1}|,i+n_{0}^{\mu-1}+|{\bf n}_{{\bf q}}^{\mu,r-1}|+n_{r}}.\label{eq:Pf2}
\end{align}
Just like the base case, one performs the summation over $\xi_{r}$
in (\ref{eq:Pf1}) and (\ref{eq:Pf2}) respectively. This gives that
\begin{align}
 & \sum_{\xi_{r}=0}^{v}\big(b_{i}^{i+n_{0}^{\mu-1}-1}\big)^{v-\xi_{r}+r-\mu}\big(b_{i-|{\bf n}_{{\bf p}}^{\mu,r-1}|-n_{r}}^{i+n_{0}^{\mu-1}+|{\bf n}_{{\bf q}}^{\mu,r-1}|-1}\big)^{\xi_{r}}\nonumber \\
 & =(b_{i}^{i+n_{0}^{\mu-1}-1})^{r-\mu}\times\frac{\big(b_{i-|{\bf n}_{({\bf p},r)}^{\mu,r}|}^{i+n_{0}^{\mu-1}+|{\bf n}_{{\bf q}}^{\mu,r-1}|-1}\big)^{v+1}-\big(b_{i}^{i+n_{0}^{\mu-1}-1}\big)^{v+1}}{b_{(I,-1),n_{0}^{\mu-1}}^{{\bf n}_{({\bf p},r),{\bf q}}^{\mu,r}}}\label{eq:Pf11}
\end{align}
and respectively, 
\begin{align}
 & \sum_{\xi_{r}=0}^{v}\big(b_{i}^{i+n_{0}^{\mu-1}-1}\big)^{v-\xi_{r}+r-\mu}\big(b_{i-|{\bf n}_{{\bf p}}^{\mu,r-1}|}^{i+n_{0}^{\mu-1}+|{\bf n}_{{\bf q}}^{\mu,r-1}|+n_{r}-1}\big)^{\xi_{r}}\nonumber \\
 & =(b_{i}^{i+n_{0}^{\mu-1}-1})^{r-\mu}\times\frac{\big(b_{i-|{\bf n}_{{\bf p}}^{\mu,r}|}^{i+n_{0}^{\mu-1}+|{\bf n}_{({\bf q},r)}^{\mu,r}|-1}\big)^{v+1}-\big(b_{i}^{i+n_{0}^{\mu-1}-1}\big)^{v+1}}{b_{(I,1),n_{0}^{\mu-1}}^{{\bf n}_{{\bf p},({\bf q},r)}^{\mu,r}}}.\label{eq:Pf21}
\end{align}
By substituting (\ref{eq:Pf11}, \ref{eq:Pf21}) into (\ref{eq:Pf1},
\ref{eq:Pf2}) respectively, one can write
\begin{align}
 & \sum_{\xi_{0}+\cdots\xi_{r-1}+\xi_{r}=v}{\rm ad}_{A}^{\xi_{r}}\circ{\rm ad}_{\hat{F}_{w}(H_{L_{r}})}\circ\cdots\circ{\rm ad}_{A}^{\xi_{1}}\circ{\rm ad}_{\hat{F}_{w}(H_{L_{1}})}\circ{\rm ad}_{A}^{\xi_{0}}\big(\hat{F}_{w}(D_{n_{0}}\tilde{L})\big)\nonumber \\
 & =\sum_{i}\sum_{\mu=1}^{r}\Theta_{L_{\mu-1},\cdots,L_{1},n_{0}}^{i}(b_{i}^{i+n_{0}^{\mu-1}-1})^{v+r+1-\mu}\sum_{I\in{\cal W}_{\mu,r}}\frac{(-1)^{\varepsilon(I)}\Psi_{I,n_{0}^{\mu-1}}^{{\bf L}_{{\bf p},{\bf q}}^{\mu,r}}}{b_{(i_{\mu}),n_{0}^{\mu-1}}^{{\bf n}_{{\bf p},{\bf q}}^{\mu,r}}\cdots b_{I,n_{0}^{\mu-1}}^{{\bf n}_{{\bf p},{\bf q}}^{\mu,r}}}\nonumber \\
 & \ \ \ \ \ \ \times E_{i-|{\bf n}_{{\bf p}}^{\mu,r}|,i+n_{0}^{\mu-1}+|{\bf n}_{{\bf q}}^{\mu,r}|}-\sum_{i}\sum_{\mu=1}^{r}\Theta_{L_{\mu-1},\cdots,L_{1},n_{0}}^{i}(b_{i}^{i+n_{0}^{\mu-1}-1})^{r-\mu}\nonumber \\
 & \ \ \ \ \ \ \ \ \ \ \ \ \ \ \ \sum_{I\in{\cal W}_{\mu,r}}\frac{(-1)^{\varepsilon(I)}\Psi_{I,n_{0}^{\mu-1}}^{{\bf L}_{{\bf p},{\bf q}}^{\mu,r}}}{b_{(i_{\mu}),n_{0}^{\mu-1}}^{{\bf n}_{{\bf p},{\bf q}}^{\mu,r}}\cdots b_{I,n_{0}^{\mu-1}}^{{\bf n}_{{\bf p},{\bf q}}^{\mu,r}}}\big(b_{i-|{\bf n}_{{\bf p}}^{\mu,r}|}^{i+n_{0}^{\mu-1}+|{\bf n}_{{\bf q}}^{\mu,r}|}\big)^{v+1}E_{i-|{\bf n}_{{\bf p}}^{\mu,r}|,i+n_{0}^{\mu-1}+|{\bf n}_{{\bf q}}^{\mu,r}|}\nonumber \\
 & =:I+J.\label{eq:J}
\end{align}
Now we define 
\begin{align}
\Theta_{L_{r},\cdots,L_{1},n_{0}}^{j}\triangleq & -\big(b_{j}^{j+n_{0}^{r}}\big)\sum_{\mu=1}^{r}\sum_{i,I:i-|{\bf n}_{{\bf p}}^{\mu,r}|=j}\Theta_{L_{\mu-1},\cdots,L_{1},n_{0}}^{i}\nonumber \\
 & \ \ \ \ \ \ \ \ \ \ \ \ \ \ \ \times(b_{i}^{i+n_{0}^{\mu-1}-1})^{r-\mu}\frac{(-1)^{\varepsilon(I)}\Psi_{I,n_{0}^{\mu-1}}^{{\bf L}_{{\bf p},{\bf q}}^{\mu,r}}}{b_{(i_{\mu}),n_{0}^{\mu-1}}^{{\bf n}_{{\bf p},{\bf q}}^{\mu,r}}\cdots b_{I,n_{0}^{\mu-1}}^{{\bf n}_{{\bf p},{\bf q}}^{\mu,r}}}.\label{eq:Ther+1}
\end{align}
It follows that 
\[
J=\sum_{j}\Theta_{L_{r},\cdots,L_{1},n_{0}}^{j}(b_{j}^{j+n_{0}^{r}-1})^{v}E_{j,j+n_{0}^{r}}.
\]
This is regarded as the term corresponding to $\mu=r+1$ in the induction
claim ${\bf P}(r)$ (cf. (\ref{eq:Pr}) and the $I$-term corresponds
to the summation $\sum_{\mu=1}^{r}$ in the claim (\ref{eq:Pr})).
It is clear from its expression that $\Theta_{L_{r},\cdots,L_{_{1}},n_{0}}^{j}$
is a meromorphic function in the complex variables $(b_{j},b_{j+1},\cdots)$
(before the one-dimensional reduction $b_{j}=b_{j}^{*}+w$).

\vspace{2mm}\noindent \textit{\uline{Step 3: Analyticity of the
symmetrisation of \mbox{$J$}.}}

\vspace{2mm} To finish the proof, it remains to show that the function
\[
\sum_{\sigma\in{\cal S}_{r}}\Theta_{L(\sigma(r)),\cdots,L(\sigma(1)),n_{0}}^{j}
\]
is (as a function of $w$) analytic near $w=0$. To this end, it is
convenient to use the original expression of $J$ defined in (\ref{eq:J}).
Equivalently, we aim at showing that the $\mathfrak{sl}_{N+1}(\mathbb{C})$-valued
function
\begin{align}
\sum_{\sigma\in{\cal S}_{r}}J= & -\sum_{i}\sum_{\mu=1}^{r}\sum_{\sigma\in{\cal S}_{r}}\sum_{I\in{\cal W}_{\mu,r}}\Theta_{L_{\sigma(\mu-1)},\cdots,L_{\sigma(1)},n_{0}}^{i}(b_{i}^{i+(n^{\sigma})_{0}^{\mu-1}-1})^{r-\mu}\nonumber \\
 & \ \ \ \ \ \ \times\frac{(-1)^{\varepsilon(I)}\Psi_{I,(n^{\sigma})_{0}^{\mu-1}}^{({\bf L}^{\sigma})_{{\bf p},{\bf q}}^{\mu,r}}}{b_{(i_{\mu}),(n^{\sigma})_{0}^{\mu-1}}^{({\bf n}^{\sigma})_{{\bf p},{\bf q}}^{\mu,r}}\cdots b_{I,(n^{\sigma})_{0}^{\mu-1}}^{({\bf n}^{\sigma})_{{\bf p},{\bf q}}^{\mu,r}}}\big(b_{i-|({\bf n}^{\sigma})_{{\bf p}}^{\mu,r}|}^{i+(n^{\sigma})_{0}^{\mu-1}+|({\bf n}^{\sigma})_{{\bf q}}^{\mu,r}|-1}\big)^{v+1}\nonumber \\
 & \ \ \ \ \ \ \times E_{i-|({\bf n}^{\sigma})_{{\bf p}}^{\mu,r}|,i+(n^{\sigma})_{0}^{\mu-1}+|({\bf n}^{\sigma})_{{\bf q}}^{\mu,r}|}\label{eq:MainAnaPf}
\end{align}
is analytic near $w=0$. Here ${\bf L}^{\sigma}\triangleq(L_{\sigma(1)},\cdots,L_{\sigma(r)})$
and respectively, we denote ${\bf n}^{\sigma}\triangleq(n_{\sigma(1)},\cdots,n_{\sigma(r)})$
and $(n^{\sigma})_{0}^{\mu-1}\triangleq n_{0}+n_{\sigma(1)}+\cdots+n_{\sigma(\mu-1)}.$ 

Let $i$ and $\mu$ be given fixed. The first crucial observation
is that any permutation $\sigma\in{\cal S}_{r}$ can be written as
\[
\sigma=\tau\circ(\zeta\otimes\theta),
\]
where $\tau$ is a $(\mu-1,r+1-\mu)$-shuffle, $\zeta$ is a permutation
over $\{1,\cdots,\mu-1\}$ and $\theta$ is a permutation over $\{\mu,\cdots,r\}$.
In this way, the summation over $\sigma$ and $I$ in (\ref{eq:MainAnaPf})
can be rewritten as 
\begin{align}
 & \sum_{\tau\in{\cal P}(\mu-1,r+1-\mu)}\sum_{\zeta\in{\cal S}_{\{1,\cdots,\mu-1\}}}\Theta_{L_{\tau\circ\zeta(\mu-1)},\cdots,L_{\tau\circ\zeta(1)},n_{0}}^{i}(b_{i}^{i+(n^{\tau\circ\zeta})_{0}^{\mu-1}-1})^{r-\mu}\nonumber \\
 & \ \ \ \ \ \ \times\sum_{s=0}^{r+1-\mu}(-1)^{s}\sum_{\theta\in{\cal S}_{\{\mu,\cdots,r\}}}\sum_{I\in{\cal W}_{\mu,r}:|{\bf p}|=s}\frac{\Psi_{I,(n^{\tau\circ\zeta})_{0}^{\mu-1}}^{({\bf L}^{\tau\circ\theta})_{{\bf p},{\bf q}}^{\mu,r}}}{b_{(i_{\mu}),(n^{\tau\circ\zeta})_{0}^{\mu-1}}^{({\bf n}^{\tau\circ\theta})_{{\bf p},{\bf q}}^{\mu,r}}\cdots b_{I,(n^{\tau\circ\zeta})_{0}^{\mu-1}}^{({\bf n}^{\tau\circ\theta})_{{\bf p},{\bf q}}^{\mu,r}}}\nonumber \\
 & \ \ \ \ \ \ \times\big(b_{i-|({\bf n}^{\tau\circ\theta})_{{\bf p}}^{\mu,r}|}^{i+(n^{\tau\circ\zeta})_{0}^{\mu-1}+|({\bf n}^{\tau\circ\theta})_{{\bf q}}^{\mu,r}|-1}\big)^{v+1}E_{i-|({\bf n}^{\tau\circ\theta})_{{\bf p}}^{\mu,r}|,i+(n^{\tau\circ\zeta})_{0}^{\mu-1}+|({\bf n}^{\tau\circ\theta})_{{\bf q}}^{\mu,r}|}.\label{eq:MainAnaPf2}
\end{align}
It is important to note that the expression in the last two lines
depends only on $\tau$ and is independent of $\zeta,$ since 
\[
(n^{\tau\circ\zeta})_{0}^{\mu-1}\triangleq n_{0}+n_{\tau(\zeta(1))}+\cdots+n_{\tau(\zeta(\mu-1))}=n_{0}+n_{\tau(1)}+\cdots+n_{\tau(\mu-1)}.
\]
Let us denote the summation over $(\theta,I)$ in (\ref{eq:MainAnaPf2})
by ${\cal T}_{s,\tau}$. It follows that 
\begin{align*}
(\ref{eq:MainAnaPf2})\ = & \sum_{\tau\in{\cal P}(\mu-1,r+1-\mu)}(b_{i}^{i+(n^{\tau})_{0}^{\mu-1}-1})^{r-\mu}\sum_{s=0}^{r+1-\mu}(-1)^{s}{\cal T}_{s,\tau}\\
 & \ \ \ \ \ \ \times\sum_{\zeta\in{\cal S}_{\{1,\cdots,\mu-1\}}}\Theta_{L_{\tau\circ\zeta(\mu-1)},\cdots,L_{\tau\circ\zeta(1)},n_{0}}^{i}.
\end{align*}
According to the induction hypothesis ${\bf P}(r-1)$, the function
\[
\sum_{\zeta\in{\cal S}_{\{1,\cdots,\mu-1\}}}\Theta_{L_{\tau\circ\zeta(\mu-1)},\cdots,L_{\tau\circ\zeta(1)},n_{0}}^{i}
\]
is (as a function of $w$) analytic near $w=0$ and vanishes at $w=0$. 

Now our task is reduced to showing that the function ${\cal T}_{s,\tau}$
is analytic near $w=0$ for any fixed $s$ and $\tau$. To ease notation,
we will temporarily write 
\[
\bar{n}_{0}\triangleq(n^{\tau})_{0}^{\mu-1},\ \bar{{\bf L}}\triangleq(L_{\tau(\mu)},\cdots,L_{\tau(r)}),\ \bar{{\bf n}}\triangleq(n_{\tau(\mu)},\cdots,n_{\tau(r)}).
\]
Given a permutation $\theta\in{\cal S}_{\{\mu,\cdots,r\}}$, we denote
\[
\bar{{\bf L}}^{\theta}\triangleq(\bar{L}_{\theta(\mu)},\cdots,\bar{L}_{\theta(r)})=(L_{\tau(\theta(\mu))},\cdots,L_{\tau(\theta(r))})
\]
and respectively,
\[
\bar{{\bf n}}^{\theta}\triangleq(\bar{n}_{\theta(\mu)},\cdots,\bar{n}_{\theta(r)})=(n_{\tau(\theta(\mu))},\cdots,n_{\tau(\theta(r))}).
\]
Under the above simplified notation, our target ${\cal T}_{s,\tau}$
is defined by 
\begin{align}
{\cal T}_{s,\tau}\triangleq & \sum_{\theta\in{\cal S}_{\{\mu,\cdots,r\}}}\sum_{I\in{\cal W}_{\mu,r}:|{\bf p}|=s}\frac{\Psi_{I,\bar{n}_{0}}^{(\bar{{\bf L}}^{\theta})_{{\bf p},{\bf q}}^{\mu,r}}}{b_{(i_{\mu}),\bar{n}_{0}}^{(\bar{{\bf n}}^{\theta})_{{\bf p},{\bf q}}^{\mu,r}}\cdots b_{I,\bar{n}_{0}}^{(\bar{{\bf n}}^{\theta})_{{\bf p},{\bf q}}^{\mu,r}}}\nonumber \\
 & \ \ \ \ \ \ \times\big(b_{i-|(\bar{{\bf n}}^{\theta})_{{\bf p}}^{\mu,r}|}^{i+\bar{n}_{0}+|(\bar{{\bf n}}^{\theta})_{{\bf q}}^{\mu,r}|-1}\big)^{v+1}E_{i-|(\bar{{\bf n}}^{\theta})_{{\bf p}}^{\mu,r}|,i+\bar{n}_{0}+|(\bar{{\bf n}}^{\theta})_{{\bf q}}^{\mu,r}|}\label{eq:Tstau}
\end{align}

The next crucial observation is that the set of words $I\in{\cal W}_{\mu,r}$
with precisely $s$ number of $-1$'s is in one-to-one correspondence
with the set ${\cal P}(s;\mu,r)$ of $(s,r+1-\mu-s)$-shuffles over
$\{\mu,\cdots,r\}$, i.e. permutations $\rho\in{\cal S}_{\{\mu,\cdots,r\}}$
satisfying
\[
\rho(\mu)<\rho(\mu+1)<\cdots<\rho(\mu+s-1),\ \rho(\mu+s)<\cdots<\rho(r).
\]
Indeed, given $\rho\in{\cal P}(s;\mu,r)$ one can define 
\[
I_{\rho}=(i_{\mu},\cdots,i_{r}):\ i_{\rho(j)}=\begin{cases}
-1, & j=\mu,\cdots,\mu+s-1;\\
1, & j=\mu+s,\cdots,r.
\end{cases}
\]
In this way, ${\bf p}=(\rho(\mu),\cdots,\rho(\mu+s-1))$ and ${\bf q}=(\rho(\mu+s),\cdots,\rho(r))$.
Using this correspondence, one can write

\begin{align}
{\cal T}_{s,\tau}= & \sum_{\theta\in{\cal S}_{\{\mu,\cdots,r\}}}\sum_{\rho\in{\cal P}(s;\mu,r)}\frac{\Psi_{I_{\rho},\bar{n}_{0}}^{(\bar{{\bf L}}^{\theta\circ\rho})_{\bar{{\bf p}},\bar{{\bf q}}}^{\mu,r}}}{b_{(i_{\mu}),\bar{n}_{0}}^{(\bar{{\bf n}}^{\theta\circ\rho})_{\bar{{\bf p}},\bar{{\bf q}}}^{\mu,r}}\cdots b_{I_{\rho},\bar{n}_{0}}^{(\bar{{\bf n}}^{\theta\circ\rho})_{\bar{{\bf p}},\bar{{\bf q}}}^{\mu,r}}}\nonumber \\
 & \ \ \ \ \ \ \times\big(b_{i-|(\bar{{\bf n}}^{\theta\circ\rho})_{\bar{{\bf p}}}^{\mu,r}|}^{i+\bar{n}_{0}+|(\bar{{\bf n}}^{\theta\circ\rho})_{\bar{{\bf q}}}^{\mu,r}|-1}\big)^{v+1}E_{i-|(\bar{{\bf n}}^{\theta\circ\rho})_{\bar{{\bf p}}}^{\mu,r}|,i+\bar{n}_{0}+|(\bar{{\bf n}}^{\theta\circ\rho})_{\bar{{\bf q}}}^{\mu,r}|}.\label{eq:KeyLemPf3}
\end{align}
Here $\bar{{\bf p}}\triangleq(\mu,\cdots,\mu+s-1)$ and $\bar{{\bf q}}\triangleq(\mu+s,\cdots,r)$
are the two canonical subwords which are both independent of $\rho$.
Recall from Remark \ref{rem:IDep} that the function $\Psi_{I_{\rho},\bar{n}_{0}}^{(\bar{{\bf L}}^{\theta\circ\rho})_{\bar{{\bf p}},\bar{{\bf q}}}^{\mu,r}}$
depends on $I_{\rho}$ only through the numbers of $\mp1$'s in it.
In particular, one has 
\[
\Psi_{I_{\rho},\bar{n}_{0}}^{(\bar{{\bf L}}^{\theta\circ\rho})_{\bar{{\bf p}},\bar{{\bf q}}}^{\mu,r}}=\Psi_{\bar{I},\bar{n}_{0}}^{(\bar{{\bf L}}^{\theta\circ\rho})_{\bar{{\bf p}},\bar{{\bf q}}}^{\mu,r}}
\]
for all $\rho\in{\cal P}(s;\mu,r)$ where 
\[
\bar{I}=(\underbrace{-1,\cdots,-1}_{s},\underbrace{1,\cdots,1}_{r+1-\mu-s}).
\]
With this observation in mind, the point is that one can now exchange
the summations in (\ref{eq:KeyLemPf3}), and after doing so, since
the $\theta$-summation is a full symmetrisation one could just replace
$\theta\circ\rho$ by $\theta$ (ignoring the role of \textcolor{black}{$\rho$}).
In other words, one has 
\begin{align*}
{\cal T}_{s,\tau}= & \sum_{\rho\in{\cal P}(s;\mu,r)}\sum_{\theta\in{\cal S}_{\{\mu,\cdots,r\}}}\frac{\Psi_{\bar{I},\bar{n}_{0}}^{(\bar{{\bf L}}^{\theta})_{\bar{{\bf p}},\bar{{\bf q}}}^{\mu,r}}}{b_{(i_{\mu}),\bar{n}_{0}}^{(\bar{{\bf n}}^{\theta})_{\bar{{\bf p}},\bar{{\bf q}}}^{\mu,r}}\cdots b_{I_{\rho},\bar{n}_{0}}^{(\bar{{\bf n}}^{\theta})_{\bar{{\bf p}},\bar{{\bf q}}}^{\mu,r}}}\\
 & \ \ \ \ \ \ \times\big(b_{i-|(\bar{{\bf n}}^{\theta})_{\bar{{\bf p}}}^{\mu,r}|}^{i+\bar{n}_{0}+|(\bar{{\bf n}}^{\theta})_{\bar{{\bf q}}}^{\mu,r}|-1}\big)^{v+1}E_{i-|(\bar{{\bf n}}^{\theta})_{\bar{{\bf p}}}^{\mu,r}|,i+\bar{n}_{0}+|(\bar{{\bf n}}^{\theta})_{\bar{{\bf q}}}^{\mu,r}|}\\
= & \sum_{\theta\in{\cal S}_{\{\mu,\cdots,r\}}}\Psi_{\bar{I},\bar{n}_{0}}^{(\bar{{\bf L}}^{\theta})_{\bar{{\bf p}},\bar{{\bf q}}}^{\mu,r}}\cdot\big(b_{i-|(\bar{{\bf n}}^{\theta})_{\bar{{\bf p}}}^{\mu,r}|}^{i+\bar{n}_{0}+|(\bar{{\bf n}}^{\theta})_{\bar{{\bf q}}}^{\mu,r}|-1}\big)^{v+1}\\
 & \ \ \ \ \ \ \times\big[\sum_{\rho\in{\cal P}(s;\mu,r)}\frac{1}{b_{(i_{\mu}),\bar{n}_{0}}^{(\bar{{\bf n}}^{\theta})_{\bar{{\bf p}},\bar{{\bf q}}}^{\mu,r}}\cdots b_{I_{\rho},\bar{n}_{0}}^{(\bar{{\bf n}}^{\theta})_{\bar{{\bf p}},\bar{{\bf q}}}^{\mu,r}}}\big]E_{i-|(\bar{{\bf n}}^{\theta})_{\bar{{\bf p}}}^{\mu,r}|,i+\bar{n}_{0}+|(\bar{{\bf n}}^{\theta})_{\bar{{\bf q}}}^{\mu,r}|}.
\end{align*}
where we changed the order of summation back to reach the last equality. 

It remains to show that for every fixed $\theta,$ the summation
\[
{\cal T}_{s,\tau;\theta}\triangleq\sum_{\rho\in{\cal P}(s;\mu,r)}\frac{1}{b_{(i_{\mu}),\bar{n}_{0}}^{(\bar{{\bf n}}^{\theta})_{\bar{{\bf p}},\bar{{\bf q}}}^{\mu,r}}\cdots b_{I_{\rho},\bar{n}_{0}}^{(\bar{{\bf n}}^{\theta})_{\bar{{\bf p}},\bar{{\bf q}}}^{\mu,r}}}
\]
defines an analytic function near $w=0$. The third crucial observation
is that ${\cal T}_{s,\tau;\theta}$ is precisely in the form of Lemma
\ref{lem:ComLem}. To see this point, let us shift the indices $\{\mu,\cdots,r\}$
back to the standard form $\{1,\cdots,R\}$ ($R\triangleq r+1-\mu$).
Namely, we denote 
\[
\hat{{\bf n}}\triangleq(\hat{n}_{1},\cdots,\hat{n}_{R}),\ \hat{n}_{i}\triangleq\bar{n}_{\theta(\mu+i-1)};\ \hat{{\bf p}}\triangleq(1,\cdots,s),\hat{{\bf q}}\triangleq(s+1,\cdots,R).
\]
In addition, any shuffle $\rho\in{\cal P}(s;\mu,r)$ corresponds to
some $\hat{\rho}\in{\cal P}(s,R-s)$ in the obvious way (and also
at the level of words $I\leftrightarrow\hat{I}_{\rho}$). As a result,
one can rewrite 
\[
{\cal T}_{s,\tau;\theta}=\sum_{\hat{\rho}\in{\cal P}(s,R-s)}\frac{1}{b_{(\hat{i}_{1}),\bar{n}_{0}}^{\hat{{\bf n}}_{\hat{{\bf p}},\hat{{\bf q}}}^{1,R}}\cdots b_{\hat{I}_{\rho},\bar{n}_{0}}^{\hat{{\bf n}}{}_{\hat{{\bf p}},\hat{{\bf q}}}^{1,R}}}
\]
We define the sequence of numbers $\{c_{j}\}_{1\leqslant j\leqslant R+1}$
in the following way:

\[
\begin{array}{cccccccc}
b_{i-\sum_{j=1}^{s}\hat{n}_{j}}^{i-\sum_{j=1}^{s-1}\hat{n}_{j}-1} & \cdots & b_{i-\hat{n}_{1}-\hat{n}_{2}}^{i-\hat{n}_{1}-1} & b_{i-\hat{n}_{1}}^{i-1} & b_{i}^{i+\bar{n_{0}}-1} & b_{i+\bar{n}_{0}}^{i+\bar{n}_{0}+\hat{n}_{s+1}-1} & \cdots & b_{i+\bar{n}_{0}+\sum_{j=s+1}^{R-1}\hat{n}_{j}}^{i+\bar{n}_{0}+\sum_{j=s+1}^{R}\hat{n}_{j}-1}\\
c_{1} & \cdots & c_{s-1} & c_{s} & c_{s+1} & c_{s+2} & \cdots & c_{R+1}.
\end{array}
\]
Note that there is a one-to-one correspondence between ${\cal P}(s,R-s)$
and ${\cal T}(s,r-s)$ (cf. Lemma \ref{lem:ComLem} for its definition)
given by 
\begin{align*}
\hat{\rho}\mapsto\eta: & \ \big(\eta(1),\eta(2),\cdots,\eta(s),\eta(s+1),\eta(s+2),\cdots,\eta(R+1)\big)\\
 & \ \ \ \ \triangleq\big((\hat{\rho}(s)+1,\hat{\rho}(s-1)+1,\cdots,\hat{\rho}(1)+1,1,\hat{\rho}(s+1)+1,\cdots,\hat{\rho}(R)+1).
\end{align*}
It is then readily checked that 
\[
b_{(\hat{i}_{1}),\bar{n}_{0}}^{\hat{{\bf n}}_{\hat{{\bf p}},\hat{{\bf q}}}^{1,R}}\cdots b_{\hat{I},\bar{n}_{0}}^{\hat{{\bf n}}{}_{\hat{{\bf p}},\hat{{\bf q}}}^{1,R}}=c_{\eta^{-1}(2)}^{\eta^{-1}(2)}c_{\eta^{-1}(2)}^{\eta^{-1}(3)}\cdots c_{\eta^{-1}(2)}^{\eta^{-1}(R+1)},
\]
where the right hand side is defined by (\ref{eq:CSum}). According
to Lemma \ref{lem:ComLem}, one concludes that 
\begin{equation}
\sum_{\hat{\rho}\in{\cal P}(s,R-s)}\frac{1}{b_{(\hat{i}_{1}),\bar{n}_{0}}^{\hat{{\bf n}}_{\hat{{\bf p}},\hat{{\bf q}}}^{1,R}}\cdots b_{\hat{I},\bar{n}_{0}}^{\hat{{\bf n}}{}_{\hat{{\bf p}},\hat{{\bf q}}}^{1,R}}}=\big(\prod_{k=1}^{s}c_{k}^{s}\big)^{-1}\big(\prod_{k=s+2}^{R+1}c_{s+2}^{k}\big)^{-1}.\label{eq:KeyLemPf4}
\end{equation}
By the definition of $c_{j}$, it is apparent that both $c_{k}^{s}$
and $c_{s+2}^{k}$ are certain consecutive sums of the $b_{j}$'s.
In particular, (\ref{eq:KeyLemPf4}) is of the form $\frac{1}{b_{s_{1}}^{t_{1}}\cdots b_{s_{m}}^{t_{m}}}$
for suitable $s_{1}\leqslant t_{1},\cdots,s_{m}\leqslant t_{m}$.
Since the base point $(b_{1}^{*},\cdots,b_{N}^{*})$ is assumed to
be non-degenerate, it follows that the function ${\cal T}_{s,\tau;\theta}$
(as a function of $w$) is analytic near $w=0.$

The proof of the induction step is now complete.

\subsubsection*{Completing the proof of Lemma \ref{lem:AnaLem}}

We have obtained from (\ref{eq:MainSeries}) and Lemma \ref{lem:MainSeriesKeyLem}
that 
\begin{align}
 & \hat{F}_{w}\big(H_{K_{r}}\partial_{{\rm e}_{1}}\hat{\otimes}_{{\rm s}}\cdots\hat{\otimes}_{{\rm s}}H_{K_{1}}\partial_{{\rm e}_{1}}(H_{1}(n_{0}))\big)\nonumber \\
 & =\sum_{l}\frac{B_{l}}{l!}\sum_{i}\sum_{\mu=1}^{r+1}\sum_{\sigma\in{\cal S}_{r}}\Theta_{K_{\sigma(\mu-1)},\cdots,K_{\sigma(1)},n_{0}}^{i}\big(b_{i}^{i+(n^{\sigma})_{0}^{\mu-1}-1}\big)^{l+1-\mu}\nonumber \\
 & \ \ \ \ \ \ \sum_{I\in{\cal W}_{\mu,r}}\frac{(-1)^{\varepsilon(I)}\Psi_{I,(n^{\sigma})_{0}^{\mu-1}}^{({\bf K}^{\sigma})_{{\bf p},{\bf q}}^{\mu,r}}}{b_{(i_{\mu}),(n^{\sigma})_{0}^{\mu-1}}^{({\bf n}^{\sigma})_{{\bf p},{\bf q}}^{\mu,r}}\cdots b_{I,(n^{\sigma})_{0}^{\mu-1}}^{({\bf n}^{\sigma})_{{\bf p},{\bf q}}^{\mu,r}}}E_{i-|({\bf n}^{\sigma})_{{\bf p}}^{\mu,r}|,i+(n^{\sigma})_{0}^{\mu-1}+|({\bf n}^{\sigma})_{{\bf q}}^{\mu,r}|.}\label{eq:Pf5}
\end{align}
The argument for proving the analyticity of (\ref{eq:Pf5}) near $w=0$
is identical to the analysis in Step 3 of the proof of Lemma \ref{lem:MainSeriesKeyLem}.
The key idea, the same as before, is to write any permutation $\sigma\in{\cal S}_{r}$
as $\sigma=\tau\circ(\zeta\otimes\theta)$ where $\tau\in{\cal P}(\mu-1,r+1-\mu)$,
$\zeta\in{\cal S}_{\mu-1}$ and $\theta\in{\cal S}_{\{\mu,\cdots,r\}}$.
This allows one to split the $\sigma$-action in (\ref{lem:MainSeriesKeyLem})
into permutations over the $\{1,\cdots,\mu-1\}$ and $\{\mu,\cdots,r\}$
parts separately. By exactly the same analysis as in Step 3 of the
proof of Lemma \ref{lem:MainSeriesKeyLem}, one finds that 
\begin{align}
 & \hat{F}_{w}\big(H_{K_{r}}\partial_{{\rm e}_{1}}\hat{\otimes}_{{\rm s}}\cdots\hat{\otimes}_{{\rm s}}H_{K_{1}}\partial_{{\rm e}_{1}}(H_{1}(n_{0}))\big)\nonumber \\
 & =\sum_{i}\sum_{\mu=1}^{r+1}\sum_{\tau\in{\cal P}(\mu-1,r+1-\mu)}\big(b_{i}^{i+(n^{\tau})_{0}^{\mu}-1}\big)^{1-\mu}\phi\big(b_{i}^{i+(n^{\tau})_{0}^{\mu-1}}\big)\nonumber \\
 & \ \ \ \ \ \ \times\big(\sum_{\zeta\in{\cal S}_{\mu-1}}\Theta_{K_{\tau\circ\zeta(\mu-1)},\cdots,K_{\tau\circ\zeta(1),n_{0}}}^{i}\big)\cdot\Xi_{K_{\tau(\mu)},\cdots,K_{\tau(1)},(n^{\tau})_{0}^{\mu-1}}^{i}.\label{eq:MainAnaPf3}
\end{align}
Here 
\begin{align*}
\Xi_{K_{\tau(\mu)},\cdots,K_{\tau(r)},(n^{\tau})_{0}^{\mu-1}}^{i}\triangleq & \sum_{\theta\in{\cal S}_{\{\mu,\cdots,r\}}}\sum_{I\in{\cal W}_{\mu,r}:|{\bf p}|=s}\frac{(-1)^{\varepsilon(I)}\Psi_{I,(n^{\tau})_{0}^{\mu-1}}^{({\bf K}^{\tau\circ\theta})_{{\bf p},{\bf q}}^{\mu,r}}}{b_{(i_{\mu}),(n^{\tau})_{0}^{\mu-1}}^{({\bf n}^{\tau\circ\theta})_{{\bf p},{\bf q}}^{\mu,r}}\cdots b_{I,(n^{\tau})_{0}^{\mu-1}}^{({\bf n}^{\tau\circ\theta})_{{\bf p},{\bf q}}^{\mu,r}}}\\
 & \ \ \ \ \ \ \ \ \ E_{i-|({\bf n}^{\tau\circ\theta})_{{\bf p}}^{\mu,r}|,i+(n^{\tau})_{0}^{\mu-1}+|({\bf n}^{\sigma})_{{\bf q}}^{\mu,r}|}
\end{align*}
for $1\leqslant\mu\leqslant r$ and as a convention 
\[
\Xi_{K_{\tau(\mu)},\cdots,K_{\tau(r)},(n^{\tau})_{0}^{\mu-1}}^{i}\triangleq E_{i,i+n_{0}^{r}}
\]
if $\mu=r+1.$ The function $\Xi_{K_{\tau(\mu)},\cdots,K_{\tau(r)},(n^{\tau})_{0}^{\mu-1}}^{i}$
is an $\mathfrak{sl}_{N+1}(\mathbb{C})$-valued analytic function
near $w=0$ for the same reason leading to the analyticity of ${\cal T}_{s,\tau}$
as before (cf. (\ref{eq:Tstau})). The analyticity of (\ref{eq:MainAnaPf3})
near $w=0$ thus follows from the facts that 

\vspace{2mm}\noindent (i) $w=0$ is a simple pole for the function
\[
\phi\big(b_{i}^{i+(n^{\tau})_{0}^{\mu-1}}\big)=\phi(b_{i}^{*}+\cdots+b_{i+(n^{\tau})_{0}^{\mu-1}-1}^{*}+(n^{\tau})_{0}^{\mu-1}w);
\]
(ii) the function 
\[
\sum_{\zeta\in{\cal S}_{\mu-1}}\Theta_{K_{\tau\circ\zeta(\mu-1)},\cdots,K_{\tau\circ\zeta(1),n_{0}}}^{i}
\]
vanishes at $w=0.$ 

\vspace{2mm}\noindent This show that $\hat{F}_{w}\big(H_{K_{r}}\partial_{{\rm e}_{1}}\hat{\otimes}_{{\rm s}}\cdots\hat{\otimes}_{{\rm s}}H_{K_{1}}\partial_{{\rm e}_{1}}(H_{1}(n_{0}))\big)$
is an $\mathfrak{sl}_{N+1}(\mathbb{C})$-valued analytic function
near $w=0$. The proof of Lemma \ref{lem:AnaLem} is now complete.

\section*{Acknowledgement}

XG gratefully acknowledges the support from ARC grant DE210101352. SW gratefully acknowledges the support from Melbourne Research Scholarship. 

\bibliographystyle{plainurl}
\bibliography{bibliography}

\end{document}